\documentclass[11pt]{amsart}
\usepackage{amsmath}
\usepackage{amssymb}
\usepackage{amscd}
\def\NZQ{\mathbb}               % the font for N,Z,Q,R,C
\def\NN{{\NZQ N}}
\def\QQ{{\NZQ Q}}
\def\ZZ{{\NZQ Z}}
\def\RR{{\NZQ R}}

\newtheorem{Theorem}{Theorem}[section]
\newtheorem{Lemma}[Theorem]{Lemma}
\newtheorem{Corollary}[Theorem]{Corollary}
\newtheorem{Proposition}[Theorem]{Proposition}
\newtheorem{Remark}[Theorem]{Remark}

\newtheorem{Example}[Theorem]{Example}

\newtheorem{Definition}[Theorem]{Definition}

\let\epsilon\varepsilon
\let\phi=\varphi
\let\kappa=\varkappa

\begin{document}
\title{The  Minkowski equality of filtrations}
\author{Steven Dale Cutkosky}

\thanks{Partially supported by NSF grant DMS-1700046.}

\address{Steven Dale Cutkosky, Department of Mathematics,
University of Missouri, Columbia, MO 65211, USA}
\email{cutkoskys@missouri.edu}

\begin{abstract}
Suppose that $R$ is an analytically irreducible or excellent local domain with maximal ideal $m_R$. We 
 consider multiplicities and mixed multiplicities of $R$ by filtrations of $m_R$-primary ideals. We show that the theorem of Teissier, Rees and Sharp, and Katz, characterizing equality in the Minkowski inequality for multiplicities of ideals, is true for divisorial filtrations, and for the larger category of bounded filtrations. This theorem is not true for arbitrary filtrations of $m_R$-primary ideals. 
 \end{abstract}


\maketitle

\section{Introduction}\label{Intro}

The study of mixed multiplicities of $m_R$-primary ideals in a  local ring $R$ with maximal ideal $m_R$  was initiated by Bhattacharya \cite{Bh}, Rees  \cite{R} and Teissier  and Risler \cite{T1}. 
In \cite{CSS}  the notion of mixed multiplicities is extended to arbitrary,  not necessarily Noetherian, filtrations of $R$ by $m_R$-primary ideals ($m_R$-filtrations). It is shown in \cite{CSS} that many basic theorems for mixed multiplicities of $m_R$-primary ideals are true for $m_R$-filtrations. 

The development of the subject of mixed multiplicities and its connection to Teissier's work on equisingularity \cite{T1} can be found in \cite{GGV}.   A  survey of the theory of  mixed multiplicities of  ideals  can be found in  \cite[Chapter 17]{HS}, including discussion of the results of  the papers \cite{R1} of Rees and \cite{S} of  Swanson, and the theory of Minkowski inequalities of Teissier \cite{T1}, \cite{T2}, Rees and Sharp \cite{RS} and Katz \cite{Ka}.   Later, Katz and Verma \cite{KV}, generalized mixed multiplicities to ideals that are not all $m_R$-primary.  Trung and Verma \cite{TV} computed mixed multiplicities of monomial ideals from mixed volumes of suitable polytopes.  
%Mixed multiplicities are also used by Huh in the analysis of the coefficients of the chromatic polynomial of graph theory in \cite{H}.

%We will be concerned with  multiplicities and mixed multiplicities of  (not necessarily Noetherian) filtrations, which are defined as follows. 

%\begin{Definition}
A filtration $\mathcal I=\{I_n\}_{n\in\NN}$ of  a ring $R$ is a descending chain
$$
R=I_0\supset I_1\supset I_2\supset \cdots
$$
of ideals such that $I_iI_j\subset I_{i+j}$ for all $i,j\in \NN$.  An $m_R$-filtration $\mathcal I=\{I_n\}$ is a filtration $\mathcal I=\{I_n\}_{n\in\NN}$ of $R$ such that   $I_n$ is $m_R$-primary for $n\ge 1$.

A filtration $\mathcal I=\{I_n\}_{n\in\NN}$ of  a ring $R$ is said to be Noetherian if $\bigoplus_{n\ge 0}I_n$ is a finitely generated $R$-algebra.
 %\end{Definition}
 
 %We will refer to the classical case of a filtration $\mathcal I=\{I^n\}$ consisting of the powers of a fixed $m_R$-primary ideal $I$ as a filtration of powers of an $m_R$-primary ideal.

The following theorem is the key result needed to define the multiplicity of an $m_R$-filtration.  Let $\ell_R(M)$ denote the length of an $R$-module $M$.

\begin{Theorem} \label{TheoremI20} (\cite[Theorem 1.1]{C2} and  \cite[Theorem 4.2]{C3}) Suppose that $R$ is a  local ring of dimension $d$, and  $N(\hat R)$ is the nilradical of the $m_R$-adic completion $\hat R$ of $R$.  Then   the limit 
\begin{equation}\label{I5}
\lim_{n\rightarrow\infty}\frac{\ell_R(R/I_n)}{n^d}
\end{equation}
exists for any $m_R$-filtration  $\mathcal I=\{I_n\}$, if and only if $\dim N(\hat R)<d$.
\end{Theorem}

The problem of existence of such limits (\ref{I5}) has been considered by Ein, Lazarsfeld and Smith \cite{ELS} and Musta\c{t}\u{a} \cite{Mus}.
When the ring $R$ is a domain and is essentially of finite type over an algebraically closed field $k$ with $R/m_R=k$, Lazarsfeld and Musta\c{t}\u{a} \cite{LM} showed that
the limit exists for all $m_R$-filtrations.  Cutkosky \cite{C3} proved it in the complete generality  stated above in Theorem \ref{TheoremI20}. 
Lazarsfeld and Musta\c{t}\u{a}  use  in \cite{LM} the method of counting asymptotic vector space dimensions of graded families using ``Okounkov bodies''. This method, which is reminiscent of the geometric methods used by Minkowski in number theory, was developed by Okounkov \cite{Ok}, Kaveh and Khovanskii \cite{KK} and Lazarsfeld and   Musta\c{t}\u{a}   \cite{LM}. We also use this wonderful method. 
The fact that $\dim N(R)=d$ implies there exists a filtration without a limit was observed by Dao and Smirnov.

As can be seen from this theorem,  one must impose the condition that 
the dimension of the nilradical of the completion $\hat R$ of $R$ is less than the dimension of $R$ to ensure the existence of limits. The nilradical $N(R)$ of a $d$-dimensional ring $R$ is 
$$
N(R)=\{x\in R\mid x^n=0 \mbox{ for some positive integer $n$}\}.
$$
We have that $\dim N(R)=d$ if and only if there exists a minimal prime $P$ of $R$ such that $\dim R/P =d$ and $R_P$ is not reduced. In particular,  the condition $\dim N(\hat R)<d$ holds if $R$ is analytically unramified; that is, $\hat R$ is reduced.  Thus it holds if $R$ is excellent and reduced. 
We define the multiplicity of $R$ with respect to the $m_R$-filtration $\mathcal I=\{I_n\}$ to be 
$$
e_R(\mathcal I)=e_R(\mathcal I;R)=
\lim_{n\rightarrow \infty}\frac{\ell_R(R/I_n)}{n^d/d!}.
$$

The multiplicity of a ring with respect to a non Noetherian filtration can be an irrational number. 
A simple example on a regular local ring is given in \cite{CSS}.

Mixed multiplicities of filtrations are defined in \cite{CSS}. 
 Let $M$ be a finitely generated $R$-module where $R$ is a $d$-dimensional  local ring with $\dim N(\hat R)<d$. Let 
 $$
 \mathcal I(1)=\{I(1)_n\},\ldots, \mathcal I(r)=\{I(r)_n\}
 $$
  be $m_R$-filtrations. 
 In  \cite[Theorem 6.1]{CSS} and  \cite[Theorem 6.6]{CSS}, it is shown that the function
\begin{equation}\label{M2}
P(n_1,\ldots,n_r)=\lim_{m\rightarrow \infty}\frac{\ell_R(M/I(1)_{mn_1}\cdots I(r)_{mn_r}M)}{m^d}
\end{equation}
is  a homogeneous polynomial  of total degree $d$ with real coefficients for all  $n_1,\ldots,n_r\in\NN$.  
The mixed multiplicities of $M$ are defined from the coefficients of $P$, generalizing the definition of mixed multiplicities for $m_R$-primary ideals. Specifically,   
 we write 
\begin{equation}\label{eqV6}
P(n_1,\ldots,n_r)=\sum_{d_1+\cdots +d_r=d}\frac{1}{d_1!\cdots d_r!}e_R(\mathcal I(1)^{[d_1]},\ldots, \mathcal I(r)^{[d_r]};M)n_1^{d_1}\cdots n_r^{d_r}.
\end{equation}
We say that $e_R(\mathcal I(1)^{[d_1]},\ldots,\mathcal I(r)^{[d_r]};M)$ is the mixed multiplicity of $M$ of type $(d_1,\ldots,d_r)$ with respect to the $m_R$-filtrations $\mathcal I(1),\ldots,\mathcal I(r)$.
Here we are using the notation 
\begin{equation}\label{eqI6}
e_R(\mathcal I(1)^{[d_1]},\ldots, \mathcal I(r)^{[d_r]};M)
\end{equation}
  to be consistent with the classical notation for mixed multiplicities of $M$ with respect to $m_R$-primary ideals from \cite{T1}. The mixed multiplicity of $M$ of type $(d_1,\ldots,d_r)$ with respect to $m_R$-primary ideals $I_1,\ldots,I_r$, denoted by $e_R(I_1^{[d_1]},\ldots,I_r^{[d_r]};M)$ (\cite{T1}, \cite[Definition 17.4.3]{HS}) is equal to the mixed multiplicity $e_R(\mathcal I(1)^{[d_1]},\ldots,\mathcal I(r)^{[d_r]};M)$, where the Noetherian $I$-adic filtrations $\mathcal I(1),\ldots,\mathcal I(r)$ are defined by $\mathcal I(1)=\{I_1^i\}_{i\in \NN}, \ldots,\mathcal I(r)=\{I_r^i\}_{i\in \NN}$.

We have that 
\begin{equation}\label{eqX31}
e_R(\mathcal I(j);M)=e_R(\mathcal I(j)^{[d]};M)
\end{equation}
for all $j$.
 
%%%%%%%%%%%%%%

The multiplicities and mixed multiplicities of powers of $m_R$-primary ideals  are always positive (\cite{T1} or \cite[Corollary 17.4.7]{HS}). The multiplicities and mixed multiplicities of $m_R$-filtrations are always nonnegative, as is clear for multiplicities, and is established for mixed multiplicities in \cite[Proposition 1.3]{CSV}. However, they can be zero. If $R$ is analytically irreducible, then all mixed multiplicities are positive if and only if the multiplicities $e_R(\mathcal I(j);R)$ are positive for $1\le j\le r$. This is established in \cite[Theorem 1.4]{CSV}.

When the module $M$ is $R$ and $R$ is understood, we will usually write $e(\mathcal I)=e_R(\mathcal I)$ and $e(\mathcal I(1)^{[d_1]},\ldots,\mathcal I(r)^{[d_r]})=e_R(\mathcal I(1)^{[d_1]},\ldots,\mathcal I(r)^{[d_r]})$.
%%%%%%%%%%%%%%%%%

%%%%%%%%%%%%%%%
\subsection{Divisorial and bounded $m_R$-filtrations}
Suppose that $R$ is a $d$-dimensional  local domain, with quotient field $K$. A valuation $\nu$ of $K$ is called an $m_R$-valuation if $\nu$ dominates $R$ ($R\subset \mathcal O_{\nu}$ and $m_{\nu}\cap R=m_R$ where $\mathcal O_{\nu}$ is the valuation ring of $\nu$ with maximal ideal $m_{\nu}$) and ${\rm trdeg}_{R/m_R}\mathcal O_{\nu}/m_{\nu}=d-1$.

Associated to an $m_R$-valuation $\nu$ are valuation ideals
\begin{equation}\label{eqX2}
I(\nu)_n=\{f\in R\mid \nu(f)\ge n\}
\end{equation}
for $n\in \NN$.
In general, the $m_R$-filtration $\mathcal I(\nu)=\{I(\nu)_n\}$ is not Noetherian. 
In  a two-dimensional normal local ring $R$, the condition that the  filtration of valuation ideals of $R$ is Noetherian for all  $m_R$-valuations dominating $R$ is the condition (N) of Muhly and Sakuma \cite{MS}. It is proven in \cite{C8} that a complete normal local ring of dimension two satisfies condition (N) if and only if its divisor class group is a torsion group. 
An example is given in \cite{CGP} of an $m_R$-valuation $\nu$ of a 3-dimensional regular local ring $R$ such that the filtration $\mathcal I(\nu)$ is not Noetherian. The multiplicity $e(\mathcal I(\nu))$  is however a rational number.  
In Section \ref{SecExample}, we give an example  of an $m_R$-valuation $\nu$ dominating a normal  excellent local domain $R$ of dimension three such that $e_R(\mathcal I(\nu))=e_R(\mathcal I(\nu_{E_2}))$ is an irrational number. The filtration is necessarily non Noetherian.

\begin{Definition}\label{DefDF} Suppose that $R$ is a  local domain. We
say that an $m_R$-filtration $\mathcal I$ is a divisorial filtration if 
$\mathcal I=\mathcal I(a_1\mu_1+\cdots+a_s\mu_s)=\{I(a_1\mu_1+\cdots+a_s\mu_s)_m\}$
where $\mu_1,\ldots,\mu_s$ are $m_R$-valuations, $a_1,\ldots,a_s\in \NN$ are not all zero,  and
$$
I(a_1\mu_1+\cdots+a_s\mu_s)_m:=I(\mu_1)_{a_1m}\cap\cdots\cap I(\mu_s)_{a_sm}
$$
for $m\in \NN$. We sometimes write $D=a_1\mu_1+\cdots+a_s\mu_s$  and $\mathcal I(D)=\mathcal I(a_1\mu_1+\cdots+a_s\mu_s)=\{I(mD)\}$.

We can also define real divisorial $m_R$-filtrations by taking $a_1,\ldots,a_s\in \RR_{\ge 0}$ and defining an $m_R$-filtration
$\mathcal I(a_1\mu_1+\cdots+a_s\mu_s)=\{I(a_1\mu_1+\cdots+a_s\mu_s)_n\}$ by
$$
I(a_1\mu_1+\cdots+a_s\mu_s)_n=I(\mu_1)_{\lceil na_1\rceil }\cap\cdots\cap I(\mu_s)_{\lceil na_s\rceil}.
$$
\end{Definition}

Let $R$ be a local ring and $\mathcal I=\{I_m\}$ be an $m_R$-filtration. Then the integral closure $\overline{R[\mathcal I]}$ of $R[\mathcal I]=\sum_{m\ge 0}I_mu^m$ in $R[u]$ is
$\overline{R[\mathcal I]}=\sum_{n\ge 0}J_mu^m$ where $\{J_m\}$ is the $m_R$-filtration defined by $J_m=\{f\in R\mid f^r\in \overline{I_{rm}}$ for some $r>0\}$ (Lemma \ref{BdLemma1}). 

Let $R$ be a local domain. 
If $\mathcal I=\mathcal I(D)$ is a divisorial $m_R$-filtration, then $R[\mathcal I(D)]$ is integrally closed (Lemma \ref{BdLemma2}). 

\begin{Definition} Let $R$ be a local domain. 
An $m_R$-filtration $\mathcal I=\{I_n\}$ is said to be bounded if there exists an integral divisorial $m_R$-filtration $\mathcal I(D)$ such that $\overline{R[\mathcal I]}=R[\mathcal I(D)]$. 

A filtration $\mathcal I$ is said to be real bounded if there exists a real divisorial filtration $\mathcal I(D)$ such that $\overline{R[\mathcal I]}=R[\mathcal I(D)]$.
\end{Definition}

 If $\mathcal I=\{I^n\}$ is the classical $m_R$-filtration of powers of a fixed $m_R$-primary ideal $I$, then $\mathcal I$ is bounded (Lemma \ref{BdLemma3}).

Suppose that $R$ is an excellent or analytically unramified local domain and $I$ is an $m_R$-primary ideal in $R$. Let $X$ be the normalization of the blowup of $I$, with projective birational morphism $\phi:X\rightarrow \mbox{Spec}(R)$. Let $E_1,\ldots,E_t$ be the prime exceptional divisors of $X$. For $1\le i\le t$, let $\nu_{E_i}$ be the discrete valuation whose valuation ring is  $\mathcal O_{\nu_i}=\mathcal O_{X,E_i}$.
 If $R$ is normal, then $X$ is equal to the blowup of the integral closure $\overline{I^s}$ of an appropriate power $I^s$ of $I$. The $\nu_{E_i}$ are $m_R$-valuations. They are the Rees valuations of $I$.  Every $m_R$-valuation is a Rees valuation of some $m_R$-primary ideal.

\subsection{Rees's Theorem} \label{SubSecEqChar}
Rees has shown in \cite{R} that if $R$ is a formally equidimensional  local ring and $I\subset I'$ are $m_R$-primary ideals then the following are equivalent:
\begin{enumerate}
\item[1)] $e(I')=e(I)$
\item[2)] $\overline{\sum_{n\ge 0}(I')^nt^n}=\overline{\sum_{n\ge 0}I^nt^n}$.
\item[3)] $\overline{I'}=\overline{I}$
\end{enumerate}
 The statement 3) $\Rightarrow$ 1) is true for an arbitrary local ring.  
 
 This raises the question of whether  the conditions
 \begin{enumerate}
\item[1)] $e(\mathcal I')=e(\mathcal I)$
\item[2)] $\overline{\sum_{n\ge 0}I'_nt^n}=\overline{\sum_{n\ge 0}I_nt^n}$.
\end{enumerate} 
are equivalent for arbitrary $m_R$-filtrations $\mathcal I'\subset \mathcal I$.
 
 The statement 2) $\Rightarrow$ 1) is true  for arbitrary $m_R$-filtrations in a local ring which satisfies $\dim N(\hat R)<d$. This is shown in \cite[Theorem 6.9]{CSS} and Appendix \cite{C6}.
However the statement 1) $\Rightarrow$ 2)  is not true in general for $m_R$-filtrations (a simple example in a regular local ring is given in \cite{CSS}).

Rees's theorem is true for bounded $m_R$-filtrations.

\begin{Theorem}\label{NewRBT}(Theorem \ref{RBT}) Suppose that $R$ is an excellent local domain, $\mathcal I(1)$ is a real  bounded $m_R$-filtration and $\mathcal I(2)$ is an arbitrary $m_R$-filtration such that $\mathcal I(1)\subset \mathcal I(2)$. Then the following are equivalent
\begin{enumerate}
\item[1)]  $e(\mathcal I(1))=e(\mathcal I(2))$.
\item[2)]  There is equality of  integral closures  
$$
\overline{\sum_{m\ge 0}I(1)_mt^m}=\overline{\sum_{m\ge 0}I(2)_mt^m}
$$
  in $R[t]$.
\end{enumerate}
 \end{Theorem}

%%%%%%%%%%%%%%%%%%%

\subsection{The Minkowski inequalities and equality of  mixed multiplicities}\label{SubSecINEQ}

%%%%%%%%%%%%%%%%%%%%%%%%%%%%%
%The Minkowski inequalities hold for $m_R$-filtrations. 

The Minkowski inequalities were formulated and proven for   $m_R$-primary ideals in reduced equicharacteristic zero local rings by Teissier \cite{T1}, \cite{T2} and proven for  $m_R$-primary ideals in full generality, for  local rings,  by Rees and Sharp \cite{RS}. The same inequalities hold for filtrations.

\begin{Theorem}(Minkowski Inequalities for filtrations)(\cite[Theorem 6.3]{CSS})\label{TheoremMI}  Suppose that $R$ is a  $d$-dimensional  local ring with $\dim N(\hat R)<d$, $M$ is a finitely generated $R$-module and $\mathcal I(1)=\{I(1)_j\}$ and $\mathcal I(2)=\{I(2)_j\}$ are $m_R$-filtrations. Let $e_i=e_R(\mathcal I(1)^{[d-i]},\mathcal I(2)^{[i]};M)$ for $0\le i\le d$.
Then 
\begin{enumerate}
\item[1)] $e_i^2\le e_{i-1}e_{i+1}$ 
for $1\le i\le d-1$.
\item[2)]   
$e_ie_{d-i}\le e_0e_d$ for $0\le i\le d$.
\item[3)] $e_i^d\le e_0^{d-i}e_d^i$ for $0\le i\le d$. 
\item[4)]  $e_R(\mathcal I(1)\mathcal I(2));M)^{\frac{1}{d}}\le e_0^{\frac{1}{d}}+e_d^{\frac{1}{d}}$,  where $\mathcal I(1)\mathcal I(2)=\{I(1)_jI(2)_j\}$.
\end{enumerate}
\end{Theorem}

We write out the last inequality without   abbreviation as
\begin{equation}\label{eqMinkIn}
e_R(\mathcal I(1)\mathcal I(2));M)^{\frac{1}{d}}\le e_R(\mathcal I(1);M)^{\frac{1}{d}}+e_R(\mathcal I(2);M)^{\frac{1}{d}}
\end{equation}
where $\mathcal I(1)\mathcal I(2)=\{I(1)_mI(2)_m\}$.
This equation is called The Minkowski Inequality.

 The fourth inequality  4)  was proven for $m_R$-filtrations  in a regular local ring with algebraically closed residue field by Musta\c{t}\u{a} (\cite[Corollary 1.9]{Mus}) and more recently in this situation by Kaveh and Khovanskii (\cite[Corollary 7.14]{KK1}). The inequality 4) was proven with our assumption that $\dim N(\hat R)<d$ in \cite[Theorem 3.1]{C3}.
Inequalities 2) - 4) can be deduced directly from inequality 1), as explained in \cite{T1}, \cite{T2}, \cite{RS} and \cite[Corollary 17.7.3]{HS}.   

There is a beautiful characterization of when equality holds in the Minkowski inequality (\ref{eqMinkIn}) by 
Teissier \cite{T3} (for Cohen-Macaulay normal two-dimensional complex analytic $R$), Rees and Sharp \cite{RS} (in dimension 2) and Katz \cite{Ka} (in complete generality).

They have shown that 
 if $R$ is a formally equidimensional  local ring and $I(1), I(2)$  are $m_R$-primary ideals then the following are equivalent:
\begin{enumerate}
\item[1)]  The Minkowski inequality 
$$
e_R(I(1)I(2))^{\frac{1}{d}}=e(I(1))^{\frac{1}{d}}+e(I(2))^{\frac{1}{d}}
$$
holds.
\item[2)] 
There exist positive integers $a$ and $b$ such that 
$$
\overline{\sum_{n\ge 0}I(1)^{an}t^{n}}=\overline{\sum_{n\ge 0}I(2)^{bn}t^{n}}.
$$
\item[3)] There exist positive integers $a$ and $b$ such that 
$\overline{I(1)^a}=\overline{I(2)^b}$

\end{enumerate}

 The Teissier, Rees and Sharp, Katz theorem leads to the question of 
  whether the following conditions are equivalent for $m_R$-filtrations $\mathcal I(1)$ and $\mathcal I(2)$.

   \begin{enumerate}
\item[1)] The Minkowski equality 
$$
e_R(\mathcal I(1)\mathcal I(2))^{\frac{1}{d}}=e(\mathcal I(1))^{\frac{1}{d}}+e(\mathcal I(2)^{\frac{1}{d}}
$$
holds.
\item[2)] There exist positive integers $a$ and $b$ such that 
$$
\overline{\sum_{n\ge 0}I(1)^{an}t^{n}}=\overline{\sum_{n\ge 0}I(2)^{bn}t^{n}}.
$$
\end{enumerate}

 We show in Theorem \ref{MinNew} that if $\mathcal I(1)$ and $\mathcal I(2)$ are $m_R$-filtrations on a local ring $R$ such that  $\dim N(\hat R)<d$
 and condition 2) holds then the Minkowski equality 1) holds, but the converse statement, that  the Minkowski equality 1) implies condition 2)   is not true for filtrations, even in a regular local ring, as is shown in a simple example in \cite{CSS}.
 
 In Theorems \ref{TRSKT} and \ref{TRSKTA}, we show that 1) and 2) are equivalent for bounded $m_R$-filtrations on an analytically irreducible or excellent  local domain, giving a complete generalization of the Teissier, Rees and Sharp, Katz Theorem for bounded $m_R$-filtrations.

\begin{Theorem}\label{NewTRSKT}(Theorem \ref{TRSKT} and Theorem \ref{TRSKTA}) Suppose that $R$ is a $d$-dimensional analytically irreducible or excellent local domain and $\mathcal I(1)$ and $\mathcal I(2)$ are bounded $m_R$-filtrations. Then the following are equivalent
\begin{enumerate}
\item[1)]  The Minkowski equality
$$
e(\mathcal I(1)\mathcal I(2))^{\frac{1}{d}}=e(\mathcal I(1))^{\frac{1}{d}}+e(\mathcal I(2))^{\frac{1}{d}}
$$
holds.
\item[2)] There exist positive integers $a,b$ such that there is equality of  integral closures
$$
\overline{\sum_{n \ge 0}I(1)_{an}t^n}=\overline{\sum_{n\ge 0}I(2)_{bn}t^n}
$$
 in $R[t]$.
\end{enumerate}
 \end{Theorem}

\section{An overview of the proof}

In this section, we suppose that $R$ is a $d$-dimensional normal excellent local domain.

\subsection{Multiplicities of filtrations}  We summarize Sections \ref{SecFrame} and \ref{SecMF} in this subsection.
We use the method of  counting asymptotic vector space dimensions of graded families by computing volumes of convex bodies associated to appropriate semigroups
 introduced in \cite{Ok}, \cite{LM} and \cite{KK}.
Let $\nu$ be a valuation of the quotient field $K$ of $R$ which dominates $R$ and has value group isomorphic to $\ZZ^d$. Further suppose that $\nu(f)\in \NN^d$ if $0\ne f\in R$. Then we can associate to an $m_R$-filtration $\mathcal I=\{I_n\}$ a semigroup $\Gamma(\mathcal I)\subset \NN^{d+1}$ defined by
$\Gamma(\mathcal I)=\{(\nu(f),n)\mid f\in I_n\}$. Let $\Delta(\mathcal I)$ be the intersection of the closure of the  real cone generated by $\Gamma(\mathcal I)$ with $\RR^d\times\{1\}$. 
 Similarly, we define $\Delta(R)$ to be the subset of $\RR^d$ constructed from $\Gamma(R)$ by replacing $I_n$ with $R$ for all $n$. 
 
 For $c\in \RR_{>0}$, let 
 $$
 H_c^-=\{(x_1,\ldots,x_d)\in \RR^d\mid x_1+\cdots+x_d\le c\}.
 $$
 
 Using some commutative algebra, we  find a constant $c>0$ such that 
 \begin{equation}\label{eqIn4}
\Delta(\mathcal I)\setminus (\Delta(\mathcal I)\cap H_c^-)
=\Delta(R)\setminus   (\Delta(R)\cap H_c^-).
\end{equation}

Then $\Delta(\mathcal I)\cap H_c^-$ and $\Delta(R)\cap H_c^{-}$ are compact convex sets and by (\ref{eqMF5}),
\begin{equation}\label{eqIn5*}
\frac{e_R(\mathcal I)}{d!}=\delta[{\rm Vol}(\Delta(R)\cap H_c^-)-{\rm Vol}(\Delta(\mathcal I)\cap H_c^-)]
\end{equation}
where $\delta=[\mathcal O_{\nu}/m_{\nu}:R/m_R]$.

\subsection{ The Integral closure of a filtration $\mathcal I$ and the convex sets $\Delta(\mathcal I)$}
Suppose that $\mathcal I'\subset \mathcal I$ are $m_R$-filtrations. Then we have $\Delta(\mathcal I')\subset \Delta(\mathcal I)$, so we have $e_R(\mathcal I)= e_R(\mathcal I')$ if and only if $\Delta(\mathcal I')=\Delta(\mathcal I)$. 

If $\mathcal I'$ is a Noetherian $m_R$-filtration, and  $\mathcal I$ is an $m_R$-filtration such that $\mathcal I'\subset \mathcal I$, then we have that 
$e_{R}(\mathcal I')=e_R(\mathcal I)$ if and only if $\Delta(\mathcal I')=\Delta(\mathcal I)$ which holds if and only if $R[\mathcal I]=\sum_{m\ge 0}I_mu^m\subset  \overline{\sum_{m\ge 0}I'_mu^m}=\overline{R[\mathcal I']}$. This can be proven as follows.  By taking suitable Veronese subalgebras, we  reduce to the case where $\mathcal I$ and $\mathcal I'$ are the filtrations of powers of fixed $m_R$-primary ideals $I$ and $I'$, so that the result then follows from  Rees's Theorem \cite{R} for normal excellent local domains. Rees's theorem was discussed at the beginning of Subsection \ref{SubSecEqChar}. 

For arbitrary $m_R$-filtrations $\mathcal I'\subset \mathcal I$ such that $R[\mathcal I]=
\sum I_mt^m\subset 
\overline{\sum_{m\ge 0}I'_mt^m}=\overline{R[\mathcal I']}$ we have that $e_R(\mathcal I')=e_R(\mathcal I)$, as shown in \cite[Theorem 6.9]{CSS} and \cite[Appendix]{C6}. However, as we mentioned  in the beginning of Subsection \ref{SubSecEqChar}, there exists a non-Noetherian $m_R$-filtration $\mathcal I'$ and a Noetherian $m_R$-filtration $\mathcal I$ such that $\mathcal I'\subset \mathcal I$, $e_R(\mathcal I')=e_R(\mathcal I)$ and $R[\mathcal I]=\sum_{m\ge 0}I_mt^m$ is not a subset of $\overline{R[\mathcal I']}=\overline{\sum_{m\ge 0}I'_mt^m}$. 

\subsection{The invariant $\gamma_{\mu}(\mathcal I)$}\label{SubSecInv} This subsection is a summary of Subsection \ref{SubSecGam}.
Let $\mu$ be an $m_R$-valuation and $\mathcal I$ be an $m_R$-filtration. Define $\tau_m=\min\{\mu(f)\mid f\in I_m\}$ and $\gamma_{\mu}(\mathcal I)=\inf_m\{\tau_m\}$. The numbers $\tau_m\in \ZZ_{>0}$ for all $m$ but $\gamma_{\mu}(\mathcal I)$ can be an irrational number, even when $\mathcal I$ is a divisorial $m_R$-filtration, as shown in Section \ref{SecExample}) and explained in Subsection \ref{SubExEx}.

Theorem \ref{Theorem2} shows that if $\mathcal I'\subset \mathcal I$ and $e_R(\mathcal I')=e_R(\mathcal I)$ then $\gamma_{\mu}(\mathcal I')=\gamma_{\mu}(\mathcal I)$ for all
$m_R$-valuations $\mu$. This is proven by taking the valuation $\nu$ used to compute $\Delta$ to be composite with $\mu$, so $\nu(f)=(\mu(f),\cdots)\in \NN^d$ for $f\in R$. The condition $e_R(\mathcal I')=e_R(\mathcal I)$ implies $\Delta(\mathcal I')=\Delta(\mathcal I)$ and $\gamma_{\mu}(\mathcal I'),\gamma_{\mu}(\mathcal I)$ are the smallest points of the projections of $\Delta(\mathcal I')$, respectively $\Delta(\mathcal I)$ onto the first coordinate of $\RR^d$.

\subsection{Divisorial Filtrations} In this subsection, we summarize material from Section \ref{SecVal}.
Let $\phi:X\rightarrow \mbox{Spec}(R)$ be a birational projective morphism such that $X$ is normal and is the blow up of an $m_R$-primary ideal. Let $E_1,\ldots,E_r$ be the prime exceptional divisors of $\phi$, and for $1\le i \le r$, let $\nu_{E_i}$ be the $m_R$-valuation whose valuation ring is $\mathcal O_{X,E_i}$.  Suppose that $D=\sum a_iE_i$ with $a_i\in \NN$ is an effective   Weil divisor on $X$ with exceptional support. 

  Define $\gamma_{E_i}(D)=\gamma_{\nu_{E_i}}(\mathcal I(D))$ for $1\le i \le r$. 
  Then $\gamma_{E_i}(D)\ge a_i$ for all $i$. 
We have that $ma_i$ is the prescribed order of vanishing of elements of $I(mD)$ along $E_i$   but $m\gamma_{E_i}(D)$ is asymptotically the actual vanishing.

We remark that   $\gamma_{\mu}(\mathcal I(D))$ can be an irrational number. By Theorem \ref{Theorem5}, the example $X$ of Section \ref{SecExample} has two 
 prime exceptional divisors $E_1$ and $E_2$ such that 
$$
\gamma_{E_1}(E_2)=\frac{3}{9-\sqrt{3}}
$$
is an irrational number. This example is surveyed in Subsection \ref{SubExEx}.

We have that 
\begin{equation}\label{eqIn1}
I(mD)=I(\lceil \sum m\gamma_{E_i}(D)E_i\rceil)
\end{equation}
for all $m\in \NN$, where $\lceil x\rceil$ is the round up of a real number $x$. In this way, we are led to extend our category of divisorial $m_R$-filtrations to real divisorial $m_R$-filtrations.

Now let $\mathcal I=\mathcal I(\sum_{i=1}^s a_i\mu_i)$ with $a_i\in \NN$ be a divisorial $m_R$-filtration. 
 A representation of $\mathcal I$ is a pair $\phi:X\rightarrow \mbox{Spec}(R)$ and a divisor
 $\sum_{i=1}^s a_iE_i$ such that $X$ is as in the above paragraph, and $\nu_{E_i}=\mu_i$ for $1\le i\le s\le r$. We remark that it is not always possible to construct an $X$ for which $r=s$, even in dimension $d=2$. An example of a two dimensional excellent normal local domain without a ``one fibered ideal''  is given in \cite{C8}. A one fibered ideal is an $m_R$-primary ideal $I$ such that the normalization of its blowup has only one prime exceptional divisor.
 
 \subsection{Rees's theorem for divisorial $m_R$-filtrations} It follows from Corollary \ref{CorDiv2} that if $\mathcal I(D_1)\subset \mathcal I(D_2)$ are divisorial $m_R$-filtrations such that $e(\mathcal I(D_2))=e(\mathcal I(D_1))$, then $\mathcal I(D_2)=\mathcal I(D_1)$.
 This is proven in Section \ref{SecMF}.
  Let $X\rightarrow \mbox{Spec}(R)$ be a representation of $D_1$ and $D_2$, and write
 $D_1=\sum a_iE_i$ and $D_2=\sum b_iE_i$ as Weil divisors on $X$. 
 
 By Theorem \ref{Theorem2}, whose proof was discussed in Subsection \ref{SubSecInv},
 $\gamma_{E_i}(D_1)=\gamma_{E_i}(D_2)$ for $1\le i\le r$. 
 Thus $I(mD_1)=I(mD_2)$ for all $m\in \NN$ by (\ref{eqIn1}).
 
 \subsection{The Teissier, Rees and Sharp, Katz Theorem for divisorial $m_R$-filtrations} 

 Suppose that we have equality in the Minkowski inequality (\ref{eqMinkIn}) for the divisorial $m_R$-filtrations $\mathcal I(D_1)$ and $\mathcal I(D_2)$.  We will give an outline of our proof  that there exist $a,b\in \ZZ_{>0}$ such that $I(maD_1)=I(bmD_2)$ for all $m\in \NN$. Let
  $$
 f(n_1,n_2):=\lim_{m\rightarrow \infty} \frac{\ell_R(R/I(mn_1D_1)I(mn_2D_2))}{m^d}.
 $$
 Using the Minkowski inequalities $e_i^d\le e_0^{d-i}e_d^i$ of 3) of Theorem \ref{TheoremMI}, we obtain in (\ref{eq47}) of Section \ref{SecMinEQ} that 
 $$
 f(n_1,n_2)=\frac{1}{d!}(e_0^{\frac{1}{d}}n_1+e_d^{\frac{1}{d}}n_2)^d
 $$
 where $e_0=e_R(\mathcal I(D_1))$ and $e_d=e_R(\mathcal I(D_2))$.
 
 We now survey Section \ref{SecMMF}.
 Define semigroups $\Gamma(n_1,n_2)=\Gamma(\{I(mn_1D_1)I(mn_2D_2)\})$ and 
 associated closed convex sets $\Delta(n_1,n_2)$. We can find $\phi\in \RR_{>0}$ such that 
 letting
 $$
 H_{\Phi,n_1,n_2}^-=\{(x_1,\ldots,x_d)\in \RR^d\mid x_1+\cdots+x_d\le\phi e_0^{\frac{1}{d}}n_1+\phi e_d^{\frac{1}{d}}n_2\},
 $$
 $$
 \Delta_{\Phi}(n_1,n_2)=\Delta(n_1,n_2)\cap H_{\Phi,n_1,n_2}^-
 $$
 and
 $$
 \tilde \Delta_{\Phi}(n_1,n_2)=\Delta(R)\cap H_{\Phi,n_1,n_2}^-,
 $$
  we have (\ref{eq11}) that
 $$
 f(n_1,n_2)=\delta[{\rm Vol}(\tilde\Delta_{\Phi}(n_1,n_2)-{\rm Vol}(\Delta_{\Phi}(n_2,n_2))]
 $$
 as in (\ref{eqIn5*}).
 Since $\Delta(R)$ is a closed cone with vertex at the origin, by (\ref{eq31}) and (\ref{eq41})
 $$
 {\rm Vol}(\tilde \Delta_{\Phi}(n_1,n_2))=(n_1\alpha_1+n_2\alpha_2)^d\phi^d{\rm Vol}(\Delta(R)\cap H_1^-).
 $$
 
 We now survey Section \ref{SecGeom}.
 We define in (\ref{eq3})
 $$
 h(n_1,n_2)={\rm Vol}(\Delta_{\Phi}(n_1,n_2))={\rm Vol}(\tilde \Delta_{\Phi}(n_1,n_2))
 -\frac{f(n_1,n_2)}{\delta} =\lambda(\alpha_1n_1+\alpha_2n_2)^d.
 $$
 for some $\lambda\in \RR_{>0}$.
 
 Let 
 $$
 g(n_1,n_2):= {\rm Vol}(n_1\Delta_{\Phi}(1,0)+n_2\Delta_{\Phi}(0,1)),
 $$
 which is a homogeneous real polynomial of degree $d$ (Theorem \ref{ConvMix})
 Since 
 $$
 n_1\Delta_{\Phi}(1,0)+n_2\Delta_{\Phi}(0,1)\subset  \Delta_{\Phi}(n_1,n_2), 
 $$
 we have that $g(n_1,n_2)\le h(n_1,n_2)$ for all $n_1,n_2\in\NN$, $g(1,0)=h(1,0)$ and $g(0,1)=h(0,1)$. Thus for $0<t<1$,
 $$
 \begin{array}{lll}
 h(1-t,t)^{\frac{1}{d}}&=&(1-t)h(1,0)^{\frac{1}{d}}+th(0,1)^{\frac{1}{d}}
 =(1-t)g(1,0)^{\frac{1}{d}}+tg(0,1)^{\frac{1}{d}}\\
 &\le& g(1-t,t)^{\frac{1}{d}}\le h(1-t,t)^{\frac{1}{d}}.
 \end{array}
 $$
 where the first inequality on the second line is the Brunn-Minkowski inequality of convex geometry (Theorem \ref{BMTheorem}). We see from this equation that we  have equality in the 
  Brunn-Minkowski inequality. Thus by Theorem \ref{BMTheorem}, we have that $\Delta_{\Phi}(1,0)$ and $\Delta_{\Phi}(0,1)$ are homothetic; that is, there is an affine transformation $T(\vec x)=
  c\vec x+\gamma$ such that $T(\Delta_{\Phi}(1,0))=\Delta_{\Phi}(1,0)$. 
  We  then show  in Theorem \ref{Theorem3} that 
  $$
  e_d^{\frac{1}{d}}\Delta_{\Phi}(1,0)=e_0^{\frac{1}{d}}\Delta_{\Phi}(0,1),
  $$ 
 and applying Theorem \ref{Theorem1}, which is proved like Theorem \ref{Theorem2} discussed in Subsection \ref{SubSecInv}, we get that 
 \begin{equation}\label{eqIn2}
 \frac{\gamma_{E_j}(D_1)}{e_0^{\frac{1}{d}}}=\frac{\gamma_{E_j}(D_2)}{e_d^{\frac{1}{d}}} \end{equation}
 for $1\le j\le r$. 
 
 It is shown in Theorem \ref{Theorem8} that (assuming the Minkowski equality holds) the  real number $\frac{e_d^{\frac{1}{d}}}{e_0^{\frac{1}{d}}}$ is actually  a rational number $\frac{a}{b}$.   This is in spite of the fact that the multiplicities $e_0$ and $e_d$ can be irrational numbers and the $\gamma_{E_j}(D_i)$ can be irrational numbers (as shown in the example of Section \ref{SecExample}, which is surveyed in subsection \ref{SubExEx}).
 
 Now combining  this fact, (\ref{eqIn1}) and (\ref{eqIn2}) we obtain in Theorem \ref{Theorem8} that
 $$
 I(maD_1)=\Gamma(X,\mathcal O_X(-\lceil \sum_{i=1}^rma\gamma_{E_i}(D_1)E_i\rceil)
 =\Gamma(X,\mathcal O_X(-\lceil \sum_{i=1}^rmb\gamma_{E_i}(D_2)E_i\rceil)
 =I(mbD_2)
 $$
 for all $m\in \NN$.
 
 The proof of Theorem \ref{Theorem8} uses the invariant
 $$
 w_{\mathcal I}(f)=\max\{m\mid f\in I_m\}
 $$
 for a filtration $\mathcal I=\{I_m\}$ and $f\in R$,
 which is either a natural number or $\infty$,
 and the fact that  an integral divisorial $m_R$-filtration $\mathcal I(D)$ has the good property that  for $f\in R$, there exists $d\in \ZZ_{>0}$ such that $w_{\mathcal I(D)}(f^{nd})=nw_{\mathcal I(D)}(f^d)$ for all $n\in \NN$ (Lemma \ref{ratfun}).
 
 It is natural to define 
  $$
 \overline w_{\mathcal I}(f)=\limsup_{n\rightarrow\infty}\frac{w_{\mathcal I}(f^n)}{n}
 $$
 which generalizes to filtrations the asymptotic Samuel function $\overline \nu_I(f)$ of an ideal in $R$ (\cite[Definition 6.9.3]{HS}). We use a theorem of Rees in \cite{R3} about the asymptotic Samuel function (reduced order) $\overline\nu_{m_R}$ in our proof of Lemma \ref{LemmaA}.
 %\vskip 5truein
 %Define 
 %$$
 %H_{\Phi,n_1,n_2}=\{(x_1,\ldots,x_d)\in \RR^d\mid x_1+\cdots+x_d=\phi e_0^{\frac{1}{d}}n_1+\phi e_d^{\frac{1}{d}}n_2\}.
 %$$

 %%%%%%%%%%%%%%%
 
  %%%%%%%%%%%%%%%%%%%%
  \subsection{An Example}\label{SubExEx}  The above concepts and results are analyzed in an example from \cite{C7} in Section \ref{SecExample}.  The example is of the blowup $\phi:X\rightarrow \mbox{Spec}(R)$ of an $m_R$ primary ideal in a normal and excellent three dimensional local ring $R$ which is a resolution of singularities. The map $\phi$ has two prime exceptional divisors $E_1$ and $E_2$. The function
  $$
  f(n_1,n_2)=\lim_{m\rightarrow \infty}\frac{\ell_R(R/I(mn_1E_1+mn_2E_2))}{m^3}
  $$
  is computed in \cite{C7} and is  reproduced here. 
  
  \begin{Theorem}(\cite[Theorem 1.4]{C7}) For $n_1,n_2\in \NN$,
    $$
   f(n_1,n_2)
  =\left\{\begin{array}{ll}
  33 n_1^3&\mbox{ if }n_2<n_1\\
  78n_1^3-81n_1^2n_2+27n_1n_2^2+9n_2^3&\mbox{ if }n_1\le n_2<n_1\left(3-\frac{\sqrt{3}}{3}\right)\\
  \left(\frac{2007}{169}-\frac{9\sqrt{3}}{338}\right)n_2^3&\mbox{ if }n_1\left(3-\frac{\sqrt{3}}{3}\right)<n_2.
  \end{array}\right.
  $$
    \end{Theorem}
    Thus $f(n_1,n_2)$  is not a polynomial, 
  but it is ``piecewise a polynomial''; that is, $\RR^2_{\ge 0}$ consists of three triangular regions determined by lines through the origin such that $f(n_1,n_2)$ is a polynomial function within each of these three regions. The line separating the second and third regions has irrational slope, and the function $f(n_1,n_2)$ has an irrational coefficient in the third region.  The  middle region is the ample cone and is also the Nef cone.

 We compute the functions $\gamma_{E_1}$ and $\gamma_{E_2}$ in \cite[Theorem 4.1]{C7}, as summarized in the following theorem. Observe that $\gamma_{E_1}$ is an  irrational number in the third region.
    
  \begin{Theorem}(\cite[Theorem 4.1]{C7}) Let $D=n_1E_1+n_2E_2$ with $n_1,n_2\in \NN$, an effective exceptional  divisor on $X$.
  \begin{enumerate}
  \item[1)] Suppose that $n_2<n_1$. Then  $\gamma_{E_1}(D)=n_1$ and $\gamma_{E_2}(D)=n_1$.
  \item[2)]   Suppose that $n_1\le n_2<n_1 \left(3-\frac{\sqrt{3}}{3}\right)$. Then   $\gamma_{E_1}(D)=n_1$ and $\gamma_{E_2}(D)=n_2$.
  \item[3)] Suppose that $n_1 \left(3-\frac{\sqrt{3}}{3}\right)<n_2$. Then   $\gamma_{E_1}(D)=\frac{3}{9-\sqrt{3}}n_2$ and $\gamma_{E_2}(D)=n_2$.
  \end{enumerate}
  In all three cases, $-\gamma_{E_1}(D)E_1-\gamma_{E_2}(D)E_2$ is nef on $X$. 
    \end{Theorem}  
    
    We determine the divisors for which Minkowski's inequality holds in the following Corollary, reproduced from Section \ref{SecExample}.
    
 \begin{Corollary}(Corollary \ref{Cor20T}) Suppose that $D_1$ and $D_2$ are
 effective integral exceptional divisors
 on $X$.   
 If $D_1$ and $D_2$ are in the first region of Theorem \ref{Theorem4}, then Minkowski's equality holds between them. If $D_1$ and $D_2$ are  in the second region, then Minkowski's equality holds between them if and only if $D_2$ is a rational multiple of $D_1$. If $D_1$ and $D_2$ are in the third region, then Minkowski's equality holds between them. 
  Minkowski's equality cannot hold between $D_1$ and $D_2$ in different regions.  
  \end{Corollary}
  
  The above theorem allows us to compute the mixed multiplicities of any two 
  divisors $D_1=a_1E_1+a_2E_2$ and $D_2=b_1E_1+b_2E_2$ by interpreting mixed multiplicities as the anti positive intersection multiplicities 
   of (\ref{eq21T}). 
      
 In particular, we can compare $f(n_1,n_2)$ with the polynomial 
 $$
  P(n_1,n_2)=\lim_{m\rightarrow \infty}\frac{\ell_R(R/I(mn_1E_1)I(mn_2E_2))}{m^d}.
  $$ 
  We calculate in (\ref{eqT20}) that 
   $$
   \begin{array}{lll}
  P(n_1,n_2)&=&
  \frac{1}{3!}e(\mathcal I(E_1)^{[3]})n_1^3+\frac{1}{2!}e(\mathcal I(E_1)^{[2]},\mathcal I(E_2)^{[1]})n_1^2n_2\\
  &&+\frac{1}{2!}e(\mathcal I(E_1)^{[1]},\mathcal I(E_2)^{[2]})n_1n_2^2
  +\frac{1}{3!}e(\mathcal I(E_2)^{[3]})n_2^3\\
  &=&33n_1^3+(\frac{891}{26}+\frac{99}{26}\sqrt{3})n_1^2n_2+(\frac{12042}{338}-\frac{27}{338}\sqrt{3})n_1n_2^2+\left(\frac{2007}{169}-\frac{9\sqrt{3}}{338}\right)n_2^3.
  \end{array}
  $$

   \section{Notation}

We will denote the nonnegative integers by $\NN$ and the positive integers by $\ZZ_{>0}$,   the set of nonnegative rational numbers  by $\QQ_{\ge 0}$  and the positive rational numbers by $\QQ_{>0}$. 
We will denote the set of nonnegative real numbers by $\RR_{\ge0}$ and the positive real numbers by $\RR_{>0}$. For a real number $x$, $\lceil x\rceil$ will denote the smallest integer that is $\ge x$ and $\lfloor x\rfloor$ will denote the largest integer that is $\le x$. If $E_1,\ldots, E_r$ are prime divisors on a normal scheme $X$ and $a_1,\ldots,a_r\in \RR$, then $\lfloor \sum a_iE_i\rfloor$ denotes the integral divisor $\sum \lfloor a_i\rfloor E_i$ and $\lceil \sum a_iE_i\rceil$ denotes the integral divisor $\sum \lceil a_i\rceil E_i$.

A local ring is assumed to be Noetherian.
The maximal ideal of a local ring $R$ will be denoted by $m_R$. The quotient field of a domain $R$ will be denoted by ${\rm QF}(R)$. We will denote the length of an $R$-module $M$ by $\ell_R(M)$. Excellent local rings have many excellent properties which are enumerated in \cite[Scholie IV.7.8.3]{EGA}. We will make use of some of these properties without further reference.

\section{Preliminaries}

\subsection{Approximation of irrational numbers}

The following formula for approximation of real numbers appears in \cite{HW} (Remark on bottom of page 156).

\begin{Lemma}\label{LemmaHW} Suppose that $\xi,\alpha\in \RR_{>0}$. Then
\begin{enumerate}
\item[a)] There exist $p_0,q_0\in \ZZ_{>0}$ such that
$$
0\le \xi-\frac{p_0}{q_0}<\frac{\alpha}{q_0}.
$$
\item[b)] There exist $p_0',q_0'\in \ZZ_{>0}$ such that
$$
-\frac{\alpha}{q_0'}<\xi-\frac{p_0'}{q_0'}\le 0
$$
\end{enumerate}
\end{Lemma}

\begin{proof} If $\xi$ is a rational number we need only write $\xi=\frac{p_0}{q_0}$ with $p_0,q_0\in \ZZ_{>0}$ (or $\xi=\frac{p_0'}{q_0'}$ with $p_0',q_0'\in \ZZ_{>0}$).

Suppose that $\xi$ is an irrational number. By \cite[Theorem 170]{HW}, we can express $\xi$ as an infinite simple continued fraction. Let $\frac{p_n}{q_n}$ be the convergents of this continued fraction for $n\in \ZZ_{>0}$. By \cite[Theorem 156]{HW}, $q_n\ge n$, and by \cite[Theorem 164]{HW} and \cite[Theorem 171]{HW}, we have that
$$
\xi-\frac{p_n}{q_n}=\frac{(-1)^n\delta_n}{q_nq_{n+1}}
$$
with $0<\delta_n<1$ for all $n$ from which the lemma follows.
\end{proof}

\subsection{The Brunn-Minkowski inequality in Convex Geometry}

Let $K$ and $L$ be compact convex subsets of $\RR^d$. For $\lambda\in \RR_{\ge 0}$, define
$$
\lambda K=\{\lambda x\mid x\in K\}
$$
and for $\lambda_1,\lambda_2\in \RR_{\ge 0}$, define the Minkowski sum
$$
\lambda_1K+\lambda_2L=\{\lambda_1x+\lambda_2y\mid x\in K, y\in L\}.
$$

A proof of the following theorem can be found in \cite[Section 29, page 42]{BF}.

\begin{Theorem}\label{ConvMix} Suppose that $K_1,\ldots,K_r$ are compact convex subsets of $\RR^d$. Then the volume function
${\rm Vol}(\lambda_1K_1+\cdots+\lambda_rK_r)$ is a homogeneous real polynomial of degree $d$ for $\lambda_1,\ldots,\lambda_r\in \RR_{\ge 0}$. 
\end{Theorem}

The coefficients of the polynomial of the theorem are called mixed volumes.

 We now state the Brunn-Minkowski Theorem of convex geometry. A couple of proofs of this   theorem are  on \cite[Page 94]{BF} and in \cite{Kla}. 
 
 \begin{Theorem}\label{BMTheorem} Let $K$ and $L$ be compact convex subsets of $\RR^d$.
 Then
 \begin{equation}\label{eqBM}
 {\rm Vol}\left((1-t)K+tL\right)^{\frac{1}{d}}\ge (1-t){\rm Vol}(K)^{\frac{1}{d}}+t{\rm Vol}(L)^{\frac{1}{d}}
 \end{equation}
 for $0\le t\le 1$. Further, if ${\rm Vol}(K)$ and ${\rm Vol}(L)$ are positive, then equality holds in (\ref{eqBM}) for some $t$ with $0<t<1$ if and only if $K$ and $L$ are homothetic; that is, there exists $0<c\in \RR$ and $\vec\gamma\in \RR^d$ such that $L=cK+\vec\gamma$.
 
 If $K$ and $L$ are homothetic, then equality holds in (\ref{eqBM}) for all $t$ with $0\le t\le 1$.
 \end{Theorem}

 \section{$m_R$-valuations and divisorial $m_R$-filtrations on local domains $R$}\label{SecVal}

\subsection{$m_R$-valuations and $m_R$-filtrations}\label{SubSecGam}

In this subsection,  suppose that $R$ is a $d$-dimensional  local domain, with quotient field $K$. A valuation $\nu$ of $K$ is called an $m_R$-valuation if $\nu$ dominates $R$ ($R\subset V_{\nu}$ and $m_{\nu}\cap R=m_R$ where $V_{\nu}$ is the valuation ring of $\nu$ with maximal ideal $m_{\nu}$) and ${\rm trdeg}_{R/m_R}V_{\nu}/m_{\nu}=d-1$.

Let $\mathcal I=\{I_i\}$ be  an $m_R$-filtration. Let $\mu$ be an $m_R$-valuation.
Let
$$
I(\mu)_m=\{f\in R\mid \mu(f)\ge m\},
$$
 and
 define
$$
\tau_{\mu,m}(\mathcal I)=\mu(I_m)=\min\{\mu(f)\mid f\in I_m\}.
$$
Since $\tau_{\mu,mn}(\mathcal I)\le n\tau_{\mu,m}(\mathcal I)$, we have that
\begin{equation}\label{eqAR1}
\frac{\tau_{\mu,mn}(\mathcal I)}{mn}\le\min\{\frac{\tau_{\mu,m}(\mathcal I)}{m},\frac{\tau_{\mu,n}(\mathcal I)}{n}\}
\end{equation}
for $m,n\in \NN$.

Define
\begin{equation}\label{eqAR3}
\gamma_{\mu}(\mathcal I)=\inf_m\frac{\tau_{\mu,m}(\mathcal I)}{m}.
\end{equation}

\subsection{Divisors on blowups of normal local domains} In this subsection suppose that $R$ is a normal excellent local domain. Let $\phi:X\rightarrow \mbox{Spec}(R)$ be a birational projective morphism such that $X$ is normal and $X$ is the blowup of an $m_R$-primary ideal. 
Let $E_1,\ldots,E_r$ be the prime divisors on $X$ with exceptional support. A real divisor $D$ on $X$ with exceptional support is a formal sum $D=\sum_{i=1}^r a_iE_i$ with $a_i\in \RR$ for all $i$. $D$ is said to be effective if $a_i\ge 0$ for all $i$. $D$ is said to be a rational divisor if all $a_i\in \QQ$ and $D$ is said to be an integral divisor if all $a_i\in \ZZ$.

Now suppose that $D$ is an effective integral divisor with exceptional support. In this case, $D$ is a Weil divisor on $X$.  A rank one reflexive sheaf is associated to the Weil divisor $D$.
Let $U$ be the open set of regular points of $X$ and let $i:U\rightarrow X$ be the inclusion. We have that $\dim(X\setminus U)\le d-2$ since $X$ is normal. 
Then $D|U$ is a Cartier divisor.
The reflexive coherent sheaf $\mathcal O_X(-D)$ of $\mathcal O_X$-modules is defined by
$\mathcal O_X(-D)=i_*\mathcal O_U(-D|U)$ The basic properties of this sheaf are developed for instance in \cite[Section 13.2]{C5}.  Since $R$ is normal, we have that $\Gamma(X,\mathcal O_X)=R$, and if $D$ is a nontrivial integral exceptional divisor with effective support, then $I(D)=\Gamma(X,\mathcal O_X(-D))$ is an $m_R$-primary ideal.

Now let $D=\sum_{i=1}^r a_iE_i$ be an effective real divisor with exceptional support. 
 Let $\mathcal I(D)$ be the $m_R$-filtration
$\mathcal I(D)=\{I(mD)\}$ where 
$$
I(mD)=\Gamma(X,\mathcal O_X(-\lceil \sum_{i=1}^r ma_iE_i\rceil)).
$$
The round up $\lceil x\rceil$ of a real number $x$ is the smallest integer $a$ such that $x\le a$.
When $D$ is an integral divisor, we have that 
$I(mD)=\Gamma(X,\mathcal O_X(-mD))$ for all $m$.

Let $\nu_{E_i}$ be the $m_R$-valuation whose valuation ring is $\mathcal O_{X,E_i}$
for $1\le i\le r$. Let $\tau_{m,i}=\tau_{E_i,m}(D)
=\tau_{m,\nu_{E_i}}(\mathcal I(D))$.
Now define
$$
\gamma_{E_i}(D)=\gamma_{\nu_{E_i}}(\mathcal I(D)).
$$
 We have that 
$$
I(mD)= \Gamma(X,\mathcal O_X(-\lceil \sum_{i=1}^ra_imE_i\rceil))=\{f\in R\mid \nu_{E_i}(f)\ge \lceil ma_i\rceil \mbox{ for }1\le i\le r\}.
%=\cap_{i=1}^rI(\nu_{E_i})_{\lceil ma_i\rceil}.
$$
Thus $\tau_{E_i,m}(D)\ge ma_i$ for all $m\in\NN$, and so 
\begin{equation}\label{eqAR12}
\gamma_{E_i}(D)\ge a_i\mbox{ for all $i$}.
\end{equation}

\begin{Lemma}(\cite[Lemma 3.1]{C6})\label{LemmaAR1}
We have that
$$
I(mD)=\Gamma(X,\mathcal O_X(-\lceil \sum_{i=1}^ra_imE_i\rceil))
=\Gamma(X,\mathcal O_X(-\lceil \sum_{i=1}^r m\gamma_{E_i}(D)E_i\rceil))
$$
for all $m\in \NN$.
\end{Lemma}

\begin{proof} We have that 
$$
\Gamma(X,\mathcal O_X(-\lceil \sum_{i=1}^rm\gamma_{E_i}(D)E_i\rceil))
\subset \Gamma(X,\mathcal O_X(-\lceil \sum_{i=1}^ra_im E_i\rceil))
$$
by (\ref{eqAR12}).

Suppose that $f\in \Gamma(X,\mathcal O_X(-\lceil\sum_{i=1}^ra_imE_i\rceil))$. Then 
$\nu_{E_i}(f)\ge\tau_{E_i,m}(D)\ge m\gamma_{E_i}(D)$ for all $i$, so that
$\nu_{E_i}(f)\ge \lceil m\gamma_{E_i}(D)\rceil$ for all $i$ since $\nu_{E_i}(f)\in \NN$.
\end{proof}

\subsection{Divisors on blowups of  local domains}\label{Not}
In this subsection, suppose that $R$ is an excellent $d$-dimensional local domain. Let $S$ be the normalization of $R$, which is a finitely generated $R$-module,  and let $m_1,\ldots,m_t$ be the maximal ideals of $S$. Let $\phi:X\rightarrow \mbox{Spec}(R)$ be a birational projective morphism such that $X$ is the normalization of   the blowup of an $m_R$-primary ideal. Since $X$ is  normal, $\phi$ factors through $\mbox{Spec}(S)$. Let $\phi_i:X_i\rightarrow \mbox{Spec}(S_{m_i})$ be the induced  projective morphisms where $X_i=X\times_{\mbox{Spec}(S)}\mbox{Spec}(S_{m_i})$. For $1\le i\le t$, let $\{E_{i,j}\}$ be the prime exceptional divisors in $\phi_i^{-1}(m_i)$.

A real divisor $D$ on $X$ with exceptional support is a formal sum $D=\sum a_{i,j}E_{i,j}$ with $a_{i,j}\in \RR$ for all $i,j$. $D$ is said to be effective if all $a_{i,j}\ge 0$. $D$ is said to be a rational divisor if all $a_{i,j}\in \QQ$ and $D$ is said to be an integral divisor if all $a_{i,j}\in \ZZ$.

Suppose that $D$ is an effective real divisor on $X$ with exceptional support. Write $D=\sum_{i,j} a_{i,j}E_{i,j}$ with  $a_{i,j}\in\RR_{\ge 0}$. Define $D_i=\sum_ja_{i,j}E_{i,j}$ for $1\le i\le t$.

Let $D=\sum_{i,j}a_{i,j}E_{i,j}$ be an effective real divisor with exceptional support on $X$.
Let $\mathcal I(D)$ be the $m_R$-filtration 
$\mathcal I(D)=\{I(mD)\}$ where 
$$
I(mD)=\Gamma(X,\mathcal O_X(-\lceil \sum_{i=1}^r ma_{i,j}E_{i,j}\rceil))\cap R.
$$

When $D$ is an integral divisor, we have that 
$I(mD)=\Gamma(X,\mathcal O_X(-mD))\cap R$ for all $m$.

 %The reflexive coherent sheaf $\mathcal O_X(-D)$ of $\mathcal O_X$-modules  is defined by $\mathcal O_X(-D)=i_*\mathcal O_U(-D|U)$ where $U$ is the open subset of regular points of $X$ and $i:U\rightarrow X$ is the inclusion. We have that  $\dim (X\setminus U)\le d-2$ since $X$ is normal. The basic properties of this sheaf are developed for instance in \cite[Section 13.2]{C5}. 
%We have that $S\subset \mathcal O_{X,p}$ for all $p\in X$, since $\mathcal O_{X,p}$ is  normal. Now $\Gamma(X,\mathcal O_X)$ is a domain with the same quotient field as $R$, and is a finitely generated $R$-module since $\phi$ is proper. Thus $\Gamma(X,\mathcal O_X)=\Gamma(X,\mathcal O_X(0))=S$.

 Now let $D$ be an effective integral divisor with exceptional support.

Let
\begin{equation}\label{eqX30}
\begin{array}{l}
J(D)=\Gamma(X,\mathcal O_X(-D)),\\
 J(D_i)=\Gamma(X_i,\mathcal O_{X_i}(-D_i)),\\
   I(D)=J(D)\cap R,\\
    I(D_i)=J(D_i)\cap R.
    \end{array}
\end{equation}
 We have that
\begin{equation}\label{eqR6}
S/J(D)\cong \bigoplus_{i=1}^t S_{m_i}/\Gamma(X_i,\mathcal O_{X_i}(-D_i))\cong \bigoplus_{i=1}^t S_{m_i}/J(D_i)
\end{equation}
and so
\begin{equation}\label{eqR15}
\ell_R(S/J(D))=\sum_{i=1}^t\ell_R(S_{m_i}/J(D_i))=\sum_{i=1}^t[S/m_i:R/m_R]\ell_{S_{m_i}}(S_{m_i}/J(D_i)).
\end{equation}
We have that $[S/m_i:R/m_R]<\infty$ for all $i$ since $S$ is a finitely generated $R$-module.

Let $D(1),\ldots,D(r)$ be effective integral divisors on $X$ with exceptional support.

\begin{Lemma}\label{LemmaR1}(\cite[Lemma 2.2]{C6})   For $n_1,\ldots,n_r\in \NN$,
$$
\lim_{n\rightarrow\infty}\frac{\ell_R(R/I(nn_1D(1))\cdots I(nn_rD(r)))}{n^d}
=
\lim_{n\rightarrow\infty}\frac{\ell_R(S/J(nn_1D(1))\cdots J(nn_rD(r)))}{n^d}.
$$

\end{Lemma}

\subsection{Divisorial $m_R$-Filtrations}  In this subsection, let $R$ be a  local domain.
%We recall some material from \cite[Section 3]{C6}.

Let $\mu_1,\ldots,\mu_s$ be $m_R$-valuations, and $a_1,\ldots,a_s\in \NN$ with $a_1+\cdots+a_s>0$. Then
we define a divisorial $m_R$-filtration 
$$
\mathcal I(a_1\mu_1+\cdots+a_s\mu_s)=\{I(a_1\mu_1+\cdots+a_s\mu_s)_n\}
$$
 by
$$
I(a_1\mu_1+\cdots+a_s\mu_s)_n=I(\mu_1)_{na_1}\cap\cdots\cap I(\mu_s)_{na_s}.
$$

We can also define real divisorial $m_R$-filtrations by taking $a_1,\ldots,a_s\in \RR_{\ge 0}$ and defining an $m_R$-filtration
$\mathcal I(a_1\mu_1+\cdots+a_s\mu_s)=\{I(a_1\mu_1+\cdots+a_s\mu_s)_n\}$ by
$$
I(a_1\mu_1+\cdots+a_s\mu_s)_n=I(\mu_1)_{\lceil na_1\rceil }\cap\cdots\cap I(\mu_s)_{\lceil na_s\rceil}.
$$
A real divisorial $m_R$-filtration will be called a rational divisorial $m_R$-filtration if $a_i\in \QQ_{\ge 0}$ for all $i$ and will be called an integral divisorial $m_R$-filtration, or just a divisorial $m_R$-filtration if
$a_i\in \NN$ for all $i$.

The first statement of the following  proposition is proven for the case when
 $\mathcal I=\mathcal I(a_1\mu_1+\cdots+a_t\mu_t)$ is an integral  divisorial $m_R$-filtration in \cite[Proposition 2.1]{C6}. 
However, the proof given there extends to the case when $\mathcal I$ is a real divisorial $m_R$-filtration. The second statement follows from \cite[Theorem 1.4]{CSV}.
 
\begin{Proposition}\label{PropPos}(\cite[Proposition 2.1]{C6}, \cite[Theorem 1.4]{CSV})   Suppose that $R$ is an excellent, analytically irreducible $d$-dimensional local domain.
\begin{enumerate}
\item[1)] Suppose that $\mathcal I=\mathcal I(a_1\mu_1+\cdots+a_t\mu_t)$ is a real divisorial $m_R$-filtration. 
Then 
$$
e_R(\mathcal I;R)>0.
$$
\item[2)] Suppose that $\mathcal I(1),\ldots,\mathcal I(r)$ are $m_R$-filtrations  such that $e_R(\mathcal I(j))>0$ for all $j$.
Then
$$
e_R(\mathcal I(1)^{[d_1]},\ldots,\mathcal I(r)^{[d_r]};R)>0
$$
for all $d_1,\ldots,d_r\in \NN$ with $d_1+\cdots+d_r=d$.
\end{enumerate}
\end{Proposition}

If $\mathcal I$ is a real divisorial $m_R$-filtration on an analytically irreducible excellent local ring $R$, then Rees's Izumi Theorem \cite{R3} shows that $\gamma_{\mu}(\mathcal I)>0$ for all $m_R$-valuations $\mu$.

\subsection{Representations of divisorial $m_R$-filtrations on normal local rings}
In this subsection, suppose that $R$ is a normal excellent local domain.
We now define  a representation of a real divisorial $m_R$-filtration $\mathcal I(b_1\mu_1+\cdots+b_s\mu_s)$.
 Let $\phi:X\rightarrow \mbox{Spec}(R)$ be a birational projective morphism that is the   blowup of an $m_R$-primary ideal such that $X$ is normal, and so that if 
$E_1,\ldots,E_r$ are the prime exceptional divisors of $\phi$
and  $\nu_{E_i}$ are the discrete valuations with valuation rings $\mathcal O_{X,E_i}$ for $1\le i\le r$, then $\mu_i=\nu_{E_i}$ for $1\le i\le s$ with $1\le s\le r$.

The pair of  $X\rightarrow \mbox{Spec}(R)$ and the real divisor $b_1E_1+\cdots+b_sE_s$ will be called a representation of the real divisorial $m_R$-filtration $\mathcal I(b_1\mu_1+\cdots+b_s\mu_s)$.

We remark that it may not be possible to construct an $X$ for which $r=s$, even in dimension $d=2$. This follows from the example of a two dimensional excellent normal local domain without a ``one fibered ideal''  given in \cite{C8}.

We now tie this back in with our original  real divisorial $m_R$-filtration $\mathcal I(b_1\mu_1+\cdots+b_s\mu_s)$, for which the pair of $X$ and $b_1E_1+\cdots+b_sE_s$ is a representation. Letting $D$ be the real divisor $D=b_1E_1+\cdots+b_sE_s$ on $X$,  we have for all $m\in\NN$ that
$$
I(m(\gamma_{E_1}(D)E_1+\cdots+\gamma_{E_r}(D)E_r))\\
=
I(mD)=
I(b_1\mu_1+\cdots+b_s\mu_s)_m
$$
for all $m$. Thus we have equality of $m_R$-filtrations
$$
\mathcal I(\gamma_{E_1}(D)E_1+\cdots+\gamma_{E_r}(D)E_r)
=
\mathcal I(D)
=\mathcal I(b_1\mu_1+\cdots+b_s\mu_s).
$$

In particular, every divisorial $m_R$-filtration has the form $\mathcal I(D)$ for some  divisor
$D=\sum a_iE_i$ with exceptional support on some $X$. 

If the pair $X'$ and $D'$ is another representation of $\mathcal I(b_1\mu_1+\cdots+b_s\mu_s)$, then there are prime exceptional divisors $E_1',\ldots, E_s'$ on $X'$ such that we have equality of  local rings $\mathcal O_{X,E_i}=\mathcal O_{X',E_i'}=\mathcal O_{\mu_i}$ for $1\le i\le s$ and $D'=\sum_{i=1}^sb_iE_i'$. 

We remark that  even when $\mathcal I$ is an integral divisorial $m_R$-filtration, $\gamma_{\mu}(\mathcal I)$ can be an irrational number for some $m_R$-valuation $\mu$. 
From \ref{Theorem4}, we find an example of $X$ with two prime exceptional divisors $E_1$ and $E_2$ such that 
$$
\gamma_{E_1}(E_2)=\frac{3}{9-\sqrt{3}}
$$
is an irrational number.  We will often abuse notation, denoting a real divisorial  $m_R$-filtration by $\mathcal I(D)$.

\subsection{Bounded $m_R$-Filtrations}\label{SubSecBound}

\begin{Definition} Let $R$ be a local ring and $\mathcal I$ be an $m_R$-filtration. Let $R[\mathcal I]$ be the $R$-algebra
$$
R[\mathcal I]=\sum_{m\ge 0}I_mt^m
$$
and
$\overline{R[\mathcal I]}$ be the integral closure of $R[\mathcal I]$ in the polynomial ring $R[t]$.
\end{Definition}

If $I$ is an ideal in a local ring $R$, let $\overline I$ denote its integral closure. 

\begin{Lemma}\label{BdLemma1} Let $R$ be a local ring and $\mathcal I$ be an $m_R$-filtration. Then
$$
\overline{R[\mathcal I]}=\sum_{m\ge 0}J_mt^m
$$
where $\{J_m\}$ is the $m_R$-filtration
$$
J_m=\{f\in R\mid f^r\in \overline{I_{rm}}\mbox{ for some }r>0\}.
$$
\end{Lemma}

\begin{Remark} If $\mathcal I=\{I^i\}$ is the filtration of powers of a fixed $m_R$-primary ideal $I$ then $J_m=\overline{I_m}$ for all $m$.
\end{Remark}

\begin{proof} 
The ring $\overline{R[\mathcal I]}$ is graded by \cite[Theorem 2.3.2]{HS}. Thus it suffices to show that for $f\in R$ and $n\in \ZZ_{>0}$ we have that $ft^n$ is integral over  $R[\mathcal I]$ if and only if $f\in\overline{I_{rn}}$ for some $r\ge 1$. Now $ft^n$ is integral over $R[\mathcal I]$ if and only if there exists a homogeneous relation
\begin{equation}\label{eqBd1}
(ft^n)^d+a_{d-1}t^n(ft^n)^{d-1}+\cdots+a_it^{n(d-i)}(ft^n)^i+\cdots+a_0t^{nd}=0
\end{equation}
for some $d>0$ with $a_i\in I_{n(d-i)}$ for all $i$.

We will show that $ft^n$ is integral over $R[\mathcal I]$ if and only if there exists $r>0$ such that $f^r\in \overline{I_{rn}}$. 

Suppose that $f^r\in \overline{I_{rn}}$. Then there exists a relation
$$
(f^r)^d+a_{d-1}(f^r)^{d-1}+\cdots+a_i(f^r)^i+\cdots+a_0=0
$$
with $a_i\in (I_{rn})^{d-i}\subset I_{rn(d-i)}$ for all $i$. Multiply this relation by $t^{rnd}$ to get a relation of type (\ref{eqBd1}), showing that $(ft^n)^r$ is integral over $R[\mathcal I]$. Thus $ft^n$ is integral over $R[\mathcal I]$.

Now  suppose that $ft^n$ is integral over $R[\mathcal I]$. We will break the proof up into two cases.
\vskip .2truein
\noindent{\bf Case 1.} Assume that $R[\mathcal I]$ is Noetherian.  Then there exists $r>0$ such that $I_{ri}=I_r^i$ for all $i\in \ZZ_{>0}$ by \cite[Proposition 3, Section 1.3, Chapter III]{Bou}. Since $f^rt^{rn}$ is integral over $R[\mathcal I]$, there exists a relation (\ref{eqBd1}) with $f$ replaced with $f^r$ and $n$ with $rn$, so $a_i\in I_{rn(d-i)}=I_{rn}^{d-i}$ and thus $f^r\in \overline{I_{rn}}$.
\vskip .2truein
\noindent{\bf Case 2.}(General Case)  Assume that $\mathcal I$ is an arbitrary $m_R$-filtration. 

For $a\in \ZZ_{>0}$, let $\mathcal I_a=\{I_{a,n}\}$ where $I_{a,n}=I_n$ if $n\le a$ and if $n>a$ then $I_{a,n}=\sum I_{a,i}I_{a,j}$  where the sum is over $i,j>0$ such that $i+j=n$.

Now $ft^n$  integral over $R[\mathcal I]$ implies there exits $a>0$ such that $ft^n$ is integral over $R[\mathcal I_a]$. By Case 1, there exists $r>0$ such that 
$f^r\in \overline{I_{a,rn}}\subset \overline{I_{rn}}$.  
\end{proof}

\begin{Lemma}\label{BdLemma2} Let $R$ be a local domain and $\mathcal I(D)$ be a divisorial $m_R$-filtration. Then $R[\mathcal I(D)]$ is integrally closed in $R[t]$.
\end{Lemma}

\begin{proof} We have that $\mathcal I(D)=\mathcal I(\alpha_1\mu_1+\cdots+\alpha_s\mu_s)$ where $\mu_1,\ldots,\mu_s$ are $m_R$-valuations and $\alpha_1,\ldots,\alpha_s\in \RR_{>0}$. Since $\overline{R[\mathcal I]}$ is graded, we must show that if $f\in R$ and $n\in \ZZ_{>0}$ are such that $ft^n\in \overline{R[\mathcal I(D)]}$, then $f\in I(nD)=I(\mu_1)_{\lceil n\alpha_1\rceil}\cap\cdots\cap I(\mu_s)_{\lceil n\alpha_s\rceil}$.
Now $ft^n\in \overline{R[\mathcal I(D)]}$ implies there exists a relation 
$$
f^d+a_{d-1}f^{d-1}+\cdots+a_if^i+\cdots+a_0=0 
$$
with $a_i\in I(n(d-i)D)$ for all $i$ by (\ref{eqBd1}). Suppose that $f\not\in I(nD)$. Then there exists $j$ such that $\mu_j(f)<\lceil n\alpha_j\rceil$. Thus $\mu_j(f)<n\alpha_j$ since $\mu_j(f)\in \NN$ and so 
$$
(d-i)\mu_j(f)<n(d-i)\alpha_j\le \lceil n(d-i)\alpha_j\rceil
$$
for all $i$ with $0\le i<d$. Thus
$$
d\mu_j(f)<\lceil n(d-i)\alpha_j\rceil+i\mu_j(f)
$$
for all $i$ with $0\le i<d$ so that
$$
\mu_j(f^d+a_{d-1}f^{d-1}+\cdots+a_if^i+\cdots+a_0)=d\mu_j(f)\in\NN.
$$
Thus $f^d+a_{d-1}f^{d-1}+\cdots+a_if^i+\cdots+a_0\ne 0$, a contradiction, and so $f\in I(nD)$.
\end{proof}

\begin{Definition} Suppose that $R$ is a local domain. An $m_R$-filtration $\mathcal I=\{I_n\}$ is said to be bounded if there exists an integral divisorial $m_R$-filtration $\mathcal I(D)$ such that
$$
\overline{R[\mathcal I]}=R[\mathcal I(D)].
$$
An $m_R$-filtration $\mathcal I=\{I_n\}$ is said to be  real bounded  if there exists a real divisorial $m_R$-filtration $\mathcal I(D)$ such that
$$
\overline{R[\mathcal I]}=R[\mathcal I(D)].
$$

\end{Definition}

\begin{Lemma}\label{BdLemma3} Suppose that $R$ is an excellent local domain and $\mathcal I=\{I^n\}$ 
is the $m_R$-filtration of powers of a fixed $m_R$-primary ideal $I$. Then $\mathcal I$ is bounded.
\end{Lemma}

\begin{proof}
We have that
$\overline{R[\mathcal I]}=\oplus_{n\ge 0}\overline{I^n}u^n$ where $\overline{I^n}$ is the integral closure of $I^n$ in $R$.
The algebra $\oplus_{n\ge 0}\overline{I^n}u^n$ is a finite $R[\mathcal I]$-module, so that $\{\overline{I^n}\}$ is a Noetherian filtration.  Let $\phi:X\rightarrow \mbox{Spec}(R)$ be the normalization of the blowup of $I$ and $E_1,\ldots, E_t$ be the prime exceptional divisors of $\phi$. Then $I\mathcal O_X=\mathcal O_X(-a_1E_1-\cdots-a_tE_t)$ for some $a_1,\ldots,a_t\in \ZZ_{>0}$ is an ample Cartier divisor on $X$ and $I^n\mathcal O_X=\mathcal O_X(-na_1E_1-\cdots-na_tE_t)$ for all $n\in \NN$. Thus for  $n\in \NN$,
$$
\overline{I^n}=\Gamma(X,\mathcal O_X(-na_1E_1-\cdots-na_tE_t))\cap R
=I(a_1\nu_{E_1}+\cdots+a_t\nu_{E_t})_n
$$
where $\nu_{E_i}$ is the $m_R$-valuation whose valuation ring is $\mathcal O_{X,E_i}$.
Thus   $\{\overline{I^n}\}$ is the divisorial filtration $\mathcal I(a_1
\nu_{E_1}+\cdots+a_t\nu_{E_t})=\mathcal I(D)$ where $D=a_1E_1+\cdots+a_tE_t$ and $\overline{R[\mathcal I]}=R[\mathcal I(D)]$.
\end{proof}

\begin{Proposition}\label{BoundProp} Suppose that $R$ is a local ring with $\dim N(\hat R)<d$ and 
$$
\mathcal I(1),\ldots,\mathcal I(r),\mathcal I'(1),\ldots,\mathcal I'(r)
$$
 are $m_R$-filtrations such that 
$\overline{R[\mathcal I'(i)]}=\overline{R[\mathcal I(i)]}$ for $1\le i\le r$. Then
we have equality of all mixed multiplicities
\begin{equation}\label{eqB1}
e(\mathcal I(1)^{[d_1]},\ldots,\mathcal I(r)^{[d_r]})
=
e(\mathcal I'(1)^{[d_1]},\ldots,\mathcal I'(r)^{[d_r]}).
\end{equation}
\end{Proposition}

\begin{proof} Write $\overline{R[\mathcal I(i)]}=\oplus_{n\ge 0}J(i)_n$  and let
$\mathcal J(i)=\{J(i)_n\}$ for $1\le i\le r$. We will show that for all mixed multiplicities,
\begin{equation}\label{eqB2}
e(\mathcal I(1)^{[d_1]},\ldots,\mathcal I(r)^{[d_r]})
=
e(\mathcal J(1)^{[d_1]},\ldots,\mathcal J(r)^{[d_r]}).
\end{equation}
The same argument applied to $\mathcal I'(1),\ldots,\mathcal I'(r)$ and $\mathcal J(1),\ldots,\mathcal J(r)$ will show show that equation (\ref{eqB1}) holds. Let 
$$
P(n_1,\ldots,n_r)=\lim_{m\rightarrow \infty}\frac{\ell_R(R/J(1)_{mn_1}\cdots J(r)_{mn_r})}{m^d}
$$
and
$$
Q(n_1,\ldots,n_r)=\lim_{m\rightarrow \infty}\frac{\ell_R(R/I(1)_{mn_1}\cdots I(r)_{mn_r})}{m^d}.
$$
Since $\oplus_{m\ge 0}J(i)_m$ is integral over $\oplus_{m\ge 0}I(i)_m$ for all $i$, we have that the graded $R$-algebra 
$\oplus_{m_1,\ldots,m_r\ge 0} J(1)_{m_1}\cdots J(r)_{m_r}$ is integral over the graded $R$-algebra 
$$
\oplus_{m_1,\ldots,m_r\ge 0} I(1)_{m_1}\cdots I(r)_{m_r}.
$$
 Thus
for fixed $n_1,\ldots,n_r\in \NN$, we have that
$\oplus_{m\ge 0}J(1)_{mn_1}\cdots J(r)_{mn_r}$ is integral over $\oplus_{m\ge 0}I(1)_{mn_1}\cdots I(r)_{mn_r}$. By \cite[Theorem 6.9]{CSS} or \cite[Appendix]{C6} (summarized in Subsection \ref{SubSecEqChar}) we have that
$$
P(n_1,\ldots,n_r)=Q(n_1,\ldots,n_r)
$$
for all $n_1,\ldots,n_r\in \NN$. Since $P(n_1,\ldots,n_r)$ and $Q(n_1,\ldots,n_r)$ are homogeneous polynomials of the same degree $d$, we have that $P(n_1,\ldots,n_r)$ and $Q(n_1,\ldots,n_r)$ have the same values for all $n_1,\ldots,n_r$ in the infinite field $\QQ$. Thus their coefficients are equal showing (\ref{eqB2}).

\end{proof}

\section{A framework to compute  multiplicities}\label{SecFrame} 

In this section, we summarize a construction from \cite[Section 3]{C6}.

Let $R$ be an excellent  local domain of dimension $d$ and let $\mu$ be an $m_R$-valuation. Since $R$ is  excellent, there exists a birational projective morphism
$\phi:X\rightarrow \mbox{Spec}(R)$ such that $X$ is the normalization of the blow up of an $m_R$-primary ideal, $X$ is normal and there exists a prime exceptional divisor $E$ on $X$ such that $\mu=\nu_E$. 

Let $t$ be a generator of the maximal ideal of the valuation ring $\mathcal O_{X,E}$.  Regarding $t^{-1}$ as an element of the quotient field $K$ of $R$, we compute its divisor $(t^{-1})=-E+D$ on $X$, which is a Cartier divisor and where $D$ is a Weil divisor which does not contain $E$ in its support ($D$ will have non exceptional support).
 Write $D=D_1-D_2$ where $D_1$ and $D_2$ are effective Weil divisors which do not contain $E$ in their supports. 

 Since $X\rightarrow \mbox{Spec}(R)$ is projective, there exists an  ample Cartier divisor $H$  on $X$. 
 %Thus there is a positive multiple $n$ of $A$ such that $\Gamma(E,\mathcal O_X(nA)\otimes\mathcal O_E)\ne 0$ and
 %&$\Gamma(X,\mathcal O_X(nA))\rightarrow \Gamma(E,\mathcal O_X(nA)\otimes\mathcal %O_E)$ is a surjection. There exists an effective Ample Cartier divisor $H$ on $X$ such that $E$ is not in the support of $H$   ($H$ is linearly equivalent to $nA$). 
 
 For all $n$, there exist natural inclusions of reflexive rank 1 sheaves
$$
\mathcal O_X(-D_2-E+nH)\subset \mathcal O_X(-D_2+nH)\subset \mathcal O_X(nH).
$$
This can be seen by restricting to the nonsingular locus $U$ of $X$ (which has codimension $\ge 2$ in $X$) and then pushing the sequence forward to $X$. Taking global sections, we thus have inclusions

$$
\Gamma(X,\mathcal O_X(-D_2-E+nH))\subset \Gamma(X,\mathcal O_X(-D_2+nH))\subset \Gamma(X,\mathcal O_X(nH)).
$$

Since $H$ is an ample Cartier divisor, there exists a multiple $n$ of $H$ such that 
$$
\Gamma(X,\mathcal O_X(-D_2-E+nH))
$$
 is a proper subset of 
$\Gamma(X,\mathcal O_X(-D_2+nH))$.  Thus there exists 
$\sigma\in \Gamma(X,\mathcal O_X(nH))$ such that the divisor $(\sigma)$  (considering $\sigma$ as a global section of $\mathcal O_X(nH)$) is an effective Cartier divisor which has the property that the Weil divisor $(\sigma)-D_2$ is effective and $E$ is not in the support of $(\sigma)-D_2$.

Thus $-E+D+(\sigma)$ is a Cartier divisor and
$$
-E+D+(\sigma)=-E+D_1-D_2+(\sigma)=-E+F
$$
 where $F=D_1-D_2+(\sigma)$ is an effective Weil divisor which does not contain $E$ in its support. 

The natural inclusions $\mathcal O_X(-nE)\rightarrow \mathcal O_X(-nE+nF)$ for $n\in \NN$ induce inclusions
$$
I(\nu)_n=\Gamma(X,\mathcal O_X(-nE))\cap R\rightarrow \Gamma(X,\mathcal O_X(-nE))\rightarrow \Gamma(X,\mathcal O_X(-nE+nF))
$$
for all $n$. 

%%%%%%%

Let $q\in E$ be a  closed point that   is nonsingular on both $X$ and  $E$ and is not contained in the support of $F$.  Let
\begin{equation}\label{eqAR2}
X=X_0\supset X_1=E\supset \cdots \supset X_d=\{q\}
\end{equation}
be a flag; that is, the $X_j$ are subvarieties of $X$ of dimension $d-j$ such that there is a regular system of parameters $b_1,\ldots,b_d$ in $\mathcal O_{X,q}$ such that  $b_1=\cdots=b_j=0$ are local equations of $X_j$ for $1\le j\le d$. 

The flag determines a valuation $\nu$ on the quotient field $K$ of $R$  which dominates $R$ as follows.  We have a sequence of natural surjections of regular local rings
\begin{equation}\label{eqGA3} \mathcal O_{X,q}=
\mathcal O_{X_0,q}\overset{\sigma_1}{\rightarrow}
\mathcal O_{X_1,q}=\mathcal O_{X_0,q}/(b_1)\overset{\sigma_2}{\rightarrow}
\cdots \overset{\sigma_{d-1}}{\rightarrow} \mathcal O_{X_{d-1},q}=\mathcal O_{X_{d-2},q}/(b_{d-1}).
\end{equation}
Define a rank-$d$ discrete valuation $\nu$ on $K$ (an Abhyankar valuation)  by prescribing for $s\in \mathcal O_{X,q}$,
$$
\nu(s)=({\rm ord}_{X_1}(s),{\rm ord}_{X_2}(s_1),\cdots,{\rm ord}_{X_d}(s_{d-1}))\in (\ZZ^d)_{\rm lex}
$$
where 
$$
s_1=\sigma_1\left(\frac{s}{b_1^{{\rm ord}_{X_1}(s)}}\right),
s_2=\sigma_2\left(\frac{s_1}{b_2^{{\rm ord}_{X_2}(s_1)}}\right),\ldots,
s_{d-1}=\sigma_{d-1}\left(\frac{s_{d-2}}{b_{d-1}^{{\rm ord}_{X_{d-1}}(s_{d-2})}}\right)
$$
and $\mbox{ord}_{X_{j+1}}(s_j)$ is the highest power of $b_{j+1}$  that divides $s_j$ in $\mathcal O_{X_j,q}$.
We have that
$$
\nu(s)=\left(\mu(s)=\nu_{E}(s),\omega\left(\frac{s}{b_1^{\mu_{E}(s)}}\right)\right)
$$
where $\omega$ is the rank-$(d-1)$ Abhyankar valuation on the function field of $ E$ determined by the flag
$$
E=X_1\supset \cdots \supset X_d=\{q\}
$$
on the projective $k$-variety $E$, where $k=R/m_R$.

%%%%%%%%%%%%%%%
By our construction, $\mathcal O_X(-E+F)$ is an invertible sheaf on $X$ and so $\mathcal O_X(-E+F)\otimes\mathcal O_{E}$ is an invertible sheaf on $E$.
Consider the graded linear series $L_n:=\Gamma(E,\mathcal O_X(-nE+nF)\otimes_{\mathcal O_X}\mathcal O_E)$ on $E$. Recall that  $b_1=0$ is a local equation of $E$ in $\mathcal O_{X,q}$. Let $g=b_1$. Thus, since $q$ is not in the support of $F$, for $n\in \NN$, we have a natural commutative diagram
\begin{equation}\label{eqT10}
\begin{array}{cccccc}
I(\mu)_n&\subset\Gamma(X,\mathcal O_X(-nE))&\rightarrow &\Gamma(X,\mathcal O_X(-nE+nF))&\rightarrow & \Gamma(E,\mathcal O_X(-nE+nF)\otimes\mathcal O_{E})\\
 &\downarrow&&  \downarrow&&\downarrow\\
&\mathcal O_X(-nE)_q&\stackrel{=}{\rightarrow}&  \mathcal O_X(-nE+nF)_q&\rightarrow&\mathcal O_X(-nE+nF)_q\otimes_{\mathcal O_{X,q}}\mathcal O_{E,q}\\
&=\mathcal O_{X,q}g^n&&   =\mathcal O_{X,q}g^n&&\cong \mathcal O_{ E,q}\otimes_{\mathcal O_{X,q}}\mathcal O_{X,q}g^n\end{array}
\end{equation}
where we denote the rightmost vertical arrow by $s\mapsto \epsilon_n(s)\otimes g^n$ 
and the bottom horizontal arrow is 
$$
f\mapsto \left[\frac{f}{g^n}\right]\otimes g^n,
$$
where $\left[\frac{f}{g^n}\right]$ is the class of $\frac{f}{g^n}$ in $\mathcal O_{E,q}$.

Let $\Xi$ be the semigroup defined by
\begin{equation}\label{eq21}
\Xi = \{(n,\omega(\epsilon_n(s)))\mid n\in \NN\mbox{ and } s\in \Gamma( E,\mathcal O_X(-nE+nF)\otimes_{\mathcal O_X}\mathcal O_{E})\}\subset \ZZ^d,
\end{equation}
and let 
\begin{equation}\label{eq22}
\mbox{$\Delta(\Xi)$ be the intersection of the closed convex cone generated by $\Xi$ in $\RR^d$ with $\{1\}\times\RR^{d-1}$.}
\end{equation}
 By the proof of Theorem 8.1 \cite{C2}, $\Delta(\Xi)$ is  compact and convex. Let 
\begin{equation}\label{eq23}
\Xi_n:=\{(n,\omega(\epsilon_n(s)))\mid s\in \Gamma(E,\mathcal O_Z(-nE+nF)\otimes_{\mathcal O_X}\mathcal O_{E})\}
\end{equation}
be the elements of $\Xi$ at level $n$.

We will require the following important observation, which follows from the diagram (\ref{eqT10}).
\begin{equation}\label{eq31}
\mbox{Suppose that  $f\in R$ and $\nu(f)=(a_1,\ldots,a_d)$. Then $\nu(f)\in \Xi_{a_1}$.}
\end{equation}

\section{Multiplicities of filtrations}\label{SecMF}

Let notations be as in Section \ref{SecFrame}, so that $R$ is an excellent local domain. We further assume in this section that $R$ is analytically irreducible.

Let $\mathcal I=\{I_i\}$ be an  $m_R$-filtration. 
For $m\in \NN$, define
$$
\Gamma(\mathcal I)_m=\{(\nu(f),m)\mid f\in I_m\}\subset \NN^{d+1}
$$
which are the elements at level $m$ of the semigroup
$$
\Gamma(\mathcal I)=\cup_{m\in \NN}\{(\nu(f),m)\mid f\in I_m\}.
$$
Define an associated closed convex set $\Delta(\mathcal I)\subset \RR^d$ as follows. 
Let $\Sigma(\mathcal I)$ be the closed convex cone with vertex at the origin generated by $\Gamma(\mathcal I)$ and let $\Delta(\mathcal I)=\Sigma(\mathcal I)\cap (\RR^d\times\{1\})$.
The set 
 $\Delta(\mathcal I)$ is the closure in the Euclidean topology of the set
$$
\left\{\left(\frac{a_1}{i},\cdots,\frac{a_d}{i}\right)\mid (a_1,\ldots,a_d,i)\in \Gamma(\mathcal I)\mbox{ and }i>0\right\}.
$$

For $m\in \NN$, define
$$
\Gamma(R)_m=\{(\nu(f),m)\mid f\in R\}\subset \NN^{d+1}.
$$
which are the elements at level $m$ of the semigroup
$$
\Gamma(R)=\cup_{m\in \NN}\{(\nu(f),m)\mid f\in R\}.
$$

Define an associated closed convex set $\Delta(R)\subset \RR^d$
as follows. 
Let $\Sigma(R)$ be the closed convex cone with vertex at the origin generated by $\Gamma(R)$ and let $\Delta(R)=\Sigma(R)\cap (\RR^d\times\{1\})$.
The set  $\Delta(R)$ is the closure in the Euclidean topology of the set
$$
\left\{\left(\frac{a_1}{i},\cdots,\frac{a_d}{i}\right)\mid (a_1,\ldots,a_d,i)\in \Gamma(R)\mbox{ and }i>0\right\}.
$$

\begin{Lemma} The closed convex set $\Delta(R)$ is a closed convex  cone in $\RR_{\ge 0}^d$
with vertex at the origin 0. 
\end{Lemma}

\begin{proof}  We identify $\RR^d\times\{1\}$ with $\RR^d$. We have that $(\nu(1),1)=(0,\ldots,0,1)\in \Gamma(R)$. Thus $(0,\ldots,0)\in \Delta(R)\subset \RR^d$.

Suppose that $(a_1,\ldots,a_d,i)\in \Gamma(R)$ with $i>0$. Let $x=(\frac{a_1}{i},\ldots,\frac{a_d}{i})\in \Delta(R)$. Let $\alpha\in \QQ_{>0}$. 
Then $\alpha=\frac{m}{n}$ with $m,n\in \ZZ_{>0}$. There exists $f\in R$ such that 
$\nu(f)=(a_1,\ldots,a_d)$. Now $f^m\in R$ so $(\nu(f^m),in)=(ma_1,\ldots,ma_d,in)\in \Gamma(R)$. Thus $\alpha x\in \Delta(R)$.

Suppose that $x\in \Delta(R)$ is non zero. Let $U=\{tx\mid t\in \RR_{\ge0}\}$. We must show that $U\subset \Delta(R)$. Let $y\in U$ be nonzero.  Then $y=sx$ for some $s\in \RR_{>0}$. 
Suppose that $\epsilon\in \RR_{>0}$. Choose $\delta \in \RR_{>0}$ such that $\delta<\min\{1,\frac{1}{|s|},\frac{1}{|x|}\}\epsilon$.
There exists $(a_1,\ldots,a_d,i)\in \Gamma(R)$ with $i>0$ such that $|x-(\frac{a_1}{i},\ldots,\frac{a_d}{i})|<\delta$ and there exist $m,n\in \ZZ_{>0}$ such that $|s-\frac{m}{n}|<\delta$.
Now $\frac{m}{n}(\frac{a_1}{i},\ldots,\frac{a_d}{i})\in \Delta(R)$ as we showed in the above paragraph. Let $\alpha=s-\frac{m}{n}$, $v=x-(\frac{a_1}{i},\ldots,\frac{a_d}{i})$. We compute
$$
\begin{array}{lll}
|y-\frac{m}{n}(\frac{a_1}{i},\ldots,\frac{a_d}{i})|
&=&|sx-(s-\alpha)(x-v)|=|sv+\alpha x-\alpha v|\\
&\le& |s||v|+|\alpha||x|+|\alpha||v|\le |s|\delta+|x|\delta+\delta^2<3\epsilon.
\end{array}
$$
Since we can make $\epsilon$ arbitrarily small and $\Delta(R)$ is a closed set,  we have that $y\in \Delta(R)$.

\end{proof}

For $c\in \RR_{>0}$, let
\begin{equation}\label{MF2}
H_c=\{(x_1,\ldots,x_d)\in \RR^d\mid x_1+\cdots+x_d=c\},
\end{equation}
\begin{equation}\label{MF3}
H^{-}_c=\{(x_1,\ldots,x_d)\in \RR^d\mid x_1+\cdots+x_d\le c\}
\end{equation}
and
\begin{equation}\label{eqMF4}
H^{+}_c=\{(x_1,\ldots,x_d)\in \RR^d\mid x_1+\cdots+x_d\ge c\}.
\end{equation}
Since $\Delta(R)$ is a closed cone in $\RR^d$ with vertex 0 and $cH_1=H_c$, $cH_1^{-}=H_c^{-}$, we have
\begin{equation}\label{eq1}
\Delta(R)\cap H_c=c(\Delta(R) \cap H_1)\mbox{ and }\Delta(R)\cap H_c^{-}=c(\Delta(R)\cap H_1^{-}).
\end{equation}

The proof of the following lemma is a simplification of the proofs of Lemmas \ref{LemmaA} and \ref{LemmaB} in the following Section \ref{SecMMF} (this is where the assumption that $R$ is analytically irreducible  is needed).

\begin{Lemma}\label{LemmaC} There exists $\lambda\in \ZZ_{>0}$ such that $\Delta(\mathcal I)\cap H_{\lambda}^+=\Delta(R)\cap H_{\lambda}^+$.
\end{Lemma}

For $c\in \RR_{>0}$ define $\Delta_c(\mathcal I)=\Delta(\mathcal I)\cap H_c^{-}$ and 
$\Delta_c(R)=\Delta(R)\cap H_c^{-}$. These sets are compact convex subsets of $\RR_{\ge 0}^d$.

Let $\lambda$ be the number defined in Lemma \ref{LemmaC}. If $\phi\ge\lambda$, then 
\begin{equation}\label{MF1}
\Delta(\mathcal I)\setminus \Delta_c(\mathcal I)=\Delta(R)\setminus \Delta_c(R).
\end{equation}
For $m\in \NN$, let 
$$
\Gamma_c(\mathcal I)_m=\{(\nu(f),m)\mid f\in I_m\mbox{ and }a_1+\cdots+a_d\le mc\}
$$
and
$$
\Gamma_c(R)=\{(\nu(f),i)\mid f\in R \mbox{ and }a_1+\cdots+a_d\le mc\}.
$$
Define semigroups 
$\Gamma_c(\mathcal I)=\cup_{m\in \NN}\Gamma_c(\mathcal I)_m$ and 
$\Gamma_c(R)=\cup_{m\in \NN}\Gamma_c(R)_m$. The semigroups $\Gamma_c(\mathcal I)$ and $\Gamma_c(R)$ satisfy the condition (5) of \cite[Theorem 3.2]{C2} since they are 
contained in $\RR_{\ge 0}^{d+1}\cap H_c^-$.

We now verify that condition (6) of \cite[Theorem 3.2]{C2} is satisfied; that is, that $\Gamma_c(\mathcal I)$ generates $\ZZ^{d+1}$ as a group. Let $G(\Gamma_c(\mathcal I))$ be the subgroup of $\ZZ^{d+1}$ generated by $\Gamma_c(\mathcal I)$. The value group of $\nu$ is $\ZZ^d$ and $e_j=\nu(b_j)$ for $1\le j\le d$ is the natural basis of $\ZZ^d$. Write $b_j=\frac{f_j}{g_j}$ with $f_j,g_j\in R$ for $1\le j\le d$. There exists $0\ne h\in I_1$. Thus $hf_j, hg_j\in I_1$. Possibly replacing $\lambda$ with a larger value, we then have that $(\nu(hf_j),1), (\nu(hg_j),1)\in \Gamma_c(\mathcal I)$ for $1\le j\le d$. 
Thus $(e_j,0)=(\nu(hf_j)-\nu(hg_j),0)\in G(\Gamma_c(\mathcal I))$ for $1\le j\le d$. Since $(\nu(hf_j),1)\in \Gamma_c(\mathcal I)$, we then have that $(0,1)\in G(\Gamma_c(\mathcal I))$, and so condition (6) of \cite[Theorem 3.2]{C2} is satisfied.

 Thus the limits
$$
\lim_{m\rightarrow\infty}\frac{\# \Gamma_c(\mathcal I)_m}{m^d}={\rm Vol}(\Delta_c(\mathcal I))
$$
and
$$
\lim_{m\rightarrow\infty}\frac{\# \Gamma_c(R)_m}{m^d}={\rm Vol}(\Delta_c(R))
$$
exist by \cite[Theorem 3.2]{C2}.  As in \cite[Theorem 5.6]{C3}, if $c\ge \lambda$, where $\lambda$ is chosen sufficiently large, then 
\begin{equation}\label{eqMF5}
\lim_{m\rightarrow \infty}\frac{\ell_R(R/I_m)}{m^d}=\delta[{\rm Vol}(\Delta_c(R))-{\rm Vol}(\Delta_c(\mathcal I))]
\end{equation}
where $\delta=[\mathcal O_{X,p}/m_p:R/m_R]$. 

Thus the  multiplicity 
$$
e_R(\mathcal I):=d!\lim_{m\rightarrow \infty}\frac{\ell_R(R/I_m)}{m^d}=
d!\delta[{\rm Vol}(\Delta_c(R))-{\rm Vol}(\Delta_c(\mathcal I))].
$$

Define 
\begin{equation}\label{eqMF6}
\Delta_{\lambda}(\mathcal I)=\Delta(\mathcal I)\cap H^-_{\lambda} \mbox{ for an $m_R$- filtration $\mathcal I$ and $\lambda\in \RR$.}
\end{equation}

\begin{Theorem}\label{Theorem2} Suppose that $R$ is an analytically irreducible excellent local domain and that $\mathcal I(1)$ and $\mathcal I(2)$ are $m_R$-filtrations  such that $I(1)_i\subset I(2)_i$ for all $i$ and $e_R(\mathcal I(1))=e_R(\mathcal I(2))$. Then
$$
\gamma_{\mu}(\mathcal I(1))=\gamma_{\mu}(\mathcal I(2))
$$
for all $m_R$-valuations $\mu$ of $R$.
\end{Theorem}

The proof which we give below follows from  the first part of the proof of \cite[Theorem 3.4]{C6}, applied to our filtrations $\mathcal I(1)$ and $\mathcal I(2)$ (instead of the divisorial $m_R$-filtrations $\mathcal I(D_1)$ and $\mathcal I(D_2)$ of Cartier divisors $D_1$ and $D_2$  of the statement of \cite[Theorem 3.4]{C6}).

\begin{proof} We apply the construction of Section \ref{SecFrame} with $\nu_{E_1}=\mu$. Let $\pi_1:\RR^d\rightarrow \RR$ be the projection onto the first factor. By the definition of $\gamma_{\mu}(\mathcal I(i))$ for $i=1,2$, and since for $c$ sufficiently large, $\gamma_{\mu}(\mathcal I(i))$ is in the  compact set $\pi_1(\Delta_c(\mathcal I(i))$, $\pi_1^{-1}(\gamma_{\mu}(\mathcal I(i))\cap \Delta_c(\mathcal I(i))\ne \emptyset$ and $\pi_1^{-1}(a)\cap \Delta_c(\mathcal I(i))=\emptyset$ if $a<\gamma_{\mu}(\mathcal I(i))$.
Since $\mathcal I(1)_i\subset \mathcal I(2)_i$ for all $i$, we have that $\Delta_c(\mathcal I(1))\subset \Delta_c(\mathcal I(2))$. Now ${\rm Vol}(\Delta_c(\mathcal I(1))>0$ for $c$ sufficiently large. Since we assume $e_R(\mathcal I(1))=e_R(\mathcal I(2))$, we have that ${\rm Vol}(\Delta_c(\mathcal I(1))={\rm Vol}(\Delta_c(\mathcal I(2))$ by (\ref{eqMF5}) and so $\Delta_c(\mathcal I(1))=\Delta_c(\mathcal I(2))$ by \cite[Lemma 3.2]{C6}. Thus $\gamma_{\mu}(\mathcal I(1))=\gamma_{\mu}(\mathcal I(2))$.
\end{proof}

%The last half of the proof of \cite[Theorem 3.4]{C6}, which we reproduce below, now proves 
%\cite[Theorem 3.4]{C6}.

%\begin{Corollary}(\cite[Theorem 3.4]{C6})\label{TheoremAR1}
%Let $D_1, D_2$ be real effective divisors on $X$ with exceptional support, such that
%$D_1\le D_2$ and $e_R(\mathcal I_1,R)=e_R(\mathcal I_2,R)$, where $\mathcal I_1=\{I(mD_1)\}$ and $\mathcal I_2=\{I(mD_2)\}$. Then
%$$
%\Gamma(X,\mathcal O_X(-mD_1))=\Gamma(X,\mathcal O_X(-mD_2))
%$$
%for all $m\in\NN$.
%\end{Corollary} 

%\begin{proof}  
%We have that both $e_R(\mathcal I(D_1)$ and $e_R(\mathcal I(D_2))$ are positive by \cite[Proposition 2.1]{C6}.

%By Theorem \ref{Theorem2}, 
%$$
%\gamma_{E_i}(D_1)=\gamma_{E_i}(D_2)
%$$
%for $1\le i\le r$. We  obtain that
%$$
%-\sum_{i=1}^r\gamma_{E_i}(D_2)E_i=-\sum_{i=1}^r\gamma_{E_i}(D_1)E_i.
%$$
%By Lemma \ref{LemmaAR1}, for all $m\ge 0$,
%$$
%\begin{array}{lll}
%I(mD_1)&=&\Gamma(X,\mathcal O_X(-\lceil \sum m\gamma_{E_i}(D_1)E_i\rceil))\\
%&=&\Gamma(X,\mathcal O_X(-\lceil \sum m\gamma_{E_i}(D_2)E_i\rceil))\\
%&=&I(mD_2).\end{array}
%$$
%\end{proof}

\begin{Corollary}\label{CorDiv} Let $R$ be a normal excellent local domain, $\mathcal I=\{I_m\}$ be an $m_R$-filtration and $\mathcal I(D)$ be a real divisorial $m_R$-filtration. Suppose that
 $I(mD)\subset I_m$ for all $m$ and $e_R(\mathcal I)=e_R(\mathcal I(D))$. Then $\mathcal I=\mathcal I(D)$.
\end{Corollary}

\begin{proof} The ring $R$ is analytically irreducible since $R$ is normal and excellent.  Let the pair $X\rightarrow \mbox{Spec}(R)$ and $D=\sum_{i=1}^r a_iE_i$ be a representation of $\mathcal I(D)$. We have that $\gamma_{\nu_{E_i}}(\mathcal I)=\gamma_{E_i}(D)$ for $1\le i\le r$ by Theorem \ref{Theorem2}.
We have that $I(mD)=\cap_{i=1}^rI(\nu_{E_i})_{\lceil m\gamma_{E_i}(D)\rceil}\subset I_m$ for all $m$ by assumption. Suppose that $f\in I_m$. Then
$$
\nu_{E_i}(f)\ge \tau_{\nu_{E_i},m}(\mathcal I)\ge m\gamma_{\nu_{E_i}}(\mathcal I)=m\gamma_{E_i}(D)
$$
for $1\le i\le r$. Thus $\nu_{E_i}(f)\ge \lceil m\gamma_{E_i}(D)\rceil$ for all $i$, and so
$f\in \cap_{i=1}^rI(\nu_{E_i})_{\lceil m\gamma_{E_i}(D)\rceil}=I(mD)$.
\end{proof}

\begin{Corollary}\label{CorDiv2} Let $R$ be an excellent local domain,  $\mathcal I(D)$ be a real divisorial $m_R$-filtration and $\mathcal I$ be an arbitrary  $m_R$-filtration. Suppose that
 $I(nD)\subset I_n$ for all $n$ and $e_R(\mathcal I(D))=e_R(\mathcal I)$. Then $\mathcal I=\mathcal I(D)$.
\end{Corollary}

\begin{proof} If $R$ is normal, the corollary is immediate from Corollary \ref{CorDiv}, so we may assume that $R$ is not normal. We use the notation of Subsection \ref{Not}. Let $S$ be the normalization of $R$ and let  $m_1,\ldots,m_t$ be the maximal ideals of $S$. Let $X\rightarrow \mbox{Spec}(R)$ and $D=\sum a_{i,j}E_{i,j}$ be a representation of $D$. Let  $X_i=X\otimes_SS_{m_i}$ for $1\le i\le t$. We have that $D=\sum_{i=1}^tD(i)$ where $D(i)=\sum_j a_{i,j}E_{i,j}$. Let  
$J(nD)=\Gamma(X,\mathcal O_X(-nD))$, so that 
$\mathcal I(D)=\{I(nD)\}$ where $I(nD)=J(nD)\cap R$. Further, we have real divisorial $m_i$-filtrations $\mathcal J(D(i))=\{J(nD_i)\}$ on $S_{m_i}$ which are defined by $J(nD(i))=\Gamma(X_i,\mathcal O_{X_i}(-nD(i)))=\Gamma(X,\mathcal O_X(-nD))S_{m_i}$.

Let $\mathcal IS_{m_i}$ be the $m_i$-filtration $\mathcal IS_{m_i}=\{I_nS_{m_i}\}_{n\ge 0}$. Then we have that
$$
S/I_nS\cong \bigoplus_{i=1}^t (S_{m_i}/I_nS_{m_i})
$$
for all $n$ and  so
$$
\ell_R(S/I_nS)=\sum_{i=1}^t[S/m_i:R/m_R]\ell_{S_{m_i}}(S_{m_i}/I_nS_{m_i}).
$$
Now the proof of \cite[Lemma 2.2]{C6} extends to this situation to show that
$$
\lim_{n\rightarrow \infty}\frac{\ell_R(R/I_n)}{n^d}=\lim_{n\rightarrow \infty}\frac{\ell_R(S/I_nS)}{n^d}
$$
from which we deduce that
\begin{equation}\label{eqNT1}
e_R(\mathcal I)=\sum_{i=1}^t[S/m_i:R/m_R]e_{S_{m_i}}(\mathcal IS_{m_i}).
\end{equation}
Similarly, 
\begin{equation}\label{eqNT2}
e_R(\mathcal I(D))=\sum_{i=1}^t[S/m_i:R/m_R]e_{S_{m_i}}(\mathcal I(D)S_{m_i}).
\end{equation}

Let $0\ne x$ be in the conductor of $S/R$. Then  $xJ(nD)\subset I(nD)$ for all $n$.
Let 
$$
A:=R[\mathcal I(D)]=\sum_{n\ge 0}I(nD)t^n,
$$
 $$
 B:=\sum_{n\ge 0}J(nD)t^n,
 $$
 and for $a\in \ZZ_{>0}$, let 
 $$
{}^aA:=R[\mathcal I(D)_a],
$$
where $\mathcal I(D)_a$ is the $a$-th truncation of $\mathcal I(D)$, that is, ${}^aA$ is the sub $R$-algebra generated by $I(nD)$ such that $n\le a$ and let
 ${}^aB$ be the sub $S$-algebra of $B$ generated by $ J(nD)$ such that $n\le a$.
 
  We have $xB\subset A$ and $xB_a\subset A_a$.
Suppose that $f\in J(mD)t^m$. Then $f\in ({}^mB)_m$ and $f^n\in ({}^mB)_{mn}$ 
 for all $n$ so that  $xf^n\in ({}^mA)_{mn}$ for all $n$ so that  ${}^mA[f]\subset \frac{1}{x}({}^mA)$ which is a finitely generated ${}^mA$-module, so $f$ is integral over the Noetherian ring ${}^mA$, and therefore  $f$ is integral over $A$. Thus $B$ is integral over $A$ and so $B$ is integral over 
$C:=\sum_{n\ge 0}I(nD)St^n$, and thus
$B_{m_i}=S_{m_i}[\mathcal J(D(i))]$ is integral over $C_{m_i}=S_{m_i}[\mathcal I(D)S_{m_i}]$ for $1\le i\le t$. We then have that
\begin{equation}\label{eqNT3}
e_{S_{m_i}}(\mathcal I(D)S_{m_i})=e_{S_{m_i}}(\mathcal J(D(i))
\end{equation}
for $1\le i\le t$ by \cite[Theorem 1.4]{C6}.

Let $G=\sum_{n\ge 0}L_nt^n$ be the integral closure of $F=\sum_{n\ge 0}I_nSt^n$ in $S[t]$. Then 
\begin{equation}\label{eqNT4}
e_{S_{m_i}}(\{L_nS_{m_i}\})=e_{S_{m_i}}(\mathcal IS_{m_i})
\end{equation}
for $1\le i\le t$ by \cite[Theorem 1.4]{C6}.  Now $I(nD)S_{m_i}\subset I_nS_{m_i}$ for all $i$, so that 
\begin{equation}\label{eqNT7}
J(nD(i))\subset L_nS_{m_i}
\end{equation}
 for all $i$ and so 
\begin{equation}\label{eqNT5}
e_{S_{m_i}}(\mathcal J(D(i)))\ge e_{S_{m_i}}(\{L_nS_{m_i}\})
\end{equation}
for $1\le i\le t$.
We have that
$$
\sum_{i=1}^t[S/m_i:R/m_R]e_{S_{m_i}}(\{L_nS_{m_i}\})
=\sum_{i=1}^t[S/m_i:R/m_R]e_{S_{m_i}}(\mathcal J(D(i)))
$$
by equations (\ref{eqNT3}), (\ref{eqNT4}), (\ref{eqNT1}) and (\ref{eqNT2}). Thus, by (\ref{eqNT5}), we have 
\begin{equation}\label{eqNT6}
e_{S_{m_i}}(\{L_nS_{m_i}\})=e_{S_{m_i}}(\mathcal J(D(i)))
\end{equation}
for $1\le i\le t$. By Corollary \ref{CorDiv}, we then have that $L_nS_{m_i}=J(nD)S_{m_i}$ for $1\le i\le t$, and so
$L_n=J(nD)$ for all $n$. Thus 
$$
I(nD)=R\cap J(nD)\subset I_n\subset L_n\cap R=J(nD)\cap R=I(nD)
$$
for all $n$.

%Since $I(mD_2)\subset I(mD_1)$ for all $m\in \NN$, we have that 
%$$
%xJ(mD_2)\subset J(mD_2)\cap R=I(mD_2)\subset I(mD_1)\subset J(mD_1)
%$$ 
%and localizing at $m_i$, we have that $xJ(mD_2(i))\subset J(mD_1(i))$ for all $i$ and $m\in \NN$. Suppose that $\gamma_{E_j}(D_1(i))>\gamma_{E_j}(D_2(i))$ for some $i$ and $j$. Let $\delta\in \RR$ be such that
%$$
%0<\delta<\gamma_{E_j}(D_1(i))-\gamma_{E_j}(D_2(i)).
%$$
 %There exists $m\in \ZZ_{>0}$ and $f\in J(mD_2(i))$ such that 
%$$
%0\le \frac{\nu_{E_j}(f)}{m}-\gamma_{E_j}(D_2(i))<
%\frac{\delta}{2}\mbox{ 
%and }0\le\frac{\nu_{E_j}(x)}{m}<\frac{\delta}{2}.
%$$
%Thus 
%$$
%0\le\frac{\nu_{E_j}(xf)}{m}-\gamma_{E_j}(D_2(i))<\delta
%$$
%so that 
%$$
%\frac{\nu_{E_j}(xf)}{m}<\gamma_{E_j}(D_1(i)).
%$$
%Thus $xf\not\in J(mD_1(i))$. By this contradiction, we have that  $\gamma_{E_j}(D_1(i))\le \gamma_{E_j}(D_2(i))$ for all $j$, and so
%$$
%J(mD_2(i))=J(m(\sum_j \gamma_{E_j}(D_2(i))E_j)\subset J(m(\sum_j \gamma_{E_j}(D_1(i))E_j)=J(mD_1(i))
%$$
%for all $m$ and $i$.

%Now
%$$
%e_R(\mathcal I(D_j))=\sum_{i=1}^ta_ie_{S_{m_i}}(\{J(mD_j(i)\})
%$$
%for $j=1,2$ where
%$a_i=[S/m_i:R/m_R]$ for $1\le i\le t$ by Lemma \ref{LemmaR1} and (\ref{eqR15}).
%We have that $e_{S_{m_i}}(J(mD_2(i)))\ge e_{S_{m_i}}(J(mD_1(i))$ for all $i$ since $J(mD_2(i))\subset J(mD_1(i))$ for all $m$.  Thus $e_{S_{m_i}}(J(mD_2(i)))= e_{S_{m_i}}(J(mD_1(i))$ for all $i$. Thus $J(mD_2(i))= J(mD_1(i))$ for $i=1,2$ and all $m$ by Corollary \ref{CorDiv}.
 %Thus $J(mD_1)=J(mD_2)$ for all $m\in \NN$   and thus
 %$I(mD_1)=I(mD_2)$ for all $m\in \NN$ since $I(mD_1)=J(mD_1)\cap R$ and $I(mD_2)=J(mD_2)\cap R$ for all $m$.
 
\end{proof}

Corollary \ref{CorDiv} is proven when $R$ is an excellent local domain and 
$D_1$ and $D_2$ are Cartier divisors in addition to $\mathcal I(D_1)$ and $\mathcal I(D_2)$ being integral divisorial $m_R$-filtrations
 in \cite[Theorem 3.5]{C6}.

\begin{Remark}\label{Rescale} Suppose that $R$ is an analytically irreducible excellent local ring, $\mathcal I=\{I_n\}$ is an $m_R$-filtration and $l\in \ZZ_{+}$. Let $\mathcal J_l$ be the $m_R$-filtration $\mathcal J_l=\{I_{ln}\}$. Then $\Delta(\mathcal J_l)=l\Delta(\mathcal I)$.
\end{Remark}

The following proposition will be used in our study of mixed multiplicities.

\begin{Proposition}\label{Prop11} Suppose that $R$ is a normal excellent local domain and that  $\gamma_1,\ldots,\gamma_r,\xi\in \RR_{\ge 0}$ are   such that $\gamma_1+\cdots+\gamma_r>0$ and $\xi>0$. Consider the $m_R$-filtrations  $\mathcal A=\mathcal I(\sum_{i=1}^r\gamma_iE_i)$ and $\mathcal B =\mathcal I(\sum_{i=1}^r\xi\gamma_iE_i)$. Then 
$$
\xi\Delta(\mathcal A)=\Delta(\mathcal B).
$$
\end{Proposition}

\begin{proof} 
We have that  $\mathcal A=\{A_n\}$ and $\mathcal B=\{B_n\}$ where
$$
A_n=\Gamma(X,\mathcal O_X(-\lceil \sum_{i=1}^rn\gamma_iE_i\rceil))\mbox{ and }
B_n=\Gamma(X,\mathcal O_X(-\lceil \sum_{i=1}^rn\xi\gamma_iE_i\rceil)).
$$

It suffices to show that
for all $\tau\in\RR$ sufficiently large, we have that
$$
\xi\left(\Delta(\mathcal A)\cap H^-_{\tau}\right)=\Delta(\mathcal B)\cap H^-_{\tau\xi}.
$$
The half space $H_c^-$ is defined in (\ref{MF3}).

Let $\mathcal C=\{C_n\}$ be the $m_R$-filtration defined by
$$
 C_n=\Gamma(X,\mathcal O_X(-\lceil \sum_{i=1}^r(n\gamma_i+1)E_i\rceil)).
 $$
 Let $E=E_1+\cdots+E_r$. 
 
 We now show that $\Delta(\mathcal C)=\Delta(\mathcal A)$. Let $0\ne f\in m_R$. Then $fA_n\subset C_n$ for all $n$. The elements of the form $\frac{\nu(g)}{m}=(\frac{a_1}{m},\ldots,\frac{a_d}{m})$ with $g\in A_m$ and $m>0$ are dense in $\Delta(\mathcal A)$. Since $\Delta(\mathcal C)\subset\Delta(\mathcal A)$ is a closed set, it suffices to show that given $\epsilon>0$, there exists $n\in \ZZ_{>0}$ and $h\in C_n$ such that $|\frac{\nu(h)}{n}-\frac{\nu(g)}{m}|<\epsilon$. For all $t\in \ZZ_{>0}$, $(\nu(g^t),tm)\in A_{tm}$ and $(\nu(fg^t),tm)\in C_{tm}$.
 Thus $\frac{\nu(f)}{tm}+\frac{\nu(g)}{m}\in \Delta(\mathcal C)$, with
 $$
 |(\frac{\nu(f)}{tm}+\frac{\nu(g)}{m})-\frac{\nu(g)}{m}|=\frac{1}{tm}|\nu(f)|<\epsilon
 $$
 for $t\gg 0$. Thus $\Delta(\mathcal C)=\Delta(\mathcal A)$.

 There exists $\tau_0\in \RR_{>0}$ such that 
 \begin{equation}\label{eqT4}
 \Delta(R)\cap H_{\tau_0}^+=\Delta(\mathcal A)\cap H_{\tau_0}^+=\Delta(\mathcal C)\cap H_{\tau_0}^+
 \end{equation}
 and
 \begin{equation}\label{eqT5}
 \Delta(R)\cap H_{\xi\tau_0}^+=\Delta(\mathcal B)\cap H_{\xi\tau_0}^+
 \end{equation}
 by Lemma \ref{LemmaC}.

 Suppose that $\tau>\tau_0$. Choose $\delta\in \RR_{>0}$ such that $\tau-\delta>\tau_0$.
 Let 
 $$
 \beta=\max\{|y|\mid y\in \Delta_{\tau}(\mathcal A)\}.
 $$
   The compact convex set $\Delta_{\tau}(\mathcal A)$ is defined in (\ref{eqMF6}). The numbers $\tau$, $\delta$ and $\beta$ will be fixed throughout the proof. 
 
 Given $\alpha\in \RR_{>0}$, there exist $p_0,q_0\in \ZZ_{>0}$ such that 
 \begin{equation}\label{eqT1}
 -\frac{\alpha}{q_0}<\frac{p_0}{q_0}-\xi\le 0
 \end{equation} 
 by Lemma \ref{LemmaHW}, so that
 \begin{equation}\label{eqT3}
 p_0\le \xi q_0<p_0+\alpha.
 \end{equation}
 Let $m$ be a positive integer, and suppose that $\alpha$ is sufficiently small that 
 $$
 \alpha<\frac{1}{m\max\{\gamma_i\}}.
 $$
 Set $p=m p_0$ and $q=m q_0$.
 Then
 $$
 \gamma_i p\le \xi \gamma_i q<\gamma_ip+\alpha\gamma_im
 $$
 for all $i$, so that
 $$
 \lceil \gamma_ip\rceil \le \lceil q\gamma_i\xi\rceil \le \lceil p\gamma_i\rceil +1
 $$
 for all $i$ and so
 $$
  -(\lceil p\gamma_i\rceil +1)\le -\lceil q\gamma_i\xi\rceil\le -\lceil \gamma_ip\rceil
  $$
  implying
  $$
  -\lceil\sum_{i=1}^r(p\gamma_i+1)E_i\rceil
     \le -\lceil \sum_{i=1}^r q\gamma_i\xi E_i\rceil\le -\lceil \sum_{i=1}^r \gamma_ipE_i \rceil
  $$
  giving us that
  \begin{equation}\label{eqT2}
  C_p\subset B_q\subset A_p.
  \end{equation}
  
  We will now show that $\Delta_{\xi\tau}(\mathcal B)=  \xi\Delta_{\tau}(\mathcal A)$.
  
 First suppose that $v=(v_1,\ldots,v_d)\in \Delta_{\xi\tau}(\mathcal B)$ and that $v_1+\cdots+v_d\ge \xi\tau-\xi\delta$.  Then $v\in \Delta(\mathcal B)\cap H^+_{\xi\tau-\xi\delta}
 =\Delta(R)\cap H^+_{\xi\tau-\xi\delta}$. Then since $\Delta(R)$ is a cone with vertex at the origin, 
 $$
 \frac{1}{\xi}v\in \left(\frac{1}{\xi}\Delta(R)\right)\cap H^+_{\tau-\delta}=\Delta(R)\cap H^+_{\tau-\delta}=\Delta(\mathcal A)\cap H^+_{\tau-\delta},
 $$
 by (\ref{eqT4}) and so $v=\xi u$ for some $u\in \Delta_{\tau}(\mathcal A)$.

 Now suppose that $v=(v_1,\ldots,v_d)\in \Delta_{\xi\tau}(\mathcal B)$ and $v_1+\cdots+v_d<\xi\tau-\xi\delta$.  Since the elements of $\Delta(
 \mathcal B)$ of the form $\frac{\nu(f)}{m}$ with  $m\in \ZZ_{>0}$ and $ f\in B_m$ are dense in $\Delta(\mathcal B)$, we may suppose that $v$ has this form. Let $\epsilon>0$. we will find $u\in \xi\Delta_{\tau}(\mathcal A)$ such that $|v-u|<\epsilon$. Since $\xi\Delta_{\tau}(\mathcal A)$ is closed, this will show that $v\in \xi\Delta_{\tau}(\mathcal A)$.
 
 Choose $\alpha\in \RR_{>0}$ such that  
 $$
 \alpha<\min\{\frac{1}{m\max{\gamma_i}},\frac{\delta}{\tau},    \frac{\epsilon}{\beta}\}.
 $$
 Choose $p_0,q_0$ which satisfy (\ref{eqT1}). Thus
 $$
 0\le\xi-\frac{p_0}{q_0}<\frac{\alpha}{q_0}.
 $$
 Write $\xi=\frac{p_0}{q_0}+\lambda$ with $0\le \lambda<\frac{\alpha}{q_0}$.
 
Set $p=mp_0$ and $q=mq_0$, so that 
 $v=\left(\frac{a_1}{q},\ldots,\frac{a_d}{q}\right)= \frac{\nu(f^{q_0})}{q}$ where  $f^{q_0}\in B_q$. 
 Set
 $$
 u=\xi \left(\frac{a_1}{p},\ldots,\frac{a_d}{p}\right)\mbox{ and }
 w=-\lambda \left(\frac{a_1}{p},\ldots,\frac{a_d}{p}\right)
 $$ 
 so that $v=\frac{p}{q} \left(\frac{a_1}{p},\ldots,\frac{a_d}{p}\right)    =u+w$. 
 Now $\left(\frac{a_1}{p},\ldots,\frac{a_d}{p}\right)\in \Delta(\mathcal A)$ by (\ref{eqT2}).
Since by (\ref{eqT3}), $\frac{q}{p}<\frac{1}{\xi}+\frac{\alpha}{\xi p_0}$ and $\alpha<\frac{\delta}{\tau}$, we have that
 $$
  \begin{array}{lll}
 \frac{a_1}{p}+\cdots+\frac{a_d}{p}&=&\frac{q}{p}\left(\frac{a_1}{q}+\cdots+\frac{a_d}{q} \right)
 < \frac{q}{p}(\xi\tau-\xi\delta)\\
 &<& \left(\frac{1}{\xi}+\frac{\alpha}{\xi p_0}\right)(\xi\tau-\xi\delta)=\tau-\delta+\frac{\alpha\tau}{p_0}-\frac{\alpha\delta}{p_0}<\tau-\delta+\frac{\delta\tau}{\tau p_0}-\frac{\alpha\delta}{p_0}\\
 &=& \tau-\delta\left(1-\frac{1}{p_0}\right)-\frac{\alpha\delta}{p_0}<\tau.
 \end{array}
 $$
 Thus $\left(\frac{a_1}{p},\ldots,\frac{a_d}{p}\right)\in \Delta_{\tau}(\mathcal A)$  and 
 $|v-u|=|w|\le|\lambda|\beta<\frac{\alpha\beta}{q_0}<\epsilon$. Since we can make $\epsilon$ arbitrarily small, we have that $v\in \xi\Delta_{\tau}(\mathcal A)$.

 We will now show that $\xi\Delta_{\tau}(\mathcal A)\subset \Delta_{\xi\tau}(\mathcal B)$.
  
 First suppose that $u=(u_1,\ldots,u_d)\in \xi\Delta_{\tau}(\mathcal A)$ and that $u_1+\cdots+u_d\ge \xi\tau-\xi\delta$.  Then 
 $\frac{1}{\xi}u\in \Delta_{\tau}(\mathcal A)$ with 
 $\frac{1}{\xi}u_1+\cdots+\frac{1}{\xi}u_d       \ge \tau-\delta$, so that 
 $$
 \frac{1}{\xi}u\in \Delta(\mathcal A)\cap H^+_{\tau-\delta}=\Delta(R)\cap H^+_{\tau-\delta}
 $$
 by (\ref{eqT4}) so
 $$
 u\in \xi\left(\Delta(R)\cap H^+_{\tau-\delta}\right)=\Delta(R)\cap H^+_{\tau\xi-\delta\xi}=\Delta(\mathcal B)\cap H^+_{\tau\xi-\delta\xi}
 $$
 by (\ref{eqT5}). 
 Since $u_1+\cdots+u_d<\tau\xi$ we have that $u\in \Delta_{\xi\tau}(\mathcal B)$.

 Now suppose that $u=(u_1,\ldots,u_d)\in \xi\Delta_{\tau}(\mathcal A)$ and that $u_1+\cdots+u_d<\xi\tau-\xi\delta$.  Since the elements of $\Delta(\mathcal C)$ of the form $\frac{\nu(f)}{m}$ with  $m\in \ZZ_{>0}$ and $ f\in C_m$ are dense in $\Delta(\mathcal C)=\Delta(\mathcal A)$, we may suppose that $u$ is $\xi$ times an element of  this form. Let $\epsilon>0$. we will find $v\in \Delta_{\tau\xi}(\mathcal B)$ such that $|u-v|<\epsilon$. Since $\Delta_{\tau\xi}(\mathcal B)$ is closed, this will show that $u\in \Delta_{\tau\xi}(\mathcal B)$.
 
 Choose $\alpha\in \RR_{>0}$ such that 
  $$
 \alpha<\max\{\frac{1}{m\max\{\gamma_i\}},\frac{\epsilon}{\beta}\}.
 $$
 Choose $p_0,q_0$ which satisfy (\ref{eqT1}). Thus
 $$
 0\le\xi-\frac{p_0}{q_0}<\frac{\alpha}{q_0}.
 $$
 Write $\xi=\frac{p_0}{q_0}+\lambda$ with $0\le \lambda<\frac{\alpha}{q_0}$.
 
Set $p=mp_0$ and $q=mq_0$, so that
 $u=\xi\left(\frac{a_1}{p},\ldots,\frac{a_d}{p}\right)$ with $\left(\frac{a_1}{p},\ldots,\frac{a_d}{p}\right)= \frac{\nu(f^{p_0})}{p}$ where $f^{p_0}\in C_p$ and $\frac{a_1}{p}+\cdots+\frac{a_d}{p}
 <\tau-\delta$.  
 Set $v=\frac{p}{q}\left(\frac{a_1}{p},\ldots,\frac{a_d}{p}\right) =\left(\frac{a_1}{q},\ldots,\frac{a_d}{q}\right) $ and $w=\lambda\left(\frac{a_1}{p},\ldots,\frac{a_d}{p}\right)$ so that $u=v+w$.  
 
  We have that $v\in \Delta(\mathcal B)$  by (\ref{eqT2}). We have that
 $$
 \frac{a_1}{q}+\cdots+\frac{a_d}{q}<\frac{p}{q}(\tau-\delta)< \xi(\tau-\delta)<\xi\tau
 $$
 so that $v\in \Delta_{\xi\tau}(\mathcal B)$.

  $|v-u|=|w|\le|\lambda|\beta<\frac{\alpha\beta}{q_0}<\epsilon$. Since we can make $\epsilon$ arbitrarily small, we have that $u\in \Delta_{\xi\tau}(\mathcal B)$.

\end{proof}

\section{Computation of Mixed Multiplicities of filtrations}\label{SecMMF} Let notation be as in Section \ref{SecFrame}, so that $R$ is a $d$-dimensional excellent local domain. We further assume that $R$ is analytically irreducible in this section.

Let $\mathcal I(1)=\{I(1)_i\}$, $\mathcal I(2)=\{I(2)_i\}$ be $m_R$-filtrations.
We now define some sub semigroups of $\NN^{d+1}$ which are associated to $\mathcal I(1)$ and $\mathcal I(2)$. For $n_1,n_2\in \NN$, define
$$
\Gamma(n_1,n_2)=\{(\nu(f),i)\mid f\in I(1)_{in_1}I(2)_{in_2}\}.
$$
We define an associated closed convex subset $\Delta(n_1,n_2)$ of $\RR^d$ as follows. 
Let $\Sigma(n_1,n_2)$ be the closed convex cone with vertex at the origin generated by $\Gamma(n_1,n_2)$ and let $\Delta(n_1,n_2)=\Sigma(n_1,n_2)\cap (\RR^d\times\{1\})$.
The set  $\Delta(n_1,n_2)$ is the closure of the set 
$$
\left\{\left(\frac{a_1}{i},\ldots,\frac{a_d}{i}\right)\mid (a_1,\ldots,a_d,i)\in \Gamma(n_1,n_2)\mbox{ and }i>0\right\}
$$
in the Euclidean topology of $\RR^d$. We have that  $\Gamma(R)=\Gamma(0,0)$ and $\Delta(R)=\Delta(0,0)$ as defined in Section \ref{SecMF}.

\begin{Lemma}\label{Lemma1} For all $m_1,m_2,n_1,n_2\in \NN$, we have that
$$
\Delta(m_1,m_2)+\Delta(n_2,n_2)\subset \Delta(m_1+n_1,m_2+n_2).
$$
In particular,
$$
n_1\Delta(1,0)+n_2\Delta(0,1)\subset \Delta(n_1,n_2).
$$
\end{Lemma}

\begin{proof} The set of points
$$
\left\{\left(\frac{a_1}{i},\ldots,\frac{a_d}{i}\right)\mid (a_1,\ldots,a_d,i)\in\Gamma(m_1,m_2)\mbox{ and }i>0\right\}
$$
is dense in the closed set $\Delta(m_1,m_2)$. Thus it suffices to show that if 
$(a_1,\ldots,a_d,i)\in \Gamma(m_1,m_2)$ and $(b_1,\ldots,b_d,j)\in \Gamma(n_1,n_2)$ with $i,j>0$,  then 
$$
\left(\frac{a_1}{i}+\frac{b_1}{j},\ldots,\frac{a_d}{i}+\frac{b_d}{j}\right)\in \Delta(m_1+n_1,m_2+n_2).
$$
 With this assumption, there exists $f\in I(1)_{im_1}I(2)_{im_2}$ such that $\nu(f)=(a_1,\ldots,a_d)$ and
there exists $g\in I(1)_{jn_1}I(2)_{jn2}$ such that $\nu(g)=(b_1,\ldots,b_d)$. We have that
$f^jg^i\in I(1)_{ij(m_1+n_1)}I(2)_{ij(m_2+n_2)}$ so 
$$
(\nu(f^jg^i),ij)=(ja_1+ib_1,\ldots,ja_d+ib_d,ij)\in \Gamma(m_1+n_1,m_2+n_2).
$$
Thus 
$$
\left(\frac{ja_1+ib_1}{ij},\ldots,\frac{ja_d+ib_d}{ij}\right)=\left(\frac{a_1}{i}+\frac{b_1}{j},\ldots,\frac{a_d}{i}+\frac{b_d}{j}\right)\in \Delta(m_1+n_1,m_2+n_2).
$$

\end{proof}

\begin{Lemma}\label{LemmaA} There exists $\overline\lambda\in \ZZ_{>0}$ such that for all $n_1,n_2\in \NN$,
$$
\begin{array}{l}
\Delta(n_1,n_2)\cap \{(x_1,\ldots,x_d)\in \RR^d \mid x_1\ge (n_1+n_2)\overline \lambda\}\\
=\Delta(R)\cap \{(x_1,\ldots,x_d)\in \RR^d \mid x_1\ge (n_1+n_2)\overline \lambda\}.
\end{array}
$$
\end{Lemma}

\begin{proof} Let $\nu_1,\ldots,\nu_t$ be the Rees valuations of $m_R$. Since $R$ is analytically irreducible,  the topologies of the $\nu_j$ on $R$ are linearly equivalent to the topology of $\mu$ on $R$ by Rees's Izumi Theorem \cite{R3}. Let $\overline \nu_{m_R}$ be the reduced order. By the Rees valuation theorem (recalled in \cite{R3}), for $x\in R$, 
$$
\overline \nu_{m_R}(x)=\min_j\left\{\frac{\nu_j(x)}{\nu_j(m_R)}\right\}
$$
so the topology of $\overline\nu_{m_R}$ is linearly equivalent to the topology induced by each $\nu_j$. Further, $\overline\nu_{m_R}$ is linearly equivalent to the $m_R$-topology by \cite{R2} since $R$ is analytically irreducible. Thus there exists $\alpha\in \ZZ_{>0}$ such that
$I(u)_{m\alpha}\subset m_R^m$ for all $m\in \NN$. Since $I(1)_1$ and $I(2)_1$ are $m_R$-primary,  there exists $c\in \ZZ_{>0}$ such that $m_R^c\subset I(1)_1$ and $m_R^c\subset I(2)_1$, so that $m_R^{c(n_1+n_2)}\subset I(1)_{n_1}I(2)_{n_2}$ for all $n_1,n_2\in \ZZ_{\ge 0}$. Let $\overline\lambda=c\alpha$. Then 
\begin{equation}\label{eq30}
I(\mu)_{(n_1+n_2)\overline\lambda}\subset m_R^{c(n_1+n_2)}\subset I(1)_{n_1}I(2)_{n_2}
\end{equation}
for all $n_1,n_2\in \ZZ_{\ge 0}$. 

Suppose $(a_1,\ldots,a_d,m)\in \Gamma(R)$ is such that $m>0$ and $\frac{a_1}{m}\ge (n_1+n_2)\overline\lambda$. Then there exists $f\in R$ such that $\nu(f)=(a_1,\ldots,a_d)$. In particular, $\mu(f)=a_1$. Thus $\mu(f)\ge m(n_1+n_2)\overline\lambda$ so that $f\in I(a)_{mn_1}I(2)_{mn_2}$ by (\ref{eq30}). Thus $(a_1,\ldots,a_d,m)\in \Gamma(n_1,n_2)$ and so 
$$
\left(\frac{a_1}{m},\ldots,\frac{a_d}{m}\right)\in \Delta(n_1,n_2).
$$
\end{proof}

\begin{Lemma}\label{LemmaB} There exists $\lambda\in \ZZ_{>0}$ such that for all $n_1,n_2\in \NN$,
$$
\Delta(n_1,n_2)\cap H^+_{(n_1+n_2)\lambda}=\Delta(R)\cap H^+_{(n_1+n_2)\lambda}.
$$
\end{Lemma}

\begin{proof} Recall the definitions of $\Xi$, $\Xi_n$ and $\Delta(\Xi)$ in equations (\ref{eq21}), (\ref{eq22}) and (\ref{eq23}). The set $\Delta(\Xi)\subset \{1\}\times \RR^{d-1}$ is compact and convex as explained after (\ref{eq22}). Thus there exists $b\in \ZZ_{>0}$ such that $\Delta(\Xi)\subset \{1\}\times [0,b]^{d-1}$. Suppose that $f\in R$ and $\mu(f)\le \delta$ for some $\delta$. Let $\nu(f)=(a_1=\mu(f),a_2,\ldots,a_d)$. Then $\nu(f)\in \Xi_{a_1}$ by (\ref{eq31}) which implies 
$$
\left(1,\frac{a_2}{a_1},\ldots,\frac{a_d}{a_1}\right)\in \Delta(\Xi)
$$
so
\begin{equation}\label{eq4}
a_i\le\delta b \mbox{ for all }i.
\end{equation}
Choose $\lambda>\overline\lambda bd$ where $\overline\lambda$ is the constant of Lemma \ref{LemmaA}. Suppose $(a_1,\ldots,a_d,m)\in \Gamma(R)$ is such that $m>0$ and 
\begin{equation}\label{eq5}
\frac{a_1}{m}+\cdots+\frac{a_d}{m}\ge (n_1+n_2)\lambda.
\end{equation}
If 
$$
\frac{a_1}{m}>(n_1+n_2)\overline\lambda\mbox{ then }
\left(\frac{a_1}{m},\ldots,\frac{a_d}{m}\right)\in \Delta(n_1,n_2)
$$
by Lemma \ref{LemmaA}. Suppose $\frac{a_1}{m}\le (n_1+n_2)\overline\lambda$. Then $a_1\le m(n_1+n_2)\overline\lambda$ so that $a_i\le m(n_1+n_2)\overline \lambda b$ by (\ref{eq4}). Thus
$$
\frac{a_1}{m}+\cdots+\frac{a_d}{m}\le (n_1+n_2)bd\overline\lambda<(n_1+n_2)\lambda,
$$
a contradiction to our assumption (\ref{eq5}). Thus the conclusions of our lemma hold. 
\end{proof}

Given $\Phi=(\alpha_1,\alpha_2,\phi)\in \RR^3_{>0}$, define 
$$
H^{+}_{\Phi,n_1,n_2}=\{(x_1,\ldots,x_d)\in \RR^d\mid x_1+\cdots+x_d>(\alpha_1n_1+\alpha_2n_2)\phi\}.
$$
Let $\lambda$ be the number defined in Lemma \ref{LemmaB}. If 
\begin{equation}\label{eq44}
\phi\ge\frac{\lambda}{\min\{\alpha_1,\alpha_2\}},
\end{equation}
 then for $n_1,n_2\in \ZZ_{\ge 0}$ with $n_1+n_2>0$, we have
$$
(\alpha_1n_1+\alpha_2n_2)\phi\ge (n_1+n_2)\lambda
$$ 
so
\begin{equation}\label{eq71}
\Delta(R)\cap H^+_{\Phi,n_1,n_2}=\Delta(n_1,n_2)\cap H^+_{\Phi,n_1,n_2}
=(n_1\Delta(1,0)+n_2\Delta(0,1))\cap H^+_{\Phi,n_1,n_2}.
\end{equation}
The second equality in (\ref{eq71}) is obtained as follows. Lemma \ref{LemmaB} implies 
that 
$$
\Delta(1,0)\cap H_{\Phi,1,0}^+=\Delta(R)\cap H_{\Phi,1,0}^+\mbox{ and }
\Delta(0,1)\cap H_{\Phi,0,1}^+=\Delta(R)\cap H_{\Phi,0,1}^+.
$$
Taking the Minkowski sum, we thus have
$$
(n_1\Delta(1,0)+n_2\Delta(0,1))\cap H^+_{\Phi,n_1,n_2}=\Delta(R)\cap H^+_{\Phi,n_1,n_2}.
$$
Set 
\begin{equation}\label{eq70}
\Delta_{\Phi}(n_1,n_2)=\Delta(n_1,n_2)\setminus \Delta(n_1,n_2)\cap H^+_{\Phi,n_1,n_2},
\end{equation}
$$
\tilde\Delta_{\Phi}(n_1,n_2)=\Delta(R)\setminus \Delta(R)\cap H^+_{\Phi,n_1,n_2}.
$$
These are compact, convex subsets of $\RR^d$. For all $n_1,n_2\in \NN$, we have that the Minkowski sum \begin{equation}\label{eq2}
n_1\Delta_{\Phi}(1,0)+n_2\Delta_{\Phi}(0,1)\subset \Delta_{\Phi}(n_1,n_2)\subset \tilde\Delta_{\Phi}(n_1,n_2).
\end{equation}
We now fix $n_1,n_2\in \NN$.
For $m\in \NN$, let
$$
\begin{array}{l}
\Gamma_{\Phi}(n_1,n_2)_m\\
=\{(\nu(f),m)=(a_1,\ldots,a_d,m)\mid f\in I(1)_{mn_1}I(2)_{mn_2}\mbox{ and }a_1+\cdots+a_d\le m(\alpha_1n_1+\alpha_2n_2)\phi\},
\end{array}
$$
$$
\begin{array}{l}
\tilde\Gamma_{\Phi}(n_1,n_2)_m\\
=\{(\nu(f),m)=(a_1,\ldots,a_d,m)\mid f\in R\mbox{ and }a_1+\cdots+a_d\le m(\alpha_1n_1+\alpha_2n_2)\phi\}.
\end{array}
$$
The semigroups $\Gamma_{\Phi}(n_1,n_2)=\cup_{m\in \NN}\Gamma_{\Phi}(n_1,n_2)_m$ and 
$\tilde \Gamma_{\Phi}(n_1,n_2)=\cup_{m\in \NN}\tilde \Gamma_{\Phi}(n_1,n_2)_m$
satisfy conditions (5) and (6) of \cite[Theorem 3.2]{C2} by the argument after Lemma \ref{LemmaC}. Thus the limits
$$
\lim_{m\rightarrow \infty}\frac{\#\tilde\Gamma_{\Phi}(n_1,n_2)_m}{m^d}={\rm Vol}(\tilde\Delta_{\Phi}(n_1,n_2))
$$
and
$$
\lim_{m\rightarrow \infty}\frac{\#\Gamma_{\Phi}(n_1,n_2)_m}{m^d}={\rm Vol}(\Delta_{\Phi}(n_1,n_2))
$$
exist by \cite[Theorem 3.2]{C2}. As in \cite[Theorem 5.6]{C3}, if 
$\phi\ge \frac{\lambda}{\min\{\alpha_1,\alpha_2\}}$, then 
\begin{equation}\label{eq11}
\lim_{m\rightarrow\infty}\frac{\ell_R(R/I(1)_{mn_1}I(2)_{mn_2})}{m^d}
=\delta[{\rm Vol}(\tilde\Delta_{\Phi}(n_1,n_2))-{\rm Vol}(\Delta_{\Phi}(n_1,n_2))]
\end{equation}
where 
\begin{equation}\label{eq46}
\delta=[\mathcal O_{\nu}/m_{\nu}:R/m_R].
\end{equation}

The function 
\begin{equation}\label{eq40}
f(n_1,n_2):= \lim_{m\rightarrow\infty}\frac{\ell_R(R/I(1)_{mn_1}I(2)_{mn_2})}{m^d}
\end{equation}
for  $n_1,n_2\in \NN$ of (\ref{M2}) and (\ref{eqV6})
is a homogeneous polynomial in $\RR[x]$ of degree $d$.

Since $\Delta_{\Phi}(1,0)$ and $\Delta_{\Phi}(0,1)$ are compact convex subsets of $\RR^d$, the function 
\begin{equation}\label{eq48}
g(n_1,n_2)={\rm Vol}(n_1\Delta_{\Phi}(1,0)+n_2\Delta_{\Phi}(0,1))
\end{equation}
is a homogeneous polynomial over $\RR$ of degree $d$ for all $n_1,n_2\in \RR_{\ge 0}$ by Theorem \ref{ConvMix}. We have that 
\begin{equation}\label{eq41}
\begin{array}{lll}
{\rm Vol}(\tilde\Delta_{\Phi}(n_1,n_2))&=&{\rm Vol}(\Delta(R)\cap H^{-}_{(n_1\alpha_1+n_2\alpha_2)\phi})\\
&=&{\rm Vol}\left((n_1\alpha_1+n_2\alpha_2)\phi(\Delta(R)\cap H_1^{-})\right)\\
&=& (n_1\alpha_1+n_2\alpha_2)^d\phi^d{\rm Vol}(\Delta(R)\cap H_1^{-})
\end{array}
\end{equation}
by (\ref{eq1}). Thus by (\ref{eq11}) and (\ref{eq41}),
\begin{equation}\label{eq42}
{\rm Vol}(\Delta_{\Phi}(n_1,n_2))=(n_1\alpha_1+n_2\alpha_2)^d\phi^d{\rm Vol}(\Delta(R)\cap H_1^{-})-\frac{1}{\delta}f(n_1,n_2)
\end{equation}
where $f(n_1,n_2)$ is the function of (\ref{eq40}). We have that
\begin{equation}\label{eq43}
{\rm Vol}(\Delta_{\Phi}(n_1,n_2))\ge g(n_1,n_2)
\end{equation}
for all $n_1,n_2\in \NN$ by (\ref{eq2}).

\begin{Theorem}\label{MinNew}  Suppose that $R$ is a local ring of dimension $d$ with $\dim N(\hat R)<d$ and $\mathcal I(1)$ and $\mathcal I(2)$ are $m_R$-filtrations of $R$. Suppose that there exist $a,b\in \ZZ_{>0}$ such that 
\begin{equation}\label{eqIC}
\overline{\sum_{n\ge 0}I(1)_{an}t^n}=\overline{\sum_{n\ge 0}I(2)_{bn}t^n}.
\end{equation}
Then $\mathcal I(1)$ and $\mathcal I(2)$ satisfy the Minkowski equality
$$
e_R(\mathcal I(1)\mathcal I(2))^{\frac{1}{d}}=e_R(\mathcal I(1))^{\frac{1}{d}}+e_R(\mathcal I(2))^{\frac{1}{d}}.
$$
\end{Theorem}

\begin{proof} Let
$$
P(n_1,n_2)=\lim_{m\rightarrow\infty}\frac{\ell_R(R/I(1)_{mn_1}I(2)_{mn_2})}{m^d}.
$$
Let $P_1,\ldots,P_s$ be the minimal primes of $\hat R$ such that $\dim \hat R/P_i=d$. 
Let $R_i=\hat R/P_i$ for $1\le i\le s$. The $R_i$ are analytically irreducible excellent local domains. 

By the proof of Theorem 4.7 of \cite{C2}  we have that
$$
P(n_1,n_2)=\sum_{k=1}^s\lim_{m\rightarrow\infty}\frac{\ell_{R_i}(R_i/I(1)_{mn_1}I(2)_{mn_2}R_i)}{m^d}.
$$

We first suppose that $\mathcal I(1)$ and $\mathcal I(2)$ are such that $R[\mathcal I(1)]=\overline{R[\mathcal I(1)]}$ and $R[\mathcal I(2)]=\overline{R[\mathcal I(2)]}$ are integrally closed. Then 
 $I(1)_{ma}=I(2)_{mb}$ for all $m\in \ZZ_{>0}$.

 Let $J_m=I(1)_{ma}=I(2)_{mb}$ for $m\in \NN$ and $\mathcal J=\{J_m\}$. Let 
 $$
 Q(n_1,n_2)=\lim_{m\rightarrow \infty}\frac{\ell_R(R/J_{mn_1}J_{mn_2})}{m^d}.
 $$
 For $1\le i\le s$, let
 $$
  Q_i(n_1,n_2)=\lim_{m\rightarrow \infty}\frac{\ell_{R_i}(R_i/J(i)_{mn_1}J(i)_{mn_2})}{m^d}
 $$ 
 where $J(i)_m=J_mR_i$. We have that
 $$
 Q(n_1,n_2)=\sum_{i=1}^sQ_i(n_1,n_2)
 $$
 by the proof of Theorem 4.7 of \cite{C2}.
 
 For each $i$, we apply the construction of Section \ref{SecMF}
  to the $m_{R_i}$-filtration $\{J(i)_m\}$ on $R_i$ and we apply the construction of this section (Section \ref{SecMMF})
   to the $m_{R_i}$-filtrations
 $\{J(i)_m\}$ and $\{J(i)_m\}$ on $R_i$. We use the notation of these sections.

 We have by Remark \ref{Rescale}
  that for all $n_1,n_2\in \NN$ with $n_1+n_2>0$ that 
 $$
 \begin{array}{l}
 (n_1+n_2)\Delta(\{J(i)_m\})\subset n_1\Delta(\{J(i)_m\})+n_2\Delta(\{J(i)_m\})=
\Delta(\{J(i)_{mn_1}\})+\Delta(\{J(i)_{mn_2}\})\\
 \subset\Delta(n_1,n_2)=\Delta(\{J(i)_{mn_1}J(i)_{mn_2}\})\subset \Delta(\{J(i)_{m(n_1+n_2)}\})=(n_1+n_2)\Delta(\{J(i)_m\})
 \end{array}
 $$
 so that $\Delta(n_1,n_2)= (n_1+n_2)\Delta(\{J(i)_m\})$. 
 
 Let $\Phi=(1,1,\phi)$ with $\phi$ sufficiently large.  Then
 \begin{equation}\label{eqT5*}
 \Delta_{\Phi}(n_1,n_2)=[(n_1+n_2)\Delta(\{J(i)_m\})]\cap H^-_{(n_1+n_2)\phi}
 =(n_1+n_2)[\Delta(\{J(i)_m\})\cap H^-_{\phi}].
 \end{equation}
 By (\ref{eq11}),
 $$
 Q_i(n_1,n_2)=\delta[{\rm Vol}(\tilde\Delta_{\Phi}(n_1,n_2))-{\rm Vol}(\Delta_{\Phi}(n_1,n_2))]
 $$
 and by (\ref{eq41}),
 $$
 {\rm Vol}(\tilde\Delta_{\Phi}(n_1,n_2))=(n_1+n_2)^d {\rm Vol}(\Delta(R_i)\cap H_{\phi}^-).
 $$
 By (\ref{eqT5*}),
 $$
 {\rm Vol}(\Delta_{\Phi}(n_1,n_2))=(n_1+n_2)^d{\rm Vol}(\Delta(\{J(i)_m\})\cap H_{\phi}^-),
 $$
 so that
 $$
 Q_i(n_1,n_2)=c_i(n_1+n_2)^d
 $$
 where 
 $$
 c_i=\delta[{\rm Vol}(\Delta(R_i)\cap H_{\phi}^-)- {\rm Vol}(\Delta(\{J(i)_m\})\cap H_{\phi}^-)].
 $$ 
 Thus letting $c=\sum_{i=1}^sc_i$, we have that $Q(n_1,n_2)=c(n_1+n_2)^d$.
 
 Now $P(am_1,bm_2)=Q(m_1,m_2)$ so
 $$
 P(n_1,n_2)=c(\frac{n_1}{a}+\frac{n_2}{b})^d,
 $$
 and substituting the values $(n_1,n_2)=(1,1)$, $(n_1,n_2)=(1,0)$ and $(n_1,n_2)=(0,1)$ we get that  
  $\mathcal I(1)$ and $\mathcal I(2)$  satisfy Minkowski's equality, establishing the theorem if  $R[\mathcal I(1)]$ and $R[\mathcal I(2)]$
  are integrally closed.  
  
 Now suppose that $\mathcal I(1)$ and $\mathcal I(2)$ are arbitrary $m_R$-filtrations satisfying (\ref{eqIC}).
 Define $m_R$-filtrations $\mathcal J(1)$ and $\mathcal J(2)$ by setting $\mathcal J(1)=\{J(1)_n\}$ and $\mathcal J(2)=\{J(2)_n\}$ where
 $$
 \sum_{n\ge 0}J(1)_nt^n=\overline{\sum_{n\ge 0}I(1)_nt^n}\mbox{ and }
  \sum_{n\ge 0}J(2)_nt^n=\overline{\sum_{n\ge 0}I(2)_nt^n}. 
 $$
 Now 
 $$
 \overline{\sum I(1)_{an}t^{an}}=\sum J(1)_{an}t^{an}\mbox{ and } \overline{\sum I(2)_{bn}t^{an}}=\sum J(2)_{bn}t^{an}
  $$
 so
 $$
 \sum_{n\ge 0}J(1)_{an}t^n=\overline{\sum_{n\ge 0}I(1)_{an}t^n}
 =\overline{\sum_{n\ge 0}I(1)_{bn}t^n}= \sum_{n\ge 0}J(2)_{bn}t^n.
   $$
 By the first part of the proof, $\mathcal J(1)$ and $\mathcal J(2)$ satisfy Minkowski's equality, and there is an expression
 $$
 \lim_{m\rightarrow\infty}\frac{\ell_R(R/J(1)_{mn_1}J(2)_{mn_2})}{m^d}=c(\frac{n_1}{a}+\frac{n_2}{b})^d.
 $$ 
 By Proposition \ref{BoundProp},
 %\cite[Proposition 5.10]{C3},
 $$
 P(n_1,n_2)=\lim_{m\rightarrow\infty}\frac{\ell_R(R/J(1)_{mn_1}J(2)_{mn_2})}{m^d}=c(\frac{n_1}{a}+\frac{n_2}{b})^d,
 $$
 and substituting the values $(n_1,n_2)=(1,1)$, $(n_1,n_2)=(1,0)$ and $(n_1,n_2)=(0,1)$ we get that  
  $\mathcal I(1)$ and $\mathcal I(2)$  satisfy Minkowski's equality.
 \end{proof}

\section{Equality in the Minkowski inequality for mixed multiplicities} \label{SecMinEQ}

Let $R$ be a $d$-dimensional analytically unramified local domain and $\mathcal  I(1)$, $\mathcal I(2)$ be $m_R$-filtrations. 

The polynomial $f(n_1,n_2)$ of (\ref{eq40}) has an expansion
$$
f(n_1,n_2)=\sum_{d_1+d_2=d}\frac{1}{d_1!d_2!}e_R(\mathcal I(1)^{[d_1]},\mathcal I(2)^{[d_2]})n_1^{d_1}n_2^{d_2}
$$ 
where $e_R(\mathcal I(1)^{[d_1]},\mathcal I(2)^{[d_2]})\in \RR$ are the mixed multiplicities of $\mathcal I(1)$ and $\mathcal I(2)$ by (\ref{M2}) and (\ref{eqV6}). Set $e_i=e_R(\mathcal I(1)^{[d-i]},\mathcal I(2)^{[i]})$ for $0\le i\le d$. Then
$$
f(n_1,n_2)=\sum_{i=0}^d\frac{1}{(d-i)!i!}e_in_1^{d-i}n_2^i.
$$
We have by (\ref{eqX31}) that
$$
e_0=e_R(\mathcal I(1))\mbox{ and } e_d=e_R(\mathcal I(2)).
$$
By Formulas 3) and 1) of Theorem \ref{TheoremMI},
\begin{equation}\label{eqT2*}
e_i^d\le e_0^{d-i}e_d^i\mbox{ for }0\le i\le d
\end{equation}
and
\begin{equation}\label{eqT3*}
e_i^2\le e_{i-1}e_{i+1}
\end{equation}
for $1\le i\le d-1$.
We expand 
$$
e_R(\mathcal I(1)\mathcal I(2))=d!f(1,1)=
\sum_{i=0}^d\binom{d}{i}e_i\le \sum_{i=0}^d\binom{d}{i}e_0^{\frac{d-i}{d}}e_d^{\frac{i}{d}}
=(e_0^{\frac{1}{d}}+e_d^{\frac{1}{d}})^d
$$
obtaining the Minkowski inequality (\ref{eqMinkIn}). Observe that 
\begin{equation}\label{eqT6}
\mbox{Equality holds in the Minkowski inequality if and only if equality holds in (\ref{eqT2*}) for all $i$.}
\end{equation}
 In this case,
 \begin{equation}\label{eq47}
 f(n_1,n_2)=\sum_{i=0}^d\frac{1}{(d-i)!i!}e_0^{\frac{d-i}{d}}e_d^{\frac{i}{d}}n_1^{d-i}n_2^i
 =\frac{1}{d!}(\sqrt[d]{e_0}n_1+\sqrt[d]{e_d}n_2)^d.
 \end{equation}
 
 We now show that when all $e_j$ are positive, the Minkowski equality holds between $\mathcal I(1)$ and $\mathcal I(2)$ if and only if equality holds in (\ref{eqT3*}).  We use an argument from \cite{HS}. Applying (\ref{eqT3*}), we have 
\begin{equation}\label{eqT4*}
\left(\frac{e_i}{e_{i-1}}\right)^{d-i}\cdots\left(\frac{e_1}{e_0}\right)^{d-i}\le 
\left(\frac{e_d}{e_{d-1}}\right)^i\cdots \left(\frac{e_{i+1}}{e_i}\right)^i
\end{equation}
where there are $i(d-i)$ terms on each side. We have equality in (\ref{eqT4*}) for $1\le i\le d-1$ if and only if equality holds in (\ref{eqT3*}) for all $i$. Now the LHS of (\ref{eqT4*}) is $\frac{e_i^{d-i}}{e_0^{d-i}}$ and the RHS is $\frac{e_d^i}{e_i^i}$ so 
\begin{equation}\label{eqT7*}
\mbox{Equality holds in (\ref{eqT3*}) for all $i$ if and only if  equality holds in (\ref{eqT2*}) for all $i$.}
\end{equation}

 \section{An  analysis of the Minkowski equality}\label{SecGeom}
 In this section, suppose that $R$ is a $d$-dimensional  analytically irreducible local domain and $\mathcal I(1)$ and $\mathcal I(2)$ are $m_R$-filtrations. 
 We further  assume that equality holds in Minkowski's inequality (\ref{eqMinkIn}) for $\mathcal I(1)$ and $\mathcal I(2)$. 
 We also make the additional assumptions that $e_0=e_R(\mathcal I(1))>0$ and $e_d=e_R(\mathcal I(2))>0$.
 We use the notation of Sections \ref{SecMinEQ}, \ref{SecFrame} and \ref{SecMMF}.

%Since $R$ is analytically irreducible and $e_R(\mathcal I(1))>0$ and $e_R(\mathcal I(2))>0$ we have that $e_i>0$ for all $i$ by \cite{CSV}. 

 Set $\alpha_1=\sqrt[d]{e_0}$ and $\alpha_2=\sqrt[d]{e_d}$. Choose $\phi$ so that (\ref{eq44}) is satisfied for these values of $\alpha_1$ and $\alpha_2$. Set  $\gamma=\phi^d{\rm Vol}(\Delta(R)\cap H_1^{-})$. We then define the function $h(n_1,n_2)={\rm Vol}(\Delta_{\Phi}(n_1,n_2))$. We have that for all $n_1,n_2\in \NN$,
 \begin{equation}\label{eq3}
 h(n_1,n_2)={\rm Vol}(\Delta_{\Phi}(n_1,n_2))=(\gamma-\frac{1}{\delta d!})(\alpha_1n_1+\alpha_2n_2)^d
 \end{equation} 
 where $\delta$ is the constant of (\ref{eq46}), by (\ref{eq42}) and (\ref{eq47}).
 
  Recall the  polynomial $g(n_1,n_2)$ of (\ref{eq48}).
  We have that ${\rm Vol}(\Delta_{\Phi}(1,0))=g(1,0)=h(1,0)>0$ and ${\rm Vol}(\Delta_{\Phi}(0,1))=g(0,1)=h(0,1)>0$
  and $g(n_1,n_2)\le h(n_1,n_2)$ for all $n_1,n_2\in \NN$ by (\ref{eq43}). Since $g$ and $h$ are homogeneous polynomials of the same degree, we have that $g(a_1,a_2)\le h(a_1,a_2)$ for all $a_1,a_2\in \QQ_{\ge 0}$. Thus, by continuity of polynomials, 
  \begin{equation}\label{eq49}
  g(a_1,a_2)\le h(a_1,a_2)\mbox{ for all $a_1,a_2\in \RR_{\ge 0}$.}
  \end{equation}
  For $0<t<1$, we have that
  $$
  \begin{array}{lll}
  h(1-t,t)^{\frac{1}{d}}&=&(1-t)h(1,0)^{\frac{1}{d}}+th(0,1)^{\frac{1}{d}}
  =(1-t)g(1,0)^{\frac{1}{d}}+tg(0,1)^{\frac{1}{d}}\\
  &&\le g(1-t,t)^{\frac{1}{d}}\le h(1-t,t)^{\frac{1}{d}}.
  \end{array}
  $$
by  (\ref{eq3}),  Theorem \ref{BMTheorem} and (\ref{eq49}). Thus $g(1-t,t)^{\frac{1}{d}}=(1-t)g(1,0)^{\frac{1}{d}}+tg(0,1)^{\frac{1}{d}}$ for $0<t<1$ and so $\Delta_{\Phi}(1,0)$ and $\Delta_{\Phi}(0,1)$ are homothetic by Theorem \ref{BMTheorem}. 

We have $\mbox{Vol}(\Delta_{\Phi}(0,1))=\frac{e_d}{e_0}\mbox{Vol}(\Delta(1,0))$ by (\ref{eq3}).
Let $T:\RR^d\rightarrow \RR^d$, given by 
$$
T(\vec x)=c\vec x+\vec \gamma
$$
 be the homothety such that $T(\Delta_{\Phi}(1,0))=\Delta_{\Phi}(0,1)$. We have that
 $$
 \frac{e_d}{e_0}{\rm Vol}(\Delta_{\Phi}(1,0)={\rm Vol}(\Delta_{\Phi}(0,1))
 ={\rm Vol}(T(\Delta_{\Phi}(1,0))=c^d{\rm Vol}(\Delta_{\Phi}(1,0)
 $$
 so 
 $$
 c=\sqrt{\frac{e_d}{e_0}}.
 $$
By (\ref{eq70}),  (\ref{eq71}) and (\ref{MF2}), we have that
\begin{equation}\label{eq61}
\Delta_{\Phi}(1,0)\cap H_{\psi}=\left\{\begin{array}{ll}
\emptyset&\mbox{ for }\psi>\sqrt[d]{e_0}\phi\\
\Delta(R)\cap H_{\sqrt[d]{e_0}\phi}&\mbox{ for }\psi=\sqrt[d]{e_0}\phi
\end{array}
\right.
\end{equation}
and
\begin{equation}\label{eq62}
\Delta_{\Phi}(0,1)\cap H_{\psi}=\left\{\begin{array}{ll}
\emptyset&\mbox{ for }\psi>\sqrt[d]{e_d}\phi\\
\Delta(R)\cap H_{\sqrt[d]{e_d}\phi}&\mbox{ for }\psi=\sqrt[d]{e_d}\phi.
\end{array}
\right.
\end{equation}

Writing $\vec\gamma=(\gamma_1,\ldots,\gamma_d)$, we have
$T(H_{\sqrt[d]{e_0}\phi})=H_{\sqrt[d]{e_d}\phi+(\gamma_1+\cdots+\gamma_d)}$.
Comparing equations (\ref{eq61}) and (\ref{eq62}), we see that
$\gamma_1+\cdots+\gamma_d=0$. 

Now $\Delta(R)\cap H_{\psi}=\psi(\Delta(R)\cap H_1)$ for all $\psi\in \RR_{>0}$ by (\ref{eq1}). Thus we may factor the homeomorphism $T:\Delta(R)\cap H_{\sqrt[d]{e_0}\phi}\rightarrow \Delta(R)\cap H_{\sqrt[d]{e_d}\phi}$ by homeomorphisms
$$
\Delta(R)\cap H_{\sqrt[d]{e_0}\phi}\stackrel{c}{\rightarrow}\Delta(R)\cap H_{\sqrt[d]{e_d}\phi}\stackrel{+\vec\gamma}{\rightarrow} \Delta(R)\cap H_{\sqrt[d]{e_d}\phi}.
$$
But $\Delta(R)\cap H_{\sqrt[d]{e_d}\phi}$ is a nonempty compact set, so the second map cannot be well defined  unless $\vec\gamma=0$.

In summary, we have established the following theorem. 

\begin{Theorem}\label{Theorem3} Suppose that $R$ is a $d$-dimensional analytically irreducible  excellent local ring and $\mathcal I(1)$ and $\mathcal I(2)$ are $m_R$-filtrations  which have positive multiplicity $e_0=e_R(\mathcal I(1))>0$ and $e_d=e_R(\mathcal I(2))>0$ and that Minkowski's equality
$$
e_R(\mathcal I(1)\mathcal I(2))^{\frac{1}{d}}=e_R(\mathcal I(1))^{\frac{1}{d}}+e_R(\mathcal I(2)^{\frac{1}{d}}
$$
holds. Let notation be as in sections \ref{SecFrame}, \ref{SecMMF} and \ref{SecMinEQ}. Then 
$$
\sqrt[d]{e_d}\Delta_{\Phi}(1,0)=\sqrt[d]{e_0}\Delta_{\Phi}(0,1)
$$
where $\Phi=(\sqrt[d]{e_0},\sqrt[d]{e_d},\phi)$ in (\ref{eq70}) and $\phi$ is sufficiently large.
\end{Theorem}

We also obtain a partial converse to Theorem \ref{Theorem3}.

\begin{Theorem}\label{Theorem7}  Suppose that $R$ is a $d$-dimensional analytically irreducible excellent  local ring and $\mathcal I(1)$ and $\mathcal I(2)$ are $m_R$-filtrations  which have positive multiplicity $e_0=e_R(\mathcal I(1))>0$ and $e_d=e_R(\mathcal I(2))>0$. 
Suppose that $\sqrt[d]{e_d}\Delta_{\Phi}(1,0)=\sqrt[d]{e_0}\Delta_{\Phi}(0,1)$ for
$\Phi=(\sqrt[d]{e_0},\sqrt[d]{e_d},\phi)$ in (\ref{eq70}) with $\phi$  sufficiently large
and that for the functions of (\ref{eq48}) and (\ref{eq3}) $g(n_1,n_2)=h(n_1,n_2)$ for all $n_1,n_2\in \ZZ^2$.
Then Minkowski's equality 
$$
e_R(\mathcal I(1)\mathcal I(2))^{\frac{1}{d}}=e_R(\mathcal I(1))^{\frac{1}{d}}+e_R(\mathcal I(2))^{\frac{1}{d}}
$$
holds between $\mathcal I(1)$ and $\mathcal I(2)$.
\end{Theorem}
\begin{proof}
The convex bodies $\Delta_{\Phi}(1,0)$ and $\Delta_{\Phi}(0,1)$ are homothetic, so by Theorem \ref{BMTheorem},
$$
g(1-t,t)^{\frac{1}{d}}=(1-t)g(1,0)^{\frac{1}{d}}+tg(0,1)^{\frac{1}{d}}
$$
for $0\le t\le 1$. Taking $t=\frac{1}{2}$ and since $g$ is a homogeneous polynomial of degree $d$, we obtain that $g(1,1)^{\frac{1}{d}}=g(1,0)^{\frac{1}{d}}+g(0,1)^{\frac{1}{d}}$.
Thus $h(1,1)^{\frac{1}{d}}=h(1,0)^{\frac{1}{d}}+h(0,1)^{\frac{1}{d}}$. 
By equations (\ref{eq41}) and (\ref{eq42}),
$$
\begin{array}{lll}
f(n_1,n_2)&=& \delta[{\rm Vol}(\tilde \Delta_{\Phi}(n_1,n_2))-{\rm Vol}(\Delta_{\Phi}(n_1,n_2))]\\
&=&\delta\phi^d{\rm Vol}(\Delta(R)\cap H_1^{-})(\sqrt[d]{e_0}n_1+\sqrt[d]{e_d}n_2)^d-\delta h(n_1,n_2).
\end{array}
$$
Set $\xi=d!\delta\phi^d{\rm Vol}(\Delta(R)\cap H_1^{-})$. We have that
$$
e_R(\mathcal I(1)\mathcal I(2))=d!f(1,1)=\xi (\sqrt[d]{e_0}+\sqrt[d]{e_d})^d-d!\delta h(1,1),
$$
$$
e_0=e_R(\mathcal I(1))=d!f(1,0)=\xi e_0-d!\delta h(1,0),
$$
$$
e_d=e_R(\mathcal I(2))=d!f(0,1)=\xi e_d-d!\delta h(0,1).
$$
Let $\chi=\frac{\xi-1}{d!\delta}$, so that $h(1,0)=\chi e_0$ and $h(0,1)=\chi e_d$.
We have
$$
h(1,1)=(h(1,0)^{\frac{1}{d}}+h(0,1)^{\frac{1}{d}})^d=\chi(\sqrt[d]{e_0}+\sqrt[d]{e_d})^d
$$
and
$$
e_R(\mathcal I(1)\mathcal I(2))=(\xi-d!\delta\chi)(\sqrt[d]{e_0}+\sqrt[d]{e_d})^d.
$$
Now $\xi-d!\delta\chi=1$ so the Minkowski equality
$$
e_R(\mathcal I(1)\mathcal I(2))^{\frac{1}{d}}=e_R(\mathcal I(1))^{\frac{1}{d}}+e_R(\mathcal I(2))^{\frac{1}{d}}
$$
holds.
\end{proof}

\begin{Theorem}\label{Theorem1} Suppose that $R$ is a $d$-dimensional  analytically irreducible excellent local ring and $\mathcal I(1)$ and $\mathcal I(2)$ are $m_R$-filtrations  such that 
$e_R(\mathcal I(1))$ and $e_R(\mathcal I(2))$ are both non zero and equality holds in the Minkowski inequality (\ref{eqMinkIn}) for $\mathcal I(1)$ and $\mathcal I(2)$. Then for all $m_R$-valuations $\mu$ of $R$, we have that 
$$
e(\mathcal I(2))^{\frac{1}{d}}\gamma_{\mu}(\mathcal I(1))
=e(\mathcal I(1))^{\frac{1}{d}}\gamma_{\mu}(\mathcal I(2)).
$$
\end{Theorem}

\begin{proof} Starting with $\nu=(\mu,\omega)$ in the construction of Section \ref{SecFrame}, construct $\Delta_{\Phi}(n_1,n_2)$ as in Section \ref{SecMMF}, so that the conclusions of Theorem \ref{Theorem3} hold. 
Let $\pi:\RR^{d}\rightarrow \RR$ be the projection onto the first factor. By definition of $\gamma_{\mu}(\mathcal I(1))$, $\gamma_{\mu}(\mathcal I(1))$ is in  the compact set $\pi(\Delta_{\Phi}(1,0))$, $\pi^{-1}(\gamma_{\mu}(\mathcal I(1))\cap \Delta_{\Phi}(1,0)\ne \emptyset$ and $\pi^{-1}(a)\cap \Delta_{\Phi}(1,0)=\emptyset$ if $a<\gamma_{\mu}(\mathcal I(1))$. In the same way, we have that $\pi^{-1}(\gamma_{\mu}(\mathcal I(2))\cap \Delta_{\Phi}(0,1)\ne \emptyset$ and $\pi^{-1}(a)\cap \Delta_{\Phi}(0,1)=\emptyset$ if $a<\gamma_{\mu}(\mathcal I(2))$.

Let $T: \RR^d\rightarrow\RR^d$ be the homothety $T(\vec x)=c\vec x$ where $c=\frac{\sqrt[d]{e_d}}{\sqrt[d]{e_0}}$ which takes $\Delta_{\Phi}(1,0)$ to $\Delta_{\Phi}(0,1)$.
Now since $T$ multiplies the first coefficient of an element of $\Delta_{\Phi}(1,0)$ by $c$,
and the smallest first coefficient of an element of $\Delta_{\Phi}(1,0)$ is $\gamma_{\mu}(\mathcal I(1)$, the smallest first coefficient of an element of $\Delta_{\Phi}(0,1)$ is $\gamma_{\mu}(\mathcal I(2))=c\gamma_\mu(\mathcal I(1))$.
\end{proof}

Let us verify that these equalities do in fact hold in the classical case of $m_R$-primary ideals $I(1)$ and $I(2)$ satisfying the Minkowski equality. In this case, we have the (Noetherian) $m_R$-filtrations
$\mathcal I(1)=\{I(1)^i\}$ and $\mathcal I(2)=\{I(2)^i\}$. Since the Minkowski equality holds, we have that there exists $m,n\in \ZZ_+$ such that $\overline{I(1)^m}=\overline{I(2)^n}$ where $\overline{I(1)^m}$ and $\overline{I(2)^n}$ are the respective integral closures of ideals by the Teissier, Rees and Sharp, Katz Theorem \cite{T3}, \cite{RS}, \cite{Ka} recalled in Subsection \ref{SubSecINEQ}. Now 
$$
e(\overline{I(1)^m})=m^de(I(1))=m^de(\mathcal I(1)),
$$
$$
e(\overline{I(2)^n})=n^de(I(2))=n^de(\mathcal I(2))
$$
 and 
 $$
 m\gamma_{\mu}(\mathcal I(1))=m\mu(I(1))=\mu(\overline{I(1)^m}),
 $$
$$
 n\gamma_{\mu}(\mathcal I(2))=n\mu(I(2))=\mu(\overline{I(2)^n}),
 $$
giving the desired formula
 $$
e(\mathcal I(2))^{\frac{1}{d}}\gamma_{\mu}(\mathcal I(1))
=e(\mathcal I(1))^{\frac{1}{d}}\gamma_{\mu}(\mathcal I(2)).
$$

\section{Equality of mixed multiplicities on normal excellent local rings}

\begin{Theorem}\label{Cor1.1} Let $R$ be a $d$-dimensional normal excellent local domain and let $\mathcal I(D_1)$ and $\mathcal I(D_2)$ be real divisorial $m_R$-filtrations. Let $X\rightarrow\mbox{Spec}(R)$ and $D_1=\sum_{i=1}^ra_iE_i$, $D_2=\sum_{i=1}^rb_iE_i$ be a representation of $D_1$ and $D_2$.
Suppose that Minkowski's equality holds for $\mathcal I(D_1)$ and $\mathcal I(D_2)$. Then there exists an effective real Weil divisor $\sum_{i=1}^r \gamma_i E_i$ such that 
$$
\gamma_{E_j}(D_i)=\gamma_je(\mathcal I(D_i))^{\frac{1}{d}}
$$
for all $j$ and $i$ and 
$$
I(mD_i)=
\Gamma(X,\mathcal O_X(-\lceil \sum_{j=1}^rm\gamma_je(\mathcal I(D_i))^{\frac{1}{d}}E_j\rceil)
$$
for $i=1$ and $2$ and all $m\in \NN$.
\end{Theorem}

\begin{proof} We have that both $e_R(\mathcal I(D_1))$ and $e_R(\mathcal I(D_2))$ are positive by Proposition \ref{PropPos}. We have that  $E_1,E_2,\ldots,E_r$ are the irreducible exceptional divisors of $\phi:X\rightarrow \mbox{Spec}(R)$.  For $1\le j\le r$, let 
$$
\gamma_j=\frac{\gamma_{E_j}(D_1)}{e(\mathcal I(D_1))^{\frac{1}{d}}}.
$$
By Theorem \ref{Theorem1},  taking $\mu=\nu_{E_j}$,
$$
\gamma_j=\frac{\gamma_{E_j}(D_2)}{e(\mathcal I(D_2))^{\frac{1}{d}}}
$$
for $1\le j\le r$. Now for $i=1,2$ and $m\in \NN$, we have by Lemma \ref{LemmaAR1} that
$$
I(mD_i)=\Gamma(X,\mathcal O_X(-\lceil \sum_{j=1}^rm\gamma_{E_j}(D_i)E_j\rceil)
=\Gamma(X,\mathcal O_X(-\lceil \sum_{j=1}^rme(\mathcal I(D_i))^{\frac{1}{d}}\gamma_jE_j\rceil).
$$
\end{proof}

  \begin{Corollary}\label{Cor1.2} Let $R$ be a normal $d$-dimensional excellent local domain and let $\mathcal I(D_1)$ and $\mathcal I(D_2)$ be real divisorial $m_R$-filtrations. Thus $e_i=e_R(\mathcal I(D_1)^{[d-i]},\mathcal I(D_2)^{[i]})>0$ for $0\le i\le d$ by Proposition \ref{PropPos}.
  Let $X\rightarrow\mbox{Spec}(R)$ and $D_1=\sum_{i=1}^ra_iE_i$, $D_2=\sum_{i=1}^rb_iE_i$ be a representation of $D_1$ and $D_2$.
   Suppose that Minkowski's equality holds for $\mathcal I(D_1)$ and $\mathcal I(D_2)$ and that for some $i$,
$$
\frac{e_i}{e_{i-1}}=\frac{a}{b}\in \QQ
$$
where $a,b\in \ZZ_{>0}$. Then
$$
I(maD_1)=I(mbD_2)
$$
for all $m\in \NN$.
\end{Corollary}

\begin{proof} With our assumption that the Minkowski inequality is an equality, we have from the observation before (\ref{eq47}) that $e_j=e_0^{\frac{d-j}{d}}e_d^{\frac{j}{d}}$ for all $j$. Since $e_0=e_R(\mathcal I(D_1)),e_d=e_R(\mathcal I(D_2))>0$, we have that $e_j^2=e_{j-1}e_{j+1}$ for  $1\le j\le d-1$ by (\ref{eqT7*}). Thus 
$\frac{e_j}{e_{j-1}}=\frac{e_{j+1}}{e_j}$ for  $1\le j\le d-1$, and so $\frac{e_j}{e_{j-1}}=\frac{a}{b}$ for  $1\le j\le d$
and so $$
\left(\frac{a}{b}\right)^d=\frac{e_d}{e_0}.
$$
We have $e_d^{\frac{1}{d}}\gamma_{E_j}(D_1)=e_0^{\frac{1}{d}}\gamma_{E_j}(D_2)$ for all $j$ by Theorem \ref{Theorem1}.
Now for all $j$,
$$
\frac{a}{b}\frac{\gamma_{E_j}(D_1)}{\gamma_{E_j}(D_2)}=
\frac{e_d^{\frac{1}{d}}\gamma_{E_j}(D_1)}{e_0^{\frac{1}{d}}\gamma_{E_j}(D_2)}=1
$$
so that $a\gamma_{E_j}(D_1)=b\gamma_{E_j}(D_2)$ for all $j$. The conclusions of the corollary now follow from Lemma \ref{LemmaAR1}.
\end{proof}

 Suppose that $R$ is a local domain and $\mathcal I=\{I_m\}$ is a filtration of $m_R$-primary ideals. 
 We define a   function $w_{\mathcal I}$ on $R$ by
 \begin{equation}\label{eqT9}
 w_{\mathcal I}(f)=\max\{m\mid f\in I_m\}
 \end{equation}
 for $f\in R$. We have that $w_{\mathcal I}(f)$ is either a natural number or $\infty$.

 \begin{Lemma}\label{ratfun}
 Suppose that $R$ is a normal excellent local domain and that $\mathcal I=\mathcal I(D)$ 
 is a rational divisorial $m_R$-filtration. 
  Suppose that $f\in m_R$ is nonzero. 
 Then $w_{\mathcal I(D)}(f)<\infty$ and 
 there exists $d\in\ZZ_{>0}$ such that $w_{\mathcal I(D)}(f^{nd})=nw_{\mathcal  I(D)}(f^d)$ for all $n\in \ZZ_{>0}$.
 \end{Lemma}

 \begin{proof} By assumption, there exists a representation $X\rightarrow \mbox{Spec}(R)$ and $D=\sum_{i=1}^r a_iE_i$ of $\mathcal I(D)$ where the $a_i$ are all nonnegative rational numbers and some $a_i>0$.
 Let $b_i=\nu_{E_i}(f)$ for $1\le i\le r$. Then $f\in I(mD)$ if and only if $\nu_{E_i}(f)=b_i\ge ma_i$ for $1\le i\le r$. Since some $a_i>0$ we have that $w_{\mathcal I(D)}(f)<\infty$.
 
 Let 
 $$
 t=\min\left\{\frac{b_i}{a_i}\mid 1\le i\le r\mbox{ and }a_i\ne 0\right\}.
 $$
  Since all $b_i>0$ and $D$ is a rational divisor, $t$ is a positive rational number, so we can write $t=\frac{c}{d}$ with $c,d\in \ZZ_{>0}$. Let $i_0$ be an index such that $t=\frac{b_{i_0}}{a_{i_0}}$. For all $n\in \ZZ_{>0}$ we have that
 $\nu_{E_i}(f^{nd})= nd\nu_{E_i}(f)=ndb_i\ge nca_i$ for all $i$ and $\nu_{E_{i_0}}(f^{nd})=ndb_{i_0}=nca_{i_0}$ so that $w_{\mathcal I(D)}(f^{dn})=nc=nw_{\mathcal I(D)}(f^d)$.
 \end{proof}

 \begin{Theorem}\label{Theorem8} Suppose that $R$ is a $d$-dimensional normal excellent local ring. Let $\mathcal I(D_1)$ and $\mathcal I(D_2)$ be rational divisorial $m_R$-filtrations. Let $X\rightarrow \mbox{Spec}(R)$,
$D_1=\sum_{i=1}^r a_iE_i$ and $D_2=\sum_{i=1}^rb_iE_i$ be a representation of $D_1$ and $D_2$ (the $a_i$ and $b_i$ are nonnegative rational numbers). Suppose that the Minkowski equality holds between $\mathcal I(D_1)$ and $\mathcal I(D_2)$. Then 
 $$
 \xi=\frac{\sqrt[d]{e_R(\mathcal I(D_2))}}{\sqrt[d]{e_R(\mathcal I(D_1))}}\in \QQ_{>0}.
 $$ 
 Writing 
 $$
 \frac{\sqrt[d]{e_R(\mathcal I(D_2))}}{\sqrt[d]{e_R(\mathcal I(D_1))}}=\frac{a}{b}
 $$
 with $a,b\in \ZZ_{>0}$, we have that
 $$
 I(maD_1)=I(mbD_2)
 $$
 for all $m\in \NN$.
  \end{Theorem}
 
 \begin{proof}
 Minkowski's inequality holds by assumption, and $e_R(\mathcal I(D_i))>0$ for  $i=1,2$ by Proposition \ref{PropPos}. Thus by Corollary \ref{Cor1.1} we have that  there exist $\overline \gamma_j\in \RR_{> 0}$ such that $\gamma_{E_j}(D_i)=e(\mathcal I(D_i))^{\frac{1}{d}}\overline \gamma_j$ for $i=1,2$ and $1\le j\le r$.
 Let $\gamma_j=\gamma_{E_j}(D_1)$ for $1\le j\le r$.
  Thus, with $\xi$ as defined in the statement of the theorem, 
  $$
  \gamma_{E_i}(D_2)=\xi \gamma_{E_i}(D_1)=\xi \gamma_i
  $$ 
  for all $i$ and for all $m\in \NN$,
  $$
  I(mD_1)=\Gamma(X,\mathcal O_X(-\lceil \sum_{i=1}^rm\gamma_iE_i\rceil)
  $$
  and
  $$
  I(mD_2)=\Gamma(X,\mathcal O_X(-\lceil \sum_{i=1}^rm\xi_i\gamma_i E_i\rceil).
  $$
  
  Let $g\in m_R$ be nonzero.  By Lemma \ref{ratfun}, there exists $\overline d\in \ZZ_+$ such that
  $w_{\mathcal I(D_i)}(g^{\overline dl})=lw_{\mathcal I(D_i)}(g^{\overline d})$ for $i=1,2$ and all $l\in \ZZ_{>0}$. Let $f=g^{\overline d}$.
  
 Let $m=w_{\mathcal I(D_2)}(f)>0$ and $n=w_{\mathcal I(D_1)}(f)>0$.
 Let $\delta\in\RR_{>0}$.
 Now by Lemma \ref{LemmaHW},  there exists $\alpha\in \RR_{>0}$ such that $\alpha<\frac{\delta}{m}$ and there exists $\alpha'\in \RR_{>0}$ such that $\alpha'<\frac{\delta}{m}$ and there exist positive integers $p_0,q_0,p_0',q_0'$ such that 
 $$
 \xi-\frac{\alpha}{q_0}<\frac{p_0}{q_0}\le\xi
 $$
 and
 $$
 \xi\le\frac{p_0'}{q_0'}<\xi+\frac{\alpha'}{q_0'}.
 $$
Let $p=p_0m$, $q=q_0m$. Then 
$\gamma_ip\le \xi\gamma_iq$ for all $i$ so that 
$I(qD_2)\subset I(pD_1)$. We have that $f^{q_0}\in I(qD_2)\subset  I(pD_1)$ so that $w_{\mathcal I(D_1)}(f^{q_0})\ge p$. Thus since $w_{\mathcal I(D_1)}(f^{q_0})=q_0w_{\mathcal I(D_1)}(f)$ (by our choice of $f$),
\begin{equation}\label{val3} 
w_{\mathcal I(D_1)}(f)\ge \frac{p_0}{q_0}m=\frac{p_0}{q_0}w_{\mathcal I(D_2)}(f)>(\xi-\frac{\alpha}{q_0})w_{\mathcal I(D_2)}(f).
\end{equation} 

Let $p'=p_0'n$, $q'=q_0'n$. Then $\gamma_i\xi q'\le  \gamma_ip'$ for all $i$ so that $I(p'D_1)\subset I(q'D_2)$. We have that $f^{p_0'}\in I(p'D_1)\subset I(q'D_2)$. Thus  since $w_{\mathcal I(D_2)}(f^{p_0'})=p_0'w_{\mathcal I(D_2)}(f)$ (by our choice of $f$), we have that
$w_{\mathcal I(D_2)}(f)\ge \frac{q_0'n}{p_0'}=\frac{q_0'}{p_0'}w_{\mathcal I(D_1)}(f)$. So
 \begin{equation}\label{val4}
w_{\mathcal I(D_1)}(f)\le\frac{p_0'}{q_0'}w_{\mathcal I(D_2)}(f)<(\xi+\frac{\alpha'}{q_0'})w_{\mathcal I(D_2)}(f).
\end{equation}
Combining equations (\ref{val3}) and (\ref{val4}), we have that 
$$
\begin{array}{lll}
w_{\mathcal I(D_1)}(f)&\le& (\xi+\frac{\alpha'}{q_0'})w_{\mathcal I(D_2)}(f)\\
&=&(\xi-\frac{\alpha}{q_0})w_{\mathcal I(D_2)}(f)+(\frac{\alpha}{q_0}+\frac{\alpha'}{q_0'})w_{\mathcal I(D_2)}(f)\\
&\le& w_{\mathcal I(D_1)}(f)+(\frac{\alpha}{q_0}+\frac{\alpha'}{q_0'})w_{\mathcal I(D_2)}(f)\\
&<&w_{\mathcal I(D_1)}(f)+2\delta.
\end{array}
$$
All these inequalities approach  equalities when the limit is taken as $\delta\mapsto 0$.
Thus $w_{\mathcal I(D_1)}(f)=\xi w_{\mathcal I(D_2)}(f)$, and so 
$$
\xi=\frac{w_{\mathcal I(D_1)}(f)}{w_{\mathcal I(D_2)}(f)}\in \QQ_{>0}.
$$

Now we prove the last statement of the theorem. By Theorem \ref{Theorem1}, we have that 
$$
\gamma_{E_i}(D_1)=\frac{e_R(\mathcal I(D_1))^{\frac{1}{d}}}{e_R(\mathcal I(D_2))^{\frac{1}{d}}}\gamma_{E_i}(D_2)
$$
for $1\le i\le r$. Substituting into 
$I(maD_1)=\Gamma(X,\mathcal O_X(-\lceil \sum_{i=1}^rma\gamma_{E_i}(D_1)E_i\rceil)$,
we obtain that 
$$
I(maD_1)=\Gamma(X,\mathcal O_X(-\lceil\sum_{i=1}^rmb\gamma_{E_i}(D_2)E_i\rceil)=I(mbD_2)
$$
for all $m\in \NN$.
\end{proof}
  
\begin{Theorem}\label{Theorem6} 
Suppose that $R$ is a $d$-dimensional normal excellent local ring. Let $\mathcal I(D_1)$ and $\mathcal I(D_2)$ be real divisorial $m_R$-filtrations. Let $X\rightarrow \mbox{Spec}(R)$,
$D_1=\sum_{i=1}^r a_iE_i$ and $D_2=\sum_{i=1}^rb_iE_i$ be a representation of $D_1$ and $D_2$. 
Then the following are equivalent
\begin{enumerate}
\item[1)] The Minkowski equality holds for $\mathcal I(D_1)$ and $\mathcal I(D_2)$
\item[2)]
$$
\frac{\gamma_{E_i}(D_2)}{\gamma_{E_i}(D_1)}=\frac{e_R(\mathcal I(D_2))^{\frac{1}{d}}}{e_R(\mathcal I(D_1))^{\frac{1}{d}}}
$$
for all $i$.
\item[3)]  For all $i$ and $j$ we have that
\begin{equation}\label{eqT7}
\frac{\gamma_{E_i}(D_2)}{\gamma_{E_i}(D_1)}=\frac{\gamma_{E_j}(D_2)}{\gamma_{E_j}(D_1)}.
\end{equation}
\end{enumerate}
\end{Theorem}

\begin{proof} We have that both $e_R(\mathcal I(D_1))$ and $e_R(\mathcal I(D_2))$ are positive by Proposition \ref{PropPos}. 

First suppose that Minkowski's equality holds between $\mathcal I(D_1)$ and $\mathcal I(D_2)$. Then by Theorem \ref{Theorem1},
$$
\frac{\gamma_{E_i}(D_2)}{\gamma_{E_i}(D_1)}=\frac{e_R(\mathcal I(D_2))^{\frac{1}{d}}}{e_R(\mathcal I(D_1))^{\frac{1}{d}}}
$$
for all $i$. Thus 2) holds. If 2) holds then 3) certainly holds.

Now suppose that
3) holds
for all $i,j$. 
Let $\gamma_i=\gamma_{E_i}(D_1)$ and let $\xi\in \RR_{>0}$ be such that 
$$
\xi=\frac{\gamma_{E_i}(D_2)}{\gamma_{E_i}(D_1)}
$$
for all $i$. Then $\gamma_{E_i}(D_2)=\xi\gamma_{E_i}(D_i)$ for all $i$. 
For $\lambda\in \RR_{>0}$ and $n\in \NN$, define
$$
K(\lambda)_n=\Gamma(X,\mathcal O_X(-\lceil n\lambda \gamma_1E_1+\cdots+n\lambda \gamma_rE_r\rceil)),
$$
and a filtration of $m_R$-primary ideals $\mathcal K(\lambda) =\{K(\lambda)_n\}$.
Observe that $\mathcal K(\lambda)=\mathcal I(\sum_{i=1}^r\lambda\gamma_iE_i)$.

 For $n_1,n_2\in \NN$ define
$$
J(n_1,n_2)_m=I(mn_1D_1)I(mn_2D_2)
$$
and a filtration of $m_R$-primary ideals $\mathcal J(n_1,n_2)=\{J(n_1,n_2)_m\}$. We have that for all $n_1,n_2$, $J(n_1,n_2)_m\subset K(n_1+n_2\xi)_m$ for all $m$ so that 
$$
\Delta(\mathcal J(n_1,n_2))\subset \Delta(\mathcal K(n_1+n_2\xi))
$$
for all $n_1,n_2$. We further have that $n_1\Delta(\mathcal I(D_1))+n_2\Delta(\mathcal I(D_2))\subset \Delta(\mathcal J(n_1,n_2))$ for all $n_1,n_2$. We have that $\Delta(\mathcal I(D_1))=\Delta(\mathcal K(1))$.
Now by Proposition \ref{Prop11}, we have that $\Delta(\mathcal I(D_2))=\xi\Delta(\mathcal K(1))$ and $\Delta(\mathcal K(n_1+n_2\xi))=(n_1+n_2\xi)\Delta(\mathcal K(1))$. So $n_1\Delta(\mathcal I(D_1))+n_2\Delta(\mathcal I(D_2))=(n_1+n_2\xi)\Delta(\mathcal K(1))$ and thus
$\Delta(\mathcal J(n_1,n_2))=n_1\Delta(\mathcal I(D_1))+n_2\Delta(\mathcal I(D_2))$ for all $n_1,n_2\in \NN$.
Thus the conditions of Theorem \ref{Theorem7} are satisfied, and so Minkowski's equality holds between $\mathcal I(D_1)$ and $\mathcal I(D_2)$.
\end{proof}

\begin{Theorem}\label{Theorem9} Suppose that $R$ is a normal excellent local ring. Let $\mathcal I(D_1)$ and $\mathcal I(D_2)$ be rational divisorial $m_R$-filtrations.
Then $\mathcal I(D_1)$ and $\mathcal I(D_2)$ satisfy the Minkowski equality if and only if there exist $a,b\in \ZZ_{>0}$ such that $I(amD_1)=I(bmD_2)$ for all $m\in \NN$. 
\end{Theorem}

\begin{proof} 
 Let $X\rightarrow \mbox{Spec}(R)$,
$D_1=\sum_{i=1}^r a_iE_i$ and $D_2=\sum_{i=1}^rb_iE_i$ be a representation of $D_1$ and $D_2$.

If $\mathcal I(D_1)$ and $\mathcal I(D_2)$ satisfy the Minkowski equality then there exist $a,b\in \ZZ_{>0}$ such that $I(amD_1)=I(bmD_2)$ for all $m\in \NN$ by Theorem \ref{Theorem8}.

Suppose that there exist $a,b\in \ZZ_{>0}$ such that $I(amD_1)=I(bmD_2)$ for all $m\in \NN$. With this assumption, $\gamma_{E_i}(aD_1)=\gamma_{E_i}(bD_2)$ for $1\le i\le r$. Now $\gamma_{E_i}(aD_1)=a\gamma_{E_i}(D_1)$ and $\gamma_{E_i}(bD_2)=b\gamma_{E_i}(D_2)$, so
$$
\frac{\gamma_{E_i}(D_2)}{\gamma_{E_i}(D_1)}=\frac{a}{b}
$$
for $1\le i\le r$. Thus the Minkowski equality holds for $D_1$ and $D_2$ by Theorem \ref{Theorem6}.
\end{proof}

\section{Excellent local domains and the Minkowski equality}\label{SecExD}

\begin{Theorem}\label{Theorem10} Suppose that $R$ is a $d$-dimensional excellent local domain.
Let $\mathcal I(D_1)$ and $\mathcal I(D_2)$ be integral divisorial $m_R$-filtrations. Then the Minkowski equality holds between $\mathcal I(D_1)$ and $\mathcal I(D_2)$ if and only if there exist $a,b\in \ZZ_{>0}$ such that $I(amD_1)=I(bmD_2)$ for all $m\in \NN$.
\end{Theorem}

\begin{proof} We use the notation of Subsection \ref{Not}. Let $S$ be the normalization of $R$ with maximal ideals $m_i$. we have that $D_1=\sum_{i=1}^tD_1(i)$, $D_2=\sum_{i=1}^tD_2(i)$. Write
$$
P(n_1,n_2)=\lim_{m\rightarrow\infty}\frac{\ell_R(R/I(mn_1D_1)I(mn_2D_2))}{m^d}=\sum_{j=0}^d \frac{1}{(d-j)!j!}e_jn_1^{d-j}n_2^j
$$
and 
$$
P_i(n_1,n_2)=\lim_{m\rightarrow\infty}\frac{\ell_{S_{m_i}}(S_{m_i}/J(mn_1D_1(i))J(mn_2D_2(i)))}{m^d}=\sum_{j=0}^d\frac{1}{(d-j)!j!}e(i)_jn_1^{d-j}n_2^j.
$$
We have that
\begin{equation}\label{eqE1}
P(n_1,n_2)=\sum_{i=1}^ta_iP_i(n_1,n_2)
\end{equation}
with $a_i=[S/m_i:R/m_R]$ for $1\le i\le t$ by Lemma \ref{LemmaR1} and (\ref{eqR15}).
Let $\mathcal J(D_k(i))$ be the filtration $\{J(mD_k(i))\}$ for $k=1,2$ and all $i$. 

Since $D_1,D_2\ne 0$ we have that some $D_1(i)\ne 0$ and some $D_2(j)\ne 0$. Thus $e(i)_0>0$ and $e(j)_d>0$ by Proposition \ref{PropPos} and so $e_0>0$ and $e_d>0$ by (\ref{eq1}).
 Since the Minkowski equality holds between $\mathcal I(D_1)$ and $\mathcal I(D_2)$ we have by (\ref{eqT6}) that equality holds in (\ref{eqT3*}) for the $e_i$, so (\ref{eq47}) holds which implies all $e_i>0$. Thus there exists $\xi\in \RR_{>0}$ such that
\begin{equation}\label{eqT11}
\xi=\frac{e_1}{e_0}=\cdots=\frac{e_d}{e_{d-1}}.
\end{equation}
By (\ref{eqE1}) we have that $e_j=\sum_{i=1}^t a_ie(i)_j$ for all $j$. By the inequality (\ref{eqT3*}) and (\ref{eqT11}) we have that
$$
\begin{array}{lll}
0&\le& \sum_{i=1}^ta_i(e(i)_{j+1}^{\frac{1}{2}}-\xi e(i)_{j-1}^{\frac{1}{2}})^2
=\sum_{i=1}^ta_i(e(i)_{j+1}-2\xi e(i)_{j+1}^{\frac{1}{2}}e(i)_{j-1}^{\frac{1}{2}}+\xi^2e(i)_{j-1})\\
&\le& \sum_{i=1}^ta_i(e(i)_{j+1}-2\xi e(i)_j+\xi^2 e(i)_{j-1})\\
&=& e_{j+1}-2\xi e_j+\xi^2e_{j-1}\\
&=&\xi^2 e_{j-1}-2\xi^2 e_{j-1}+\xi^2 e_{j-1}=0.
\end{array}
$$
 Thus 
$$
e(i)_{j+1}^{\frac{1}{2}}=\xi e(i)_{j-1}^{\frac{1}{2}}\mbox{ and }e(i)_j^2=e(i)_{j-1}e(i)_{j+1}
$$
for all $i$. Since this holds for all $j$, we have that equality holds in (\ref{eqT3*}) for all $i$ and $j$. Further, we have that for a particular $i$, either 
\begin{equation}\label{eqG1}
e(i)_j=0\mbox{ for all }j
\end{equation}
or
\begin{equation}\label{eqG2}
e(i)_j>0\mbox{ for all }j.
\end{equation}
If (\ref{eqG1}) holds for  a particular $i$, then $e(i)_0=e(i)_d=0$ so we have the degenerate case $D(i)_1=D(i)_2=0$ by Proposition \ref{PropPos}, so that
\begin{equation}\label{eqG3}
J(mD_1(i))=J(nD_2(i))=S_{m_i}\mbox{ for all }m,n \in \NN.
\end{equation}
Suppose that (\ref{eqG2}) holds for a particular $i$. Then by
(\ref{eqT7*}),  the Minkowski equality holds between $\mathcal I(D(i)_1)$ and $\mathcal I(D(i)_2)$ for this $i$. Thus there exists $\lambda_i\in \RR_{>0}$ such that 
$$
\frac{e(i)_{j+1}}{e(i)_j}=\lambda_i
$$
 for all $j$. Thus
 $$
 \xi^2=\frac{e(i)_{j+1}}{e(i)_{j-1}}=\frac{e(i)_{j+1}}{e(i)_j}\frac{e(i)_j}{e(i)_{j-1}}
 =\lambda_i^2
 $$
 so that $\lambda_i=\xi$  and so
 $$
 \frac{e(i)_d^{\frac{1}{d}}}{e(i)_0^{\frac{1}{d}}}=\xi.
 $$
 Since $D_1,D_2\ne 0$, (\ref{eqG2}) holds for some $i$, so that  $\xi\in \QQ_{>0}$ by Theorem \ref{Theorem8}. Write $\xi=\frac{a}{b}$ with $a,b\in \ZZ_{>0}$. We have that $J(maD_1(i))=J(mbD_2(i))$ for all $i$ such that (\ref{eqG2}) holds and $m\in \NN$ by Theorem \ref{Theorem8}. Thus $J(maD_1)=J(mbD_2)$ for all $m\in \NN$ by formula (\ref{eqR6}) and thus
 $I(maD_1)=I(mbD_2)$ for all $m\in \NN$ since $I(maD_1)=J(maD_1)\cap R$ and $I(mbD_2)=J(mbD_2)\cap R$ for all $m$.
 
 The converse follows from Theorem \ref{MinNew}, since $\overline{R[\mathcal I(D_j)]}=R[\mathcal I(D_j)]$ for $j=1,2$.

\end{proof}

Theorem \ref{Theorem10} is proven in dimension $d=2$ in 
 \cite[Theorem 5.9]{C6} using the theory of relative Zariski decomposition, which requires dimension two. This theory is also used to prove the  fact that the mixed multiplicities $e_i$ of integral divisorial filtrations are rational numbers in dimension two. This fact is used in the proof of \cite[Theorem 5.9]{C6}. The mixed multiplicities of integral divisorial filtrations can be irrational numbers in dimension $\ge 3$, as is shown in the example of Section \ref{SecExample}.

The following corollary is proven in the case that $d=2$ in \cite[Corollary 5.10]{C6}.

\begin{Corollary}\label{Cor1.10} Suppose that $R$ is an excellent local domain and $\nu_1$ and $\nu_2$ are $m_R$-valuations such that 
Minkowski's equality holds between the $m_R$-filtrations $\mathcal I(\nu_1)=\{I(\nu_1)_m\}$ and $\mathcal I(\nu_2)=\{I(\nu_2)_m\}$. Then $\nu_1=\nu_2$.
\end{Corollary}

 \begin{proof}
We have by Theorem \ref{Theorem10} that $I(\nu_1)_{an}=I(\nu_2)_{bn}$ for all $n$ and some positive integers $a$ and $b$ which we can take to be relatively prime.

Suppose that $0\ne f\in I(\nu_1)_n$. Then $f^a\in I(\nu_1)_{an}=I(\nu_2)_{bn}$ so that $a\nu_2(f)\ge bn$. If $f^a\in I(\nu_2)_{bn+1}$ then $f^{ab}\in I(\nu_2)_{b(bn+1)}=I(\nu_1)_{a(bn+1)}$ so that $\nu_1(f)>n$. Thus
\begin{equation}\label{eqX21}
\nu_1(f)=n\mbox{ if and only if }\nu_2(f)=\frac{b}{a}n.
\end{equation}
Further, (\ref{eqX21}) holds for every nonzero $f\in {\rm QF}(R)$ since $f$ is a quotient of nonzero elements of $R$.

Now the maps 
 $\nu_1:{\rm QF}(R)\setminus \{0\}\rightarrow \ZZ$ and 
$\nu_2:{\rm QF}(R)\setminus \{0\}\rightarrow \ZZ$ are surjective, so there exists $0\ne f\in {\rm QF}(R)$ such that $\nu_1(f)=1$  and there exists $0\ne g\in {\rm QF}(R)$ such that $\nu_2(g)=1$ which implies that $a=b=1$ since $a,b$ are relatively prime. Thus $\nu_1=\nu_2$.

\end{proof}

\begin{Remark}
With the assumptions of the above corollary and further assuming  that $R$ is normal,   the functions $w_{\mathcal I(\nu_i)}$ of (\ref{eqT9}) are  $w_{\mathcal I(\nu_i)}=\nu_i$. Thus 
$$
w_{\mathcal I(\nu_i)}(f^n)=\nu_i(f^n)=n\nu_i(f)=nw_{\mathcal I(\nu_i)}(f)
$$
 for all nonzero $f\in R$ and $i=1,2$. Thus the proof of Theorem \ref{Theorem8} shows that
$$
\xi=\frac{w_{\mathcal I(\nu_1)}(f)}{w_{\mathcal I(\nu_2)}(f)}=\frac{\nu_1(f)}{\nu_2(f)}
$$
for all nonzero $f\in m_R$. 
\end{Remark}

%\begin{proof} Let $X\rightarrow \mbox{Spec}(R)$ be such that the prime exceptional divisors on $X$ are $E_1,\ldots,E_r$ and  such that the valuation rings of the $\nu_i$ are $\mathcal O_{X,E_i}$ for $i=1,2$. Thus $I(\nu_i)_m=I(mE_i)$ for $i=1,2$ and $m\in \NN$. 
%Minkowski's inequality holds, and $e_R(\mathcal I(\nu_i))>0$ for all $i$ by \cite[Proposition 2.1]{C6}. Thus by Corollary \ref{Cor1.1} we have that  there exists $\xi\in \RR_{>0}$ such that $\gamma_{E_i}(E_2)=\xi \gamma_{E_i}(E_1)$ for all $i$.

%The functions $w_{\mathcal I(\nu_i)}$ of (\ref{eqT9}) are  $w_{\mathcal I(\nu_i)}=\nu_i$. Thus 
%$$
%w_{\mathcal I(\nu_i)}(f^n)=\nu_i(f^n)=n\nu_i(f)=nw_{\mathcal I(\nu_i)}(f)
%$$
 %for all nonzero $f\in R$ and $i=1,2$. Thus the proof of Theorem \ref{Theorem8} shows that
%$$
%\xi=\frac{w_{\mathcal I(\nu_1)}(f)}{w_{\mathcal I(\nu_2)}(f)}=\frac{\nu_1(f)}{\nu_2(f)}
%$$
%for all nonzero $f\in m_R$. 

%Thus $\nu_1(f)=\xi\nu_2(f)$ for all $f\in m_R$. If $f\in R\setminus m_R$, we have that $\nu_1(f)=\nu_2(f)=0$. Thus $\nu_1(f)=\xi\nu_2(f)$ for all nonzero $f$ in the quotient field $K$ of $R$. There exists $f\in K$ such that $\nu_2(f)=1$, so that $\xi\in \ZZ_{>0}$. Further there exists $g\in K$ such that $\nu_1(g)=1$. Thus $\frac{1}{\xi}\in \ZZ_{>0}$. Thus $\xi=1$ and so $I(\nu_1)_m=I(\nu_2)_m$ for all $m\in \ZZ_{>0}$, and so $\nu_1=\nu_2$.
%\end{proof}

\section{Bounded $m_R$-Filtrations}\label{SecBound}

Bounded $m_R$-filtrations are defined in Subsection \ref{SubSecBound}.

\begin{Theorem}\label{RBT} Suppose that $R$ is an excellent local domain, $\mathcal I(1)$ is a real bounded $m_R$-filtration and $\mathcal I(2)$ is an arbitrary $m_R$-filtration such that $\mathcal I(1)\subset \mathcal I(2)$. Then the following are equivalent
\begin{enumerate}
\item[1)]  $e(\mathcal I(1))=e(\mathcal I(2))$.
\item[2)]  There is equality of  integral closures  
$$
\overline{\sum_{m\ge 0}I(1)_mt^m}=\overline{\sum_{m\ge 0}I(2)_mt^m}
$$
  in $R[t]$.
\end{enumerate}
 \end{Theorem}
 
 \begin{proof} 2) implies 1) follows from \cite[Theorem 6.9]{CSS} or \cite[Appendix]{C6} as summarized in Subsection \ref{SubSecEqChar}.
 
We now prove 1) implies 2). Let $\mathcal I(D_1)$  be  the real divisorial $m_R$-filtrations such that $\overline{R(\mathcal I(1))}=R(\mathcal I(D_1))$.   Thus $R(\mathcal I(D_1))\subset \overline{R(\mathcal I(2))}=R[\overline{\mathcal I(2)}]$ so that $\mathcal I(D_1)\subset \overline{\mathcal I(2)}$. We have that $e(\mathcal I(1))=e(\mathcal I(D_1))$ and $e(\mathcal I(2))=e(\overline{\mathcal I(2)})$ by \cite[Theorem 6.9]{CSS} or \cite[Appendix]{C6}.  Thus $e(\mathcal I(D_1))=e(\overline{\mathcal I(2)})$ and so $R(\mathcal I(D_1))=R(\overline{\mathcal I(2)})$ by Theorem \ref{CorDiv2}. Thus 2) holds for $\mathcal I(1)$ and $\mathcal I(2)$.

%We now prove 1) implies 2). Let $\mathcal I(D_1)$ and $\mathcal I(D_2)$ be  integral divisorial $m_R$-filtrations such that $\overline{R(\mathcal I(1))}=R(\mathcal I(D_1))$ and $\overline{R(\mathcal I(2))}=R(\mathcal I(D_2))$.   Thus $R(\mathcal I(D_1))\subset R(\mathcal I(D_2))$ so that $\mathcal I(D_1)\subset \mathcal I(D_2)$. We have that $e(\mathcal I(i))=e(\mathcal I(D_i)$ for $i=1,2$ by \cite[Theorem 6.9]{CSS} or \cite[Appendix]{C6}.  Thus $e(\mathcal I(D_1))=e(\mathcal I(D_2))$ and so $R(\mathcal I(D_1))=R(\mathcal I(D_2))$ by Theorem \ref{CorDiv2}. Thus 2) holds for $\mathcal I(1)$ and $\mathcal I(2)$.
   \end{proof}

\begin{Theorem}\label{TRSKT} Suppose that $R$ is a $d$-dimensional excellent local domain and $\mathcal I(1)$ and $\mathcal I(2)$ are bounded $m_R$-filtrations. Then the following are equivalent
\begin{enumerate}
\item[1)]  The Minkowski inequality
$$
e(\mathcal I(1)\mathcal I(2))^{\frac{1}{d}}=e(\mathcal I(1))^{\frac{1}{d}}+e(\mathcal I(2))^{\frac{1}{d}}
$$
holds.
\item[2)] There exist positive integers $a,b$ such that there is equality of  integral closures
$$
\overline{\sum_{n \ge 0}I(1)_{an}t^n}=\overline{\sum_{n\ge 0}I(2)_{bn}t^n}
$$
 in $R[t]$.
\end{enumerate}
 \end{Theorem}
 
 \begin{proof} Let $\mathcal I(D_1)$ and $\mathcal I(D_2)$ be  integral divisorial $m_R$-filtrations such that $\overline{R(\mathcal I(1))}=R(\mathcal I(D_1))$ and $\overline{R(\mathcal I(2))}=R(\mathcal I(D_2))$. By Proposition \ref{BoundProp}, we have equality of functions 
 $$
 \lim_{m\rightarrow \infty} \frac{\ell(R/I(i)_{mn_1}I(i)_{mn_2})}{m^d}
 =\lim_{m\rightarrow \infty} \frac{\ell(R/I(D_i)_{mn_1}I(D_i)_{mn_2})}{m^d}
  $$
  for $i=1,2$ and all $n_1,n_2\in \NN$. Since 1) and 2) are equivalent for the integral divisorial filtrations $\mathcal I(D_1)$ and $\mathcal I(D_2)$ by Theorem \ref{Theorem10}, they are also equivalent for the bounded $m_R$-filtrations  $\mathcal I(1)$ and $\mathcal I(2)$.
 \end{proof}
 
 \section{Analytically Irreducible local Rings} Let $R$ be an analytically irreducible local domain. A local ring $R$ is analytically irreducible if   the $m_R$-adic completion $\hat R$ is a domain.  The complete local ring  $\hat R$ is then an excellent local domain.
 
 \begin{Lemma}\label{AI1}(\cite[Proposition 9.3.5]{HS}) Let $R$ be an analytically irreducible local domain. Then there is a 1-1 correspondence between $m_R$-valuations of $R$ and $m_{\hat R}$-valuations of $\hat R$.
 \end{Lemma}
 
 \begin{Lemma}\label{AI2} Let $R$ be a analytically irreducible  local domain. Let $\mu_1,\ldots,\mu_s$ be $m_R$-valuations and $n_1,\ldots,n_s\in \ZZ_s$. Let $\hat \mu_i$ be the unique extension of $\mu_i$ to a $m_{\hat R}$-valuation for $1\le i\le s$. Then 
 $$
 I(\mu_1)_{n_1}\cap \cdots \cap I(\mu_s)_{n_s}\hat R=I(\hat\mu_1)_{n_1}\cap \cdots \cap I(\hat\mu_s)_{n_s}
 $$
 and
 $$
 (I(\hat\mu_1)_{n_1}\cap \cdots \cap I(\hat\mu_s)_{n_s})\cap R=I(\mu_1)_{n_1}\cap \cdots \cap I(\mu_s)_{n_s}.
 $$
 \end{Lemma}  
 
 \begin{proof} We certainly have that $ I(\mu_1)_{n_1}\cap \cdots \cap I(\mu_s)_{n_s}\hat R\subset I(\hat\mu_1)_{n_1}\cap \cdots \cap I(\hat\mu_s)_{n_s}$. Suppose that 
 $f\in I(\hat\mu_1)_{n_1}\cap\cdots\cap I(\hat\mu_s)_{n_s}$. There exists $a>0$ such that 
 $m_{\hat R}^a\subset  I(\hat\mu_1)_{n_1}\cap\cdots\cap I(\hat\mu_s)_{n_s}$ and $a\hat\mu_i(m_{\hat R})>n_i$ for all $i$. Since $\hat R/m_{\hat R}^a\cong R/m_R^a$,  there exists $g\in R$ and $h\in m_{\hat R}^a$ such that $f=g+h$. For all $i$, we have
 $$
 \mu_i(g)=\hat\mu_i(f-h)\ge \min\{\hat\mu_i(f),\hat\mu_i(h)\}\ge n_i.
 $$
 Thus $g\in I(\mu_1)_{n_1}\cap \cdots \cap I(\mu_s)_{n_s}$. Now $h=\sum a_jb_j$ with $a_j\in m_R^a$ and $b_j\in \hat R$.  We have that
 $$
 \nu_i(a_j) =\hat\nu_i(a_j)\ge a \hat\nu_i(m_{\hat R})>n_i
 $$
 for all $i$ and $j$ so that $a_j\in I(\mu_1)_{n_1}\cap \cdots \cap I(\mu_s)_{n_s}$ for all $j$.
 Thus $f\in  (I(\mu_1)_{n_1}\cap \cdots \cap I(\mu_s)_{n_s})\hat R$.
 
 Since $A\rightarrow \hat A$ is faithfully flat, we have that 
 $$
 (I(\hat\mu_1)_{n_1}\cap \cdots \cap I(\hat\mu_s)_{n_s})\cap R
 = (I(\mu_1)_{n_1}\cap \cdots \cap I(\mu_s)_{n_s}\hat R)\cap R =I(\mu_1)_{n_1}\cap \cdots \cap I(\mu_s)_{n_s}.
 $$ 
 \end{proof}

 By Lemma  \ref{BdLemma2},  if $D=a_1\mu_1+\cdots+a_s\mu_s$ where $\mu_1,\ldots,\mu_s$ are $m_R$-valuations and $a_1,\ldots, a_s\in \RR_{>0}$, then $R[\mathcal I(D)]$ is integrally closed in $R[t]$.
 Let $\hat D=a_1\hat\mu_1+\cdots+a_s\hat\mu_s$ and $\mathcal I(\hat D)$ be the induced $m_R$-filtration on $\hat R$.
 
\begin{Lemma} Suppose that $\mathcal I=\{I_m\}$ is a (real) bounded $m_R$-filtration; that is, there exists a (real) divisorial $m_R$-filtration $\mathcal I(D)$ such that the integral closure $\overline{R[\mathcal I]}$ of $R[\mathcal I]$ in $R[t]$ is $R[\mathcal I(D)]$. Let $\hat{\mathcal I}=\{I_m\hat R\}$. Then $\hat{\mathcal I}$ is a (real) bounded $m_{\hat R}$-filtration and the integral closure $\overline{\hat R[\hat{\mathcal I}]}$ of $\hat R[\hat{\mathcal I}]$ in $\hat R[t]$ is $\hat R[\mathcal I(\hat D)]$. 
 \end{Lemma}
 
 \begin{proof} $R[\mathcal I(D)]=\sum_{m\ge 0}I(mD)t^m$ is integral over $R[\mathcal I]=\sum_{n\ge 0}I_mt^m$ so the integrally closed ring $\hat R[\mathcal I(\hat D)]=\sum_{m\ge 0}I(mD)\hat Rt^m$ is integral over $\hat R[\hat{\mathcal I}]=\sum_{n\ge 0}I_m\hat Rt^m$.
  \end{proof}

\begin{Theorem}\label{RBTA} Suppose that $R$ is an analytically irreducible local ring and $\mathcal I(1)$ and $\mathcal I(2)$ are real bounded $m_R$-filtrations such that $\mathcal I(1)\subset \mathcal I(2)$. Then the following are equivalent
\begin{enumerate}
\item[1)]  $e(\mathcal I(1))=e(\mathcal I(2))$.
\item[2)]  There is equality of  integral closures  
$$
\overline{\sum_{m\ge 0}I(1)_mt^m}=\overline{\sum_{m\ge 0}I(2)_mt^m}
$$
  in $R[t]$.
\end{enumerate}
 \end{Theorem}

  \begin{proof} We have that $\ell_{\hat R}(\hat R/I(j)_m\hat R)=\ell_R(R/I(j)_m)$
  %and $\ell_{\hat R}(\hat R/I(mD)\hat R)=\ell_R(R/I(mD))$
  for $j=1,2$ and all $m\in \NN$. Thus $e_{R}(\mathcal I(j))=e_{\hat R}(\hat{\mathcal I}(j))$
 %and     $e_R(\mathcal I(D))=e_{\hat R}(\mathcal I(\hat D))$ 
  for $j=1,2$.
 We have that $\hat R[\mathcal I(\hat D_1)]=\hat R[\mathcal I(\hat D_2)]$ if and only if $R[\mathcal I(D_1)]=R[\mathcal I(D_2)]$ by Lemma \ref{AI2}. Theorem \ref{RBTA} thus follows from Theorem \ref{RBT}.
  \end{proof}

\begin{Theorem}\label{TRSKTA} Suppose that $R$ is a $d$-dimensional analytically irreducible local ring and $\mathcal I(1)$ and $\mathcal I(2)$ are bounded $m_R$-filtrations. Then the following are equivalent
\begin{enumerate}
\item[1)]  The Minkowski inequality
$$
e(\mathcal I(1)\mathcal I(2))^{\frac{1}{d}}=e(\mathcal I(1))^{\frac{1}{d}}+e(\mathcal I(2))^{\frac{1}{d}}
$$
holds.
\item[2)] There exist positive integers $a,b$ such that there is equality of  integral closures
$$
\overline{\sum_{n \ge 0}I(1)_{an}t^n}=\overline{\sum_{n\ge 0}I(2)_{bn}t^n}
$$
 in $R[t]$.
\end{enumerate}
 \end{Theorem}

 \begin{proof} 
  
  Since $\ell_{\hat R}(\hat R/I(1)_{mn_1}I(2)_{mn_2}\hat R)=\ell_R(R/I(1)_{mn_1}I(2)_{mn_2})$  for all $m,n_1,n_2\in \NN$, we have that
  $$
e_R(\mathcal I(1)\mathcal I(2))^{\frac{1}{d}}=e_R(\mathcal I(1))^{\frac{1}{d}}+e_R(\mathcal I(2))^{\frac{1}{d}}
$$  
if and only if
$$
e_{\hat R}(\hat{\mathcal I}(1)\hat{\mathcal I}(2))^{\frac{1}{d}}=e_{\hat R}(\hat{\mathcal I}(1))^{\frac{1}{d}}+e_{\hat R}(\hat{\mathcal I}(2))^{\frac{1}{d}}.
$$

  By Lemma \ref{AI2}, we have that
  $\sum_{n\ge 0}I(D_1)_{an}\hat Rt^n
  =\sum_{n\ge 0}I(D_2)_{bn}\hat Rt^n$ if and only if $\sum_{n\ge 0}I(D_1)_{an}t^n=\sum_{n\ge 0}I(D_2)_{bn}t^n$. 
  Since $\overline{ \sum_{n\ge 0}I(j)_{cn}t^n}=\sum_{n\ge 0}I(D_j)_{cn}t^n$ and
  $$
  \overline{ \sum_{n\ge 0}I(j)_{cn}\hat Rt^n}=\sum_{n\ge 0}I(D_j)_{cn}\hat Rt^n
  $$
  for all $c\in \ZZ_{>0}$ and $j=1,2$, we have that   
  $\overline{\sum_{n\ge 0}I(1)_{an}t^n}=\overline{\sum_{n\ge 0}I(2)_{bn}t^n}$ 
  if and only if $\overline{\sum_{n\ge 0}I(1)_{an}\hat Rt^n}=\overline{\sum_{n\ge 0}I(2)_{bn}\hat Rt^n}$.   
  
  By Theorem \ref{TRSKT} we have that the conclusions of Theorem \ref{TRSKTA} holds.
  \end{proof}

\section{An Example}\label{SecExample}

In Theorem 1.4 \cite{C7}, the following example is constructed.
Let $k$ be an algebraically closed field. A 3-dimensional normal algebraic local ring $R$ over $k$
is constructed,  and the blow up $\phi:X\rightarrow \mbox{Spec}(R)$  of an $m_R$-primary ideal such that $X$ is nonsingular with two  irreducible exceptional divisors $E_1$ and $E_2$ is constructed.   

 The resolution of singularities of a three dimensional normal local ring which we construct is similar to    the one constructed in  \cite[Example 6]{CS} which is used to give an example of a  divisorial filtration  with irrational multiplicity.

  \begin{Theorem}(\cite[Theorem 1.4]{C7})\label{Theorem4} Let $D=n_1E_1+n_2E_2$ with $n_1,n_2\in \NN$.
   Then 
    $$
   \lim_{m\rightarrow\infty}
  \frac{\ell_R(R/I(mD))}{m^3} 
  =\left\{\begin{array}{ll}
  33 n_1^3&\mbox{ if }n_2<n_1\\
  78n_1^3-81n_1^2n_2+27n_1n_2^2+9n_2^3&\mbox{ if }n_1\le n_2<n_1\left(3-\frac{\sqrt{3}}{3}\right)\\
  \left(\frac{2007}{169}-\frac{9\sqrt{3}}{338}\right)n_2^3&\mbox{ if }n_1\left(3-\frac{\sqrt{3}}{3}\right)<n_2.
  \end{array}\right.
  $$
    \end{Theorem}

 We compute the functions $\gamma_{E_1}$ and $\gamma_{E_2}$ in \cite[Theorem 4.1]{C7}.   
    
  \begin{Theorem}\label{Theorem5}(\cite[Theorem 4.1]{C7}) Let $D=n_1E_1+n_2E_2$ with $n_1,n_2\in \NN$, an effective exceptional  divisor on $X$.
  \begin{enumerate}
  \item[1)] Suppose that $n_2<n_1$. Then  $\gamma_{E_1}(D)=n_1$ and $\gamma_{E_2}(D)=n_1$.
  \item[2)]   Suppose that $n_1\le n_2<n_1 \left(3-\frac{\sqrt{3}}{3}\right)$. Then   $\gamma_{E_1}(D)=n_1$ and $\gamma_{E_2}(D)=n_2$.
  \item[3)] Suppose that $n_1 \left(3-\frac{\sqrt{3}}{3}\right)<n_2$. Then   $\gamma_{E_1}(D)=\frac{3}{9-\sqrt{3}}n_2$ and $\gamma_{E_2}(D)=n_2$.
  \end{enumerate}
  In all three cases, $-\gamma_{E_1}(D)E_1-\gamma_{E_2}(D)E_2$ is nef on $X$. 
    \end{Theorem}
    
    \begin{Corollary}\label{Cor20T} Suppose that $D_1$ and $D_2$ are
 effective integral exceptional divisors
 on $X$.   
 If $D_1$ and $D_2$ are in the first region of Theorem \ref{Theorem4}, then Minkowski's equality holds between them. If $D_1$ and $D_2$ are  in the second region, then Minkowski's equality holds between them if and only if $D_2$ is a rational multiple of $D_1$. If $D_1$ and $D_2$ are in the third region, then Minkowski's equality holds between them. 
  Minkowski's equality cannot hold between $D_1$ and $D_2$ in different regions.  
  \end{Corollary}
  
  \begin{proof} This follows from Theorems \ref{Theorem6} and \ref{Theorem5}.
  \end{proof}

  The interpretation of mixed multiplicities as anti-positive intersection multiplicities is particularly useful in the calculation of examples.  
  We quote some statements from \cite{C6} which, along with the calculations in Theorem \ref{Theorem5} and the identities
  
 \begin{equation}\label{eqT12}
 (E_1^3)=468, (E_1^2\cdot E_2)=-162, (E_1\cdot E_2^2)=54, (E_2^3)=54
 \end{equation}
 on page 15 of \cite{C7}
  allow us to compute the mixed multiplicities of any divisors $D_1=a_1E+a_2E_2$ and $D_2=b_1E_1+b_2E_2$.

 It is shown in \cite[Theorem 8.3]{C6} that  we have  identities
\begin{equation}\label{eq21T}
e_R(\mathcal I(D_1)^{[d_1]},\mathcal I(D_2)^{[d_2]};R)=-\langle(-D_1)^{d_1}\cdot(-D_2)^{d_2}\rangle
\end{equation}
where $\langle(-D_1)^{d_1}\cdot (-D_2)^{d_2}\rangle$ are the anti-positive intersection products defined in \cite{C6}.
In particular, $e_R(\mathcal I(D);R)=-\langle(-D)^d\rangle$.  Thus by (\ref{eqV6}), we have that  \cite[Formula (1.8)]{C7}
 
 \begin{equation}\label{eq10T}
 \begin{array}{l}
\lim_{m\rightarrow \infty}  \frac{\ell_R(R/I(mn_1D_1)I(mn_2D_2))}{m^d}\\
  =-\sum_{d_1+d_2=d}\frac{1}{d_1!d_2!}\langle(-D_1)^{d_1}\cdot (-D_2)^{d_2}\rangle
n_1^{d_1} n_2^{d_2}.
\end{array}
\end{equation} 
 and \cite[Formula (1.9)]{C7}
 \begin{equation}\label{eq11T}
 \lim_{m\rightarrow \infty}\frac{\ell_R(R/I(mD))}{m^d}=-\frac{\langle(-D)^d\rangle}{d!}.
 \end{equation}

 \begin{Proposition}\label{Prop1T}(\cite[Proposition 2.4]{C7}) Suppose that $D_1,\ldots,D_d
 $ are effective $\QQ$-Cartier divisors with exceptional support such that the divisors $-\sum\gamma_{E_i}(D_j)E_i$ are nef for $1\le j\le d$.
  Then the positive intersection product $\langle-D_1\cdot,\ldots,\cdot -D_d\rangle$ is the ordinary intersection product
  $(-\sum\gamma_{E_i}(D_1)E_i\cdot\ldots\cdot-\sum\gamma_{E_i}(D_d)E_i)$. 
  \end{Proposition}

 We now use this method to compute the mixed multiplicities of $\mathcal I(E_1)$ and $\mathcal I(E_2)$.
 By Theorem \ref{Theorem5}
 $$
 \gamma_{E_1}(E_1)=1\,\, \gamma_{E_2}(E_1)=1,\,\,
 \gamma_{E_1}(E_2)=\frac{3}{9-\sqrt{3}},\,\, \gamma_{E_2}(E_2)=1.
 $$
 By formulas (\ref{eqT12}) and (\ref{eq10T}) and Proposition \ref{Prop1T}, 
 \begin{equation}\label{eqT20}
 \begin{array}{l}
 \lim_{m\rightarrow\infty}\frac{\ell_R(R/I(mn_1E_1)I(mn_2E_2))}{m^3}\\
 =\sum_{i_1+i_2=3}\frac{1}{i_1!i_2!}e_R(\mathcal I(E_1)^{[i_1]},\mathcal I(E_2)^{[i_2]})n_1^{i_1}n_2^{i_2}\\
 =\sum_{i_1+i_2=3}-\frac{1}{i_1!i_2!}\left((-\gamma_{E_1}(E_1)E_1-\gamma_{E_2}(E_1)E_2)^{i_1}\cdot(-\gamma_{E_1}(E_2)E_1-\gamma_{E_2}(E_2)E_2)^{i_2})\right)n_1^{i_1}n_2^{i_2}\\ 
 =\sum_{i_1+i_2=3}\frac{1}{i_1!i_2!}\left((E_1+E_2)^{i_1}\cdot(\frac{3}{9-\sqrt{3}}E_1+E_2)^{i_2}\right)n_1^{i_1}n_2^{i_2}\\
 =33n_1^3+(\frac{891}{26}+\frac{99}{26}\sqrt{3})n_1^2n_2+(\frac{12042}{338}-\frac{27}{338}\sqrt{3})n_1n_2^2+\left(\frac{2007}{169}-\frac{9\sqrt{3}}{338}\right)
n_2^3,
 \end{array} \end{equation}
  in contrast to the function of Theorem \ref{Theorem4}.

  We make a more detailed analysis in the third region. 
     
 \begin{Example} Suppose that $D_1=a_1E_1+a_2E_2$ and $D_2=b_1E_1+b_2E_2$ are integral divisors in the third region of Theorem \ref{Theorem4}, $n_1(3-\frac{\sqrt{3}}{3})<n_2$. Then $\mathcal I(D_1)$ and $\mathcal I(D_2)$ satisfy equality in Minkowski's inequality. We have 
 $$
 e_i=e_R(\mathcal I(D_1)^{[3-i]},\mathcal I(D_2)^{[i]})=a_2^{3-i}b_2^i\left(\frac{12042}{169}-\frac{27\sqrt{3}}{169}\right)
 $$
 for $0\le i\le 3$ and $\frac{e_i}{e_{i-1}}=\frac{b_2}{a_2}$ for $0\le i\le 3$. Thus
 $$
I(mb_2D_1)=I(ma_2D_2)
$$
for all $m\in \NN$.
 \end{Example}
 
 \begin{proof} By Theorem \ref{Theorem5}
 $$
 \gamma_{E_1}(D_1)=\frac{3}{9-\sqrt{3}}a_2,\,\, \gamma_{E_2}(D_1)=a_2,\,\,
 \gamma_{E_1}(D_2)=\frac{3}{9-\sqrt{3}}b_2,\,\, \gamma_{E_2}(D_2)=b_2.
 $$
 By formula (\ref{eq10T}) and Proposition \ref{Prop1T}, 
 $$
 \begin{array}{l}
 \lim_{m\rightarrow\infty}\frac{\ell_R(R/I(mD_1)I(mD_2))}{m^3}\\
 =\sum_{i_1+i_2=3}\frac{1}{i_1!i_2!}e_R(\mathcal I(D_1)^{[i_1]},\mathcal I(D_2)^{[i_2]})n_1^{i_1}n_2^{i_2}\\
 =\sum_{i_1+i_2=3}-\frac{1}{i_1!i_2!}\left((-\gamma_{E_1}(D_1)E_1-\gamma_{E_2}(D_1)E_2)^{i_1}\cdot(-\gamma_{E_1}(D_2)E_1-\gamma_{E_2}(D_2)E_2)^{i_2})\right)n_1^{i_1}n_2^{i_2}\\ 
 =\sum_{i_1+i_2=3}-\frac{1}{i_1!i_2!}\left( \left(-\frac{3}{9-\sqrt{3}}a_2E_1-a_2E_2\right)^{i_1}\cdot
 \left(-\frac{3}{9-\sqrt{3}}b_2E_1-b_2E_2\right)^{i_2}\right)n_1^{i_1}n_2^{i_2}\\
 =-\left(-\frac{3}{9-\sqrt{3}}E_1-E_2\right)^3\left[\sum_{i_1+i_2=3}\frac{1}{i_1!i_2!}a_2^{i_1}b_2^{i_2}n_1^{i_1}n_2^{i_2}\right]\\
 =\left(\frac{12042}{169}-\frac{27\sqrt{3}}{169}\right)\left[\sum_{i_1+i_2=3}\frac{1}{i_1!i_2!}a_2^{i_1}b_2^{i_2}n_1^{i_1}n_2^{i_2}\right].
   \end{array}
 $$
  We obtain the formulas for the $e_i$ of the statement of the theorem from which we conclude that the Minkowski equality is satisfied. The identity $I(mb_2D_1)=I(ma_2D_2)$ now follows from Corollary \ref{Cor1.2} and Corollary \ref{Cor1.1}.
 \end{proof}
   %In particular, 
  %$$
  %e_R(\mathcal I(D);R)
  %=\left\{\begin{array}{ll}
  %198 n^3&\mbox{ if }j<n\\
  %468n^3-486n^2j+162nj^2+54j^3&\mbox{ if }n\le j<n\left(3-\frac{\sqrt{3}}{3}\right)\\
 % \left(\frac{12042}{169}-\frac{27\sqrt{3}}{169}\right)j^3&\mbox{ if }n\left(3-\frac{\sqrt{3}}{3}\right)<j.
  %\end{array}\right.
  %$$

\end{document}